\documentclass[11pt,twoside]{scrartcl}

\usepackage{dnmath}

\usepackage{amscd}
\usepackage{latexsym}
\usepackage[all]{xy}
\usepackage{verbatim}
\usepackage{multicol} 
\usepackage{tikz} 
\usetikzlibrary{cd} 

\tikzset{
  paired poset/.style={
    x=9mm, y=10mm,
    every node/.style={font=\normalsize, inner sep=2.5pt, outer sep=0pt}
  },
  poset cover/.style={
    draw=black, line width=0.4pt, shorten <=1pt, shorten >=1pt
  }
}

\usepackage{Grothendieck}

\newcommand{\FUN}{\mathcal{F}\mkern-3mu\mathit{un}}
\renewcommand{\Fun}{\operatorname{Fun}}
\renewcommand{\End}{\operatorname{End}}

\theoremstyle{dnplain}

\theoremstyle{dndefinition}

\theoremstyle{dnremark}

\title{Grothendieck Rings of Module Categories: Localizations, Strata, and Radicals}
\author{Riley Chinburg \and Dmitri Nikshych}
\date{}
\dnshorttitle{Grothendieck Rings of Module Categories}
\dnshortauthor{R. Chinburg and D. Nikshych}
\dnaffiliation{Department of Mathematics and Statistics,\\
University of New Hampshire, Durham, NH 03824, USA}
\dnemail{riley.chinburg@unh.edu, dmitri.nikshych@unh.edu}
\begin{document}
\raggedbottom
\begin{abstract}
For a non-degenerate braided fusion category $\B$, we study the
Grothendieck $\mathbb Z_+$-ring of $\B$-module categories.
We define a convolution of étale algebras in the Drinfeld center \(\mathcal Z(\mathcal B)\) and show that the full center of a relative tensor product of module categories is canonically isomorphic to the convolution of their full centers.
The module Grothendieck ring has a canonical upper-triangular grading by the poset of localizations of $\B$. 
We describe each localization stratum by double-coset endpoint blocks and a group-ring-valued transition matrix, obtaining an explicit sandwich multiplication formula.
When \(\mathcal B\) is the center of a fusion category, the transition matrix of the maximal stratum is the 
Gram matrix of the Lagrangian algebras, so linear relations among their classes give rise to radicals.
We compute stratum rings and the radicals of their complexifications in examples from classical finite geometry, Drinfeld doubles of finite groups, exceptional WZW categories, and the Haagerup center. The examples exhibit noncommutative module Grothendieck rings with semisimple complexifications, as well as nonzero radicals in both intermediate and maximal complexified strata.
\end{abstract}

\maketitle
\tableofcontents


\section{Introduction}

Let $\B$ be a non-degenerate braided fusion category.
Throughout, all module categories are finite semisimple.
The relative tensor
product \cite[Subsection~3.4]{ENO2} makes the free abelian group on equivalence classes of indecomposable
$\B$-module categories into a $\mathbb Z_+$-ring
\[
K^\mmod(\B):=K_0(\ModCat(\B)),
\qquad
[\M][\N]=[\M\bt_\B\N].
\]
Thus, while the usual Grothendieck ring $K_0(\B)$ records tensor products of
objects of $\B$, the module Grothendieck ring records how entire module
categories compose. A classification of indecomposable module categories
determines its distinguished basis.

This multiplication contains information which is invisible in a mere
classification. In particular, although $K^\mmod(\B)$ has nonnegative integral
structure constants, it need not behave like a fusion ring. The complexified
Grothendieck ring of a fusion category is semisimple, whereas
$K^\mmod(\B)\ot_{\mathbb Z}\mathbb C$ need not be. One of our aims is to
explain this failure of semisimplicity in categorical terms and, more
generally, to describe the structure imposed on $K^\mmod(\B)$ by the
localization theory of $\B$.

A first structural description of this multiplication uses \(\alpha\)-induction
\(\alpha_\M^\pm:\B\to\FUN_\B(\M)\) for any \(\B\)-module category \(\M\)
\cite{BEK}. The associated endofunctor, which we denote by \(Z_\M\), appeared in
\cite[Proposition~3.6]{FFRS}. For a \(\B\)-module category \(\M\), we write
\[
Z_\M=(\alpha_\M^+)^{\mathrm R}\alpha_\M^-:\B\to\B,
\]
where \({\mathrm R}\) denotes the right adjoint.

We equip \(Z_\M\) with a braided lax monoidal structure and prove that
the relative tensor product of \(\B\)-module categories corresponds to composition, so that
\[
Z_{\M\bt_\B\N}\simeq Z_\M\circ Z_\N,
\]
as braided lax monoidal endofunctors. We also define a convolution of
algebras in \(\B\bt\B^{\rev}\) and show that it preserves étale algebras.
When \(\B\) is non-degenerate, the full center \(L_\M\) and the endofunctor
\(Z_\M\) encode the same invariant, and under this correspondence composition
of endofunctors matches convolution of full centers. In particular, the full
center of \(\M\bt_\B\N\) is canonically isomorphic to the convolution of
\(L_\M\) and \(L_\N\); see  Section~\ref{sect: Module categories, Lagrangian algebras, and residual data}.

The description of Lagrangian algebras by residual data \cite{DNO} reveals
the ideal structure of $K^\mmod(\B)$.  It is complementary to  the above categorical models for the relative tensor product: it organizes that product by localization.

The grading and stratification of the module Grothendieck ring are studied in Section~\ref{sect:The Grothendieck ring of module categories}. 
Let \(\mathfrak{Et}(\B)\) be the finite poset of connected étale algebras in \(\B\). For \(A\in\mathfrak{Et}(\B)\), the braided fusion category \(\B_A^{\loc}\) will be called the corresponding localization of \(\B\). Let \(\mathfrak{Loc}(\B)\) denote the poset of braided equivalence classes of such localizations, ordered by further localization: \(\C\leq\D\) if \(\D\) is braided equivalent to a localization of \(\C\). Then
\[
\mathfrak{Et}(\B)\longrightarrow\mathfrak{Loc}(\B),
\qquad
A\longmapsto\B_A^{\loc},
\]
is order-preserving.
An indecomposable \(\B\)-module category is described by connected étale algebras \(A_1,A_2\) with braided equivalent localizations, together with a residual equivalence between these localizations. If residual data
$
(A_1,A_2,\phi),\,(A_1',A_2',\phi')
$
are multiplied, then every summand \((B,C,\psi)\) of their product satisfies
$
A_1\subseteq B,\, A_2'\subseteq C.
$
Thus multiplication can only move upward in the localization order (Proposition~\ref{prop: etale monotonicity}).

Consequently, the decomposition
\[
K^\mmod(\B)=
\bigoplus_{\C\in\mathfrak{Loc}(\B)}K^\mmod(\B)_\C
\]
is upper triangular with respect to the localization order. For every
$\C\in\mathfrak{Loc}(\B)$ this gives two-sided ideals
\[
F_{\geq\C}:=
\bigoplus_{\D\geq\C}K^\mmod(\B)_\D,
\qquad
F_{>\C}:=
\bigoplus_{\D>\C}K^\mmod(\B)_\D,
\]
and the corresponding stratum ring
$
S_\C(\B):=F_{\geq\C}/F_{>\C}.
$
Multiplication in $S_\C(\B)$ retains exactly the terms which preserve the
localization $\C$ and discards those which localize further. The
semisimplicity criterion of Proposition~\ref{prop: semisimplicity poset grading}
shows that the algebra $K^\mmod(\B)\ot_{\mathbb Z}\mathbb C$ is semisimple if and only if
all of its complexified stratum algebras are semisimple.

The multiplication can be described uniformly within each localization stratum. Fix a localization \(\C\). The connected étale algebras \(A\) with \(\B_A^{\loc}\simeq\C\) form the possible endpoints of residual data in this stratum, while the residual equivalences are controlled by the braided autoequivalence group \(G=\Aut^{\mathrm{br}}(\C)\). Thus each pair of endpoints contributes a Hecke-type double coset block.

The remaining information needed for multiplication concerns the two endpoints that meet in the middle of a product. For every pair \(A_k,A_\ell\) in the localization fiber, the tensor product
$
\B_{A_k}\bt_\B\B_{A_\ell}
$
decomposes into $\B$-module transition components. Those which remain invertible after localization contribute to the same stratum; the others move to strictly higher localizations. After choosing identifications of the local categories with \(\C\), the surviving components assemble into an element
\(\Theta_{k\ell}\in\mathbb Z_{\geq0}[G]\), and the family
\(\Theta=(\Theta_{k\ell})\) is the transition matrix of the stratum. In the resulting double coset coordinates the multiplication takes the sandwich form
\[
x_{ik}(g)x_{\ell j}(h)=x_{ij}(h\Theta_{k\ell}g).
\]
Thus the structure of a localization stratum is determined by its endpoint double coset blocks together with the transition matrix between them (Theorem~\ref{thm: stratum multiplication}).

At the maximal localization the source of nonsemisimplicity becomes
particularly transparent. 
Suppose \(\mathcal B=\mathcal Z(\mathcal A)\), so that the maximal localization is $\Vect$ and its endpoints are the Lagrangian algebras \(L_1,\dots,L_m\).
By Corollary~\ref{cor: maximal Gram sandwich}, the maximal stratum has
basis $x_{ij}=x_{L_i,L_j}$ and multiplication
\[
x_{ij}x_{k\ell}=H_{jk}x_{i\ell},
\qquad
H_{jk}=\dim_k\Hom_\B(L_j,L_k).
\]
Hence, after identifying $x_{ij}$ with matrix units, the multiplication is
the sandwich product $X*Y=XHY$. The matrix $H$ is the Gram matrix of the
classes of the Lagrangian algebras as objects of $\B$. In particular,
\[
\ker H=
\left\{
(\lambda_i)\ \middle|\
\sum_i\lambda_i[L_i]=0
\text{ in }K_0(\B)\ot_{\mathbb Z}\mathbb C
\right\}.
\]
Thus singularity of $H$, and hence the radical of the complexified maximal stratum, is explained
by genuine linear relations among the underlying Lagrangian algebra objects.
The resulting nilpotent elements are virtual linear combinations of module
categories.

We illustrate these structures in several families. For centers of pointed
categories attached to elementary abelian groups of odd order, the
localization fibers are polar Grassmannians. Their transition matrices
are described by the associated orthogonal geometry
(Section~\ref{sect: Example classical}). In Witt rank at least two, the
ruling relations span all complex linear relations among the underlying
Lagrangian algebra classes and give generators for the radical of the
complexified maximal stratum. For $\Rep(D(G))$, following Davydov's
classification \cite{D}, we reformulate the parameterization of connected
\'etale algebras in terms of twisted sections and relative cohomology,
and identify their localizations with twisted Drinfeld doubles of
quotient sections. The case $G=S_3$ gives a group-theoretical example in
which both an intermediate stratum and the maximal stratum have nonzero
radicals (Section~\ref{sect: Example Drinfeld}). 

We also compute two exceptional WZW examples using the
stratum multiplication formula. 
For \(\mathcal C(\mathfrak{sl}_4,4)\), the stratum formulas determine products that the corresponding modular invariants do not distinguish, 
while the two-endpoint Ising fiber of
$\mathcal C(\mathfrak{sl}_6,6)$  produces  a noncommutative module Grothendieck
ring whose complexification is nevertheless semisimple (Section~\ref{sect: Example WZW}). Finally, the
center of the Haagerup fusion category gives an exotic,
non-group-theoretical example in which both an intermediate localization
stratum and the maximal stratum have nonzero radicals (Section~\ref{sect: Haagerup center}).

\noindent\textbf{Acknowledgments.}
The authors are grateful to Alexei Davydov for many illuminating discussions
and for sharing the unpublished work~\cite{CD}, and to
Cain Edie-Michell, Christoph Schweigert, and Kenichi Shimizu for useful
comments. D.N. would like to thank the Isaac Newton Institute for Mathematical Sciences, Cambridge, for support and hospitality during the programme 
\emph{Quantum Field Theory with Boundaries, Impurities and Defects}, where part of this work was undertaken.

The authors also acknowledge conversations with ChatGPT,
developed by OpenAI, which were useful in formulating several of the
questions, exploring possible approaches, and assisting with typesetting
and manuscript preparation. ChatGPT was used as an exploratory and
editorial aid; the mathematical content, verification of the arguments,
and all final decisions concerning the mathematics and its presentation
are the authors' own.

The authors' work was supported by the National Science Foundation under
Grant No.\ DMS-2302267.


\section{Module categories, Lagrangian algebras, and residual data}
\label{sect: Module categories, Lagrangian algebras, and residual data}

Throughout, $k$ is an algebraically closed field of characteristic zero.
We refer to \cite{EGNO} for basics of the theory of (braided) fusion
categories and module categories over them.

Let $\B$ be a braided fusion category. A left $\B$-module category
will also be regarded as a right $\B$-module category by setting
$m\triangleleft X:=X\triangleright m$. Its right module associativity
constraint is the composite
\begin{equation}
\label{eqn:right module convention}
(m\triangleleft X)\triangleleft Y
=Y\triangleright(X\triangleright m)
\xrightarrow{\sim}(Y\ot X)\triangleright m
\xrightarrow{c_{X,Y}^{-1}\triangleright\id_m}(X\ot Y)\triangleright m
=m\triangleleft(X\ot Y).
\end{equation}
The mixed constraint is
\[
(U\triangleright m)\triangleleft X
\cong(X\ot U)\triangleright m
\xrightarrow{c_{U,X}^{-1}\triangleright\id_m}
U\triangleright(m\triangleleft X),
\]
where module associators are suppressed. We use these constraints in the
relative tensor product $\M\bt_\B\N$. This right action is determined by
$\alpha^-_\M$ with the tensor convention below.

The \emph{full center} construction was introduced for algebras in modular tensor categories in \cite[Definition~4.9]{FjFRS} and further developed in \cite{KR,DavCenter}. In the present setting, it identifies indecomposable $\B$-module categories with Lagrangian algebras in the Drinfeld center $\Z(\B)$ \cite[Proposition~4.8]{DMNO}.

Algebra objects in fusion categories are associative and unital. An algebra $A$ in a braided
fusion category $\B$ is \emph{\'etale} if it is commutative and
separable, where separability means that multiplication splits
as an $A$-bimodule morphism. It is \emph{connected} if
$\dim_k\Hom_\B(\be,A)=1$.

We write \(\mathcal B_A\) for the category of left \(A\)-modules. A left \(A\)-module with action \(\mu:A\otimes M\to M\) is called \emph{local} if
\[
\mu\circ c_{M,A}\circ c_{A,M}=\mu.
\]
We write \(\mathcal B_A^{\mathrm{loc}}\) for the full subcategory of local modules, called the \emph{localization} of \(\mathcal B\) at \(A\).
For connected \'etale $A$, the category $\B_A^{\loc}$ is a braided fusion
category. If $\B$ is non-degenerate, so is $\B_A^{\loc}$
\cite[Corollary~3.30]{DMNO}, and \cite[Corollary~3.32]{DMNO} gives
\begin{equation}
\label{eqn:local-module-dimension}
\FPdim(\B_A^{\loc})
=\frac{\FPdim(\B)}{\FPdim(A)^2}.
\end{equation}
A connected \'etale algebra $A$ in a non-degenerate braided fusion
category $\B$ is \emph{Lagrangian} if
$\FPdim(A)^2=\FPdim(\B)$, equivalently if
$\B_A^{\loc}\simeq\Vect$.

\subsection{The braided lax endofunctor associated to a module category}
\label{subsect: modular invariant endofunctor}

Modular invariant matrices originated in rational conformal field theory as the coefficient matrices of torus partition functions; in the modular setting they have nonnegative integral entries and commute with the modular matrices $S$ and $T$. The ADE classification of Cappelli--Itzykson--Zuber \cite{CIZ} and the later $\alpha$-induction approach of B{\"o}ckenhauer--Evans--Kawahigashi \cite{BEK} are among the basic developments. Categorical formulations in terms of module categories, algebra objects, and full centers were developed subsequently by Ostrik, Fuchs--Runkel--Schweigert, Davydov, and others \cite{Ost03,FRS,DavCenter,DKR}. We use the same $\alpha$-induction data, but retain it before decategorification as a braided lax monoidal endofunctor of $\B$. This construction only requires the braided fusion category setting; the traditional statement involving $S$ and $T$ applies to  modular categories.
Let $\B$ be a braided fusion category, and let $\M$ be a $\B$-module category. Put
\[
\B^*_\M:=\Fun\B(\M,\M),
\qquad F\ot G=G\circ F,
\]
so that evaluation defines a right action of $\B^*_\M$ on $\M$. Here and below, $\circ$ denotes ordinary composition.

The tensor functors
$\alpha_\M^\pm:\B\to\B^*_\M$, called the two \emph{$\alpha$-induction
functors} \cite[Section 2.1]{BEK01}, have underlying endofunctor
$m\mapsto X\triangleright m$. For $U,X\in\B$, their module functor
constraints
\[
\alpha_\M^\pm(X)(U\triangleright m)
\longrightarrow U\triangleright\alpha_\M^\pm(X)(m),
\]
are induced by $c_{X,U}$ for $+$ and $c_{U,X}^{-1}$ for $-$.
Their tensor functor structure morphisms
\[
\alpha_\M^\pm(Y)\circ\alpha_\M^\pm(X)
\longrightarrow
\alpha_\M^\pm(X\ot Y)
\]
are induced by $c_{Y,X}$ for $+$ and $c_{X,Y}^{-1}$ for $-$; the unit
constraints come from the module unit constraints. Compare
\cite[equation~(4.4)]{DN2}, where the reversed tensor product is placed
in the source instead of the target. 
They satisfy the \emph{$\alpha$-interchange law}: for $X,Y\in\B$, there is a natural isomorphism of $\B$-module functors
\begin{equation}
\label{eqn:alpha interchange}
\iota^\M_{X,Y}:\alpha_\M^+(X)\circ\alpha_\M^-(Y)
\xRightarrow{\ \sim\ }
\alpha_\M^-(Y)\circ\alpha_\M^+(X),
\qquad
(\iota^\M_{X,Y})_m=c_{X,Y}\triangleright\id_m,
\end{equation}
where the module associators in the component formula are suppressed.

Let $F_\M:\Z(\B)\to\B^*_\M$ be the central action functor,
$Z\mapsto Z\triangleright-$. Its tensor functor structure morphism
\[
F_\M(Z)\ot F_\M(W)=F_\M(W)\circ F_\M(Z)
\longrightarrow F_\M(Z\ot W)
\]
is induced by $c^{\Z(\B)}_{W,Z}:W\ot Z\to Z\ot W$ and the module
associators. It restricts to the structure morphisms of $\alpha_\M^+$ and $\alpha_\M^-$ just specified. Its
restriction along the canonical braided tensor functor
$\B\bt\B^\rev\to\Z(\B)$ satisfies the ordinary composition isomorphism
\begin{equation}
\label{eqn:central action alpha}
F_\M(X\bt Y)\cong\alpha_\M^+(X)\circ\alpha_\M^-(Y).
\end{equation}
For simple objects $X_i,X_j\in\B$, the entries of the corresponding \emph{coupling matrix} are
\[
z^\M_{ij}:=\dim\Hom_{\B^*_\M}\bigl(\alpha_\M^+(X_i),\alpha_\M^-(X_j)\bigr).
\]
This matrix has nonnegative integral entries. If $\B$ is non-degenerate
and pseudo-unitary, equipped with its canonical ribbon structure,
and $\M$ is indecomposable, then it is the usual $\alpha$-induction modular invariant
matrix and commutes with $S$ and $T$.
The full center coefficients below use the diagonal kernel convention and therefore involve a dual in the second variable.

\enlargethispage{2\baselineskip}
\begin{definition}
\label{def: modular invariant endofunctor}
Let $\bigl(\alpha_\M^+\bigr)^{\mathrm R}:\B^*_\M\to\B$ be the right adjoint of $\alpha_\M^+$. Define
\begin{equation}
\label{eqn:ZM definition}
Z_\M:=\bigl(\alpha_\M^+\bigr)^{\mathrm R}\alpha_\M^-:\B\longrightarrow\B.
\end{equation}
\end{definition}
The adjunction gives a natural isomorphism
\begin{equation}
\label{eqn:ZM representing formula}
\Hom_\B(X,Z_\M(Y))
\cong
\Hom_{\B^*_\M}\bigl(\alpha_\M^+(X),\alpha_\M^-(Y)\bigr).
\end{equation}
In particular, $z^\M_{ij}=\dim\Hom_\B(X_i,Z_\M(X_j))$.

\begin{proposition}
\label{prop: Z braided lax monoidal}
The functor \(Z_\M\) has a canonical braided lax monoidal structure.
\end{proposition}
\begin{proof}
Let \(\epsilon\) be the counit of
\(\alpha_\M^+\dashv(\alpha_\M^+)^{\mathrm R}\).
The structure morphism
\[
\mu_{U,V}:Z_\M(U)\ot Z_\M(V)\longrightarrow Z_\M(U\ot V)
\]
is defined by adjunction from the composite
\[
\alpha_\M^+\bigl(Z_\M(U)\ot Z_\M(V)\bigr)
\xrightarrow{\sim}
\alpha_\M^+Z_\M(V)\circ\alpha_\M^+Z_\M(U)
\xrightarrow{\epsilon_V\circ\epsilon_U}
\alpha_\M^-(V)\circ\alpha_\M^-(U)
\xrightarrow{\sim}
\alpha_\M^-(U\ot V).
\]
The first isomorphism is the inverse of the structure morphism for \(\alpha_\M^+\), and the last is the structure morphism for \(\alpha_\M^-\). The unit is defined similarly. Associativity and the unit axioms follow from the tensor functor structures of \(\alpha_\M^\pm\) and the adjunction.

It remains to check compatibility with the braiding. 
For $X,Y\in\B$, let
$\epsilon_X:\alpha_\M^+Z_\M(X)\Rightarrow\alpha_\M^-(X)$ and
$\epsilon_Y:\alpha_\M^+Z_\M(Y)\Rightarrow\alpha_\M^-(Y)$ be the counit
components.  We use the tensor functor structures of \(\alpha_\M^+\) and \(\alpha_\M^-\) specified above.
Under the adjunction
$\alpha_\M^+\dashv(\alpha_\M^+)^{\mathrm R}$, the two sides of
\begin{equation}
\label{eq:goal}
Z_\M(c_{X,Y})\,\mu_{X,Y}
=\mu_{Y,X}\,c_{Z_\M(X),Z_\M(Y)}
\end{equation}
correspond to the boundary paths of the following diagram: down the left
column and then across the bottom, or across the top and then down the
right column. The notation $c_{A,B}\triangleright-$ denotes the natural
transformation with component $c_{A,B}\triangleright\id_m$; module
associators are suppressed.

\[
\begin{tikzcd}[column sep=66pt,row sep=34pt]
\alpha_\M^+\bigl(Z_\M(X)\ot Z_\M(Y)\bigr)
 \arrow[r,"{\alpha_\M^+(c_{Z_\M(X),Z_\M(Y)})}"]
 \arrow[d,"{c_{Z_\M(Y),Z_\M(X)}^{-1}\triangleright-}"']
& \alpha_\M^+\bigl(Z_\M(Y)\ot Z_\M(X)\bigr)
 \arrow[d,"{c_{Z_\M(X),Z_\M(Y)}^{-1}\triangleright-}"]
\\
\alpha_\M^+Z_\M(Y)\circ\alpha_\M^+Z_\M(X)
 \arrow[r,"{c_{Z_\M(Y),Z_\M(X)}\triangleright-}"]
 \arrow[d,"{\epsilon_Y\circ\id}"']
& \alpha_\M^+Z_\M(X)\circ\alpha_\M^+Z_\M(Y)
 \arrow[d,"{\id\circ\epsilon_Y}"]
\\
\alpha_\M^-(Y)\circ\alpha_\M^+Z_\M(X)
 \arrow[r,"{c_{Z_\M(X),Y}^{-1}\triangleright-}"]
 \arrow[d,"{\id\circ\epsilon_X}"']
& \alpha_\M^+Z_\M(X)\circ\alpha_\M^-(Y)
 \arrow[d,"{\epsilon_X\circ\id}"]
\\
\alpha_\M^-(Y)\circ\alpha_\M^-(X)
 \arrow[r,"{c_{Y,X}\triangleright-}"]
 \arrow[d,"{c_{X,Y}^{-1}\triangleright-}"']
& \alpha_\M^-(X)\circ\alpha_\M^-(Y)
 \arrow[d,"{c_{Y,X}^{-1}\triangleright-}"]
\\
\alpha_\M^-(X\ot Y)
 \arrow[r,"{\alpha_\M^-(c_{X,Y})}"']
& \alpha_\M^-(Y\ot X)
\arrow[phantom,from=1-1,to=2-2,"\text{(i)}" description]
\arrow[phantom,from=2-1,to=3-2,"\text{(ii)}" description]
\arrow[phantom,from=3-1,to=4-2,"\text{(iii)}" description]
\arrow[phantom,from=4-1,to=5-2,"\text{(iv)}" description]
\end{tikzcd}
\]
The first vertical arrows are the inverses of the structure morphisms of
$\alpha_\M^+$, and the last are the structure morphisms of
$\alpha_\M^-$. The two middle vertical arrows in each column factor the horizontal
compositions $\epsilon_Y\circ\epsilon_X$ on the left and
$\epsilon_X\circ\epsilon_Y$ on the right. 

Regions \textup{(i)} and \textup{(iv)} commute by cancellation of inverse braidings. Region \textup{(ii)} is module naturality of \(\epsilon_Y\) with respect to \(Z_\M(X)\), and region \textup{(iii)} is module naturality of \(\epsilon_X\) with respect to \(Y\). Hence the diagram commutes.
\end{proof}

\begin{lemma}
\label{lem:alpha comparison}
Let $\M,\N$ be $\B$-module categories and put
$\mathcal P=\M\bt_\B\N$. Consider the tensor functors
\[
\begin{aligned}
\lambda:\B^*_\M&\longrightarrow\B^*_{\mathcal P},
&F&\longmapsto F\bt_\B\id_\N,\\
\rho:\B^*_\N&\longrightarrow\B^*_{\mathcal P},
&G&\longmapsto\id_\M\bt_\B G.
\end{aligned}
\]
Let $\lambda^{\mathrm R}$ be the right adjoint of $\lambda$.
There is a canonical monoidal natural isomorphism
\begin{equation}
\label{eqn:alpha comparison}
\alpha_\M^-(\alpha_\N^+)^{\mathrm R}
\cong\lambda^{\mathrm R}\rho.
\end{equation}
\end{lemma}
\begin{proof}
Balancing gives a monoidal natural isomorphism
\begin{equation}
\label{eqn:alpha balancing}
\chi:\lambda\alpha_\M^-\xRightarrow{\ \sim\ }\rho\alpha_\N^+,
\end{equation}
induced by $(X\triangleright m)\bt_\B n\cong m\bt_\B(X\triangleright n)$.

Write $\eta^\lambda:\id\longrightarrow\lambda^{\mathrm R}\lambda$
for the unit of $\lambda\dashv\lambda^{\mathrm R}$, and
$\epsilon^+:\alpha_\N^+(\alpha_\N^+)^{\mathrm R}\longrightarrow\id$
for the counit of $\alpha_\N^+\dashv(\alpha_\N^+)^{\mathrm R}$.
Define
\begin{equation}
\label{eqn:BC mate}
\beta:\alpha_\M^-(\alpha_\N^+)^{\mathrm R}
\xrightarrow{\ \eta^\lambda\ }
\lambda^{\mathrm R}\lambda\alpha_\M^-(\alpha_\N^+)^{\mathrm R}
\xrightarrow{\ \lambda^{\mathrm R}\chi(\alpha_\N^+)^{\mathrm R}\ }
\lambda^{\mathrm R}\rho\alpha_\N^+(\alpha_\N^+)^{\mathrm R}
\xrightarrow{\ \lambda^{\mathrm R}\rho\epsilon^+\ }
\lambda^{\mathrm R}\rho.
\end{equation}
The transformation $\beta$ is monoidal since $\chi$ and the relevant
adjunction unit and counit are monoidal for the canonical lax structures
on the right adjoints. We show that it is invertible.


Write $\epsilon^\lambda:\lambda\lambda^{\mathrm R}\longrightarrow\id$
for the counit of $\lambda\dashv\lambda^{\mathrm R}$.

Give $\B^*_\N$ the left $\B$-action
$X\triangleright G=\alpha_\N^+(X)\ot G=G\circ\alpha_\N^+(X)$.
The functor $\alpha_\N^+$ and its right adjoint are module functors; the latter
has the module structure induced by adjunction and rigidity \cite{EGNO}.
Tensoring the adjunction
$\alpha_\N^+\dashv(\alpha_\N^+)^{\mathrm R}$ with $\M$ gives a module functor
\[
\widehat\alpha:=\id_\M\bt_\B\alpha_\N^+:
\M\longrightarrow\M\bt_\B\B^*_\N
\]
with right adjoint $\widehat\alpha^{\mathrm R}$. Write
$\widehat\eta$ and $\widehat\epsilon$ for the unit and counit of this
adjunction.

The correspondence between tensor products and module functors
\cite[Proposition~3.5]{ENO2} gives an equivalence
\begin{equation}
\label{eqn:BC tensor Hom}
\Phi:\B^*_{\mathcal P}\xrightarrow{\ \sim\ }
\FUN_\B\bigl(\M,\M\bt_\B\B^*_\N\bigr).
\end{equation}
Its quasi-inverse sends a functor $K$ to
$(\id_\M\bt_\B\mathrm{ev})(K\bt_\B\id_\N)$, where
$\mathrm{ev}(G\bt_\B n)=G(n)$; it is balanced by the
$\B$-module functor structure of $G$.

There are natural identifications of module functors
\begin{equation}
\label{eqn:BC tensor Hom identifications}
\Phi(\lambda F)\cong\widehat\alpha F,
\qquad
\Phi(\rho G)(m)\cong m\bt_\B G.
\end{equation}
The right unit equivalence gives a natural identification
\begin{equation}
\label{eqn:BC right unit}
\widehat\alpha^{\mathrm R}\Phi\rho
\cong\alpha_\M^-(\alpha_\N^+)^{\mathrm R}.
\end{equation}
At $G\in\B^*_\N$, the functor on the right sends $m$ to
$(\alpha_\N^+)^{\mathrm R}(G)\triangleright m$, with module constraint
$c^{-1}_{U,(\alpha_\N^+)^{\mathrm R}(G)}$. 

Using \eqref{eqn:BC tensor Hom identifications} and
\eqref{eqn:BC right unit}, define
\begin{equation}
\label{eqn:BC isomorphism}
\gamma:\lambda^{\mathrm R}\rho
\xrightarrow{\ \widehat\eta\,\lambda^{\mathrm R}\rho\ }
\widehat\alpha^{\mathrm R}\widehat\alpha\lambda^{\mathrm R}\rho
\xrightarrow{\ \widehat\alpha^{\mathrm R}\Phi(\epsilon^\lambda\rho)\ }
\widehat\alpha^{\mathrm R}\Phi\rho
\xrightarrow{\ \sim\ }
\alpha_\M^-(\alpha_\N^+)^{\mathrm R}.
\end{equation}
The second arrow uses
$\widehat\alpha\lambda^{\mathrm R}\rho
\cong\Phi\lambda\lambda^{\mathrm R}\rho$, and the last arrow is
\eqref{eqn:BC right unit}.

Under \eqref{eqn:BC tensor Hom identifications}, $\Phi(\chi_X)$ is the
balancing map
\begin{equation}
\label{eqn:BC balancing component}
(X\triangleright m)\bt_\B\id_\N
\longrightarrow
m\bt_\B\alpha_\N^+(X).
\end{equation}
By this identification and the definition of $\widehat\epsilon$,
\[
\Phi\bigl(\rho\epsilon^+\circ\chi(\alpha_\N^+)^{\mathrm R}\bigr)
=\widehat\epsilon_{\Phi\rho}.
\]


We suppress the evident whiskering in the diagrams below.

In the first diagram, the top row is \eqref{eqn:BC mate}, with its last
two arrows combined, and the right column is \eqref{eqn:BC isomorphism}.
Thus the top and right boundary is $\gamma\beta$.
\nopagebreak[4]
\begin{equation}
\label{eqn:BC gamma beta diagram}
\begin{tikzcd}[column sep=70pt,row sep=42pt]
\widehat\alpha^{\mathrm R}\Phi\rho
 \arrow[r,"{\eta^\lambda}"]
 \arrow[d,"{\widehat\eta}"']
& \lambda^{\mathrm R}\lambda\widehat\alpha^{\mathrm R}\Phi\rho
 \arrow[r,"{\lambda^{\mathrm R}(\rho\epsilon^+\chi)}"]
 \arrow[d,"{\widehat\eta}"']
& \lambda^{\mathrm R}\rho
 \arrow[d,"{\widehat\eta}"]
\\
\widehat\alpha^{\mathrm R}\widehat\alpha
 \widehat\alpha^{\mathrm R}\Phi\rho
 \arrow[r,"{\widehat\alpha^{\mathrm R}\widehat\alpha\eta^\lambda}"]
 \arrow[dr,equal]
& \widehat\alpha^{\mathrm R}\widehat\alpha\lambda^{\mathrm R}\lambda
 \widehat\alpha^{\mathrm R}\Phi\rho
 \arrow[r,"{\widehat\alpha^{\mathrm R}\widehat\alpha\lambda^{\mathrm R}(\rho\epsilon^+\chi)}"]
 \arrow[d,"{\widehat\alpha^{\mathrm R}\Phi\epsilon^\lambda}"']
& \widehat\alpha^{\mathrm R}\widehat\alpha\lambda^{\mathrm R}\rho
 \arrow[d,"{\widehat\alpha^{\mathrm R}\Phi\epsilon^\lambda}"]
\\
& \widehat\alpha^{\mathrm R}\widehat\alpha
 \widehat\alpha^{\mathrm R}\Phi\rho
 \arrow[r,"{\widehat\alpha^{\mathrm R}\widehat\epsilon}"']
& \widehat\alpha^{\mathrm R}\Phi\rho
\arrow[phantom,from=1-1,to=2-2,"\text{(i)}" description]
\arrow[phantom,from=1-2,to=2-3,"\text{(ii)}" description]
\arrow[phantom,from=2-1,to=3-2,"\text{(iii)}" description,yshift=9pt,xshift=12pt]
\arrow[phantom,from=2-2,to=3-3,"\text{(iv)}" description]
\end{tikzcd}
\end{equation}

Regions \textup{(i)} and \textup{(ii)} commute by naturality of
$\widehat\eta$. Region \textup{(iii)} is the triangle identity for
$\lambda\dashv\lambda^{\mathrm R}$, and region \textup{(iv)} is naturality
of $\epsilon^\lambda$, using the identification above. Thus the boundary
composite $\gamma\beta$ reduces to the triangle identity for
$\widehat\alpha\dashv\widehat\alpha^{\mathrm R}$ and hence is the identity.

For the second diagram, let $\Phi^{-1}$ denote the quasi-inverse described
above. We use $\Phi^{-1}\widehat\alpha\cong\lambda$, with the unit and
counit of this equivalence understood. The top row is
\eqref{eqn:BC isomorphism}, and the right column is \eqref{eqn:BC mate},
with its last two arrows combined. Thus the top and right boundary is
$\beta\gamma$.

\begin{equation}
\label{eqn:BC beta gamma diagram}
\begin{tikzcd}[column sep=70pt,row sep=46pt]
\lambda^{\mathrm R}\rho
 \arrow[r,"{\widehat\eta}"]
 \arrow[d,"{\eta^\lambda}"']
& \widehat\alpha^{\mathrm R}\widehat\alpha\lambda^{\mathrm R}\rho
 \arrow[r,"{\widehat\alpha^{\mathrm R}\Phi\epsilon^\lambda}"]
 \arrow[d,"{\eta^\lambda}"']
& \widehat\alpha^{\mathrm R}\Phi\rho
 \arrow[d,"{\eta^\lambda}"]
\\
\lambda^{\mathrm R}\lambda\lambda^{\mathrm R}\rho
 \arrow[r,"{\lambda^{\mathrm R}\lambda\widehat\eta}"]
 \arrow[dr,equal]
& \lambda^{\mathrm R}\lambda\widehat\alpha^{\mathrm R}
 \widehat\alpha\lambda^{\mathrm R}\rho
 \arrow[r,"{\lambda^{\mathrm R}\lambda\widehat\alpha^{\mathrm R}\Phi\epsilon^\lambda}"]
 \arrow[d,"{\lambda^{\mathrm R}\Phi^{-1}\widehat\epsilon}"']
& \lambda^{\mathrm R}\lambda\widehat\alpha^{\mathrm R}\Phi\rho
 \arrow[d,"{\lambda^{\mathrm R}(\rho\epsilon^+\chi)}"]
\\
& \lambda^{\mathrm R}\lambda\lambda^{\mathrm R}\rho
 \arrow[r,"{\lambda^{\mathrm R}\epsilon^\lambda}"']
& \lambda^{\mathrm R}\rho
\arrow[phantom,from=1-1,to=2-2,"\text{(i)}" description]
\arrow[phantom,from=1-2,to=2-3,"\text{(ii)}" description]
\arrow[phantom,from=2-1,to=3-2,"\text{(iii)}" description,yshift=11pt,xshift=12pt]
\arrow[phantom,from=2-2,to=3-3,"\text{(iv)}" description]
\end{tikzcd}
\end{equation}

Regions \textup{(i)} and \textup{(ii)} commute by naturality of
$\eta^\lambda$. Region \textup{(iii)} is the triangle identity for
$\widehat\alpha\dashv\widehat\alpha^{\mathrm R}$, and region
\textup{(iv)} is naturality of $\widehat\epsilon$, again using the
identification above. Thus the boundary composite $\beta\gamma$ reduces
to the triangle identity for $\lambda\dashv\lambda^{\mathrm R}$ and hence
is the identity.

Thus $\beta$ is the canonical monoidal natural isomorphism
\eqref{eqn:alpha comparison}.
\end{proof}

\begin{proposition}
\label{prop: modular invariant multiplicativity}
For $\B$-module categories $\M,\N$, there is a canonical isomorphism
\[
Z_{\M\bt_\B\N}\cong Z_\M\circ Z_\N
\]
of braided lax monoidal endofunctors of $\B$.
\end{proposition}
\begin{proof}
Let $\mathcal P,\lambda,\rho$ be as in
Lemma~\ref{lem:alpha comparison}.
The $\B$-actions on the two tensor factors give monoidal natural
isomorphisms
\begin{equation}
\label{eqn:alpha outer actions}
\alpha_{\mathcal P}^+\cong\lambda\alpha_\M^+,
\qquad
\alpha_{\mathcal P}^-\cong\rho\alpha_\N^-.
\end{equation}
Taking the right adjoint of the first and using
\eqref{eqn:alpha comparison}, we obtain
\begin{equation}
\label{eqn:Z multiplicativity}
Z_{\mathcal P}
\cong(\alpha_\M^+)^{\mathrm R}
       \lambda^{\mathrm R}\rho\alpha_\N^-
\cong(\alpha_\M^+)^{\mathrm R}\alpha_\M^-
       (\alpha_\N^+)^{\mathrm R}\alpha_\N^-
=Z_\M\circ Z_\N.
\end{equation}
These isomorphisms preserve the lax monoidal structures, hence give
an isomorphism of braided lax monoidal endofunctors.
\end{proof}

We next relate \(Z_\M\) to the full center. Assume now that \(\B\) is non-degenerate, so that \(\Z(\B)\simeq\B\bt\B^\rev\). Let \(R_\M\) be the right adjoint of the central action functor
$F_\M:\Z(\B)\longrightarrow\B^*_\M$.

 The \emph{full center} of $\M$ is the commutative algebra $L_\M:=R_\M(\id_\M)$, with multiplication induced by the tensor product in $\B^*_\M$ and the lax monoidal structure of $R_\M$. For $\M\simeq\B_A$ this is the full center of $A$ in the sense of Davydov \cite{DavCenter}; see also \cite{DKR}. For indecomposable $\M$, the algebra $L_\M$ is connected and Lagrangian \cite[Proposition 4.8]{DMNO}. For a decomposable module category, we use the direct sum of the full centers of its indecomposable summands, with blockwise multiplication.

Let $\O(\B)$ be a set of representatives of the isomorphism classes
of simple objects of $\B$. Write
\[
\mathfrak D_\B:=\bigoplus_{T\in\O(\B)}T^*\bt T
\]
for the diagonal Lagrangian algebra.  For a braided lax monoidal endofunctor $Z$ of $\B$, put
\begin{equation}
\label{eqn:kernel of Z}
K_Z:=(Z\bt\id_\B)(\mathfrak D_\B),
\end{equation}
with multiplication induced by the diagonal multiplication and the lax monoidal structure morphism of $Z$.

\begin{lemma}
\label{lem:full center kernel}
There is a canonical isomorphism of algebras
\begin{equation}
\label{eqn:full center kernel}
L_\M\cong K_{Z_\M}.
\end{equation}
\end{lemma}
\begin{proof}
Let \(T:\B\bt\B^\rev\to\B\), \(T(X\bt Y)=X\ot Y\), be the
forgetful functor under the canonical equivalence
\(\B\bt\B^\rev\simeq\Z(\B)\).
Recall that the diagonal algebra \(\mathfrak D_{\mathcal B}=T^{\mathrm R}(\mathbf 1)\) is equipped with the canonical algebra structure on this right adjoint image \cite[Lemma~3.5]{DMNO}. 
Its counit \(\mathrm{ev}:T(\mathfrak D_{\mathcal B})\to\mathbf 1\) is therefore an algebra morphism.

The counit \(\epsilon:\alpha_\M^+Z_\M\Rightarrow\alpha_\M^-\) is
monoidal. For a $\B$-module endofunctor \(E\), its module constraint gives
\[
\alpha_\M^-(Y)\ot E
=E\circ\alpha_\M^-(Y)
\xrightarrow{\sim}
\alpha_\M^-(Y)\circ E
=E\ot\alpha_\M^-(Y).
\]
These isomorphisms give a tensor functor structure on
\(E\bt Y\mapsto E\ot\alpha_\M^-(Y)\) from
\(\B^*_\M\bt\B^\rev\) to \(\B^*_\M\).
Applying it to \(\epsilon\bt\id_\B\) gives a
monoidal natural transformation
\(F_\M\circ(Z_\M\bt\id_\B)\Rightarrow\alpha_\M^-\circ T\).
On \(X\boxtimes Y\), this transformation is induced by
\[
\alpha_{\mathcal M}^{+}Z_{\mathcal M}(X)\otimes\alpha_{\mathcal M}^{-}(Y)
\xrightarrow{\;\epsilon_X\otimes\mathrm{id}\;}
\alpha_{\mathcal M}^{-}(X)\otimes\alpha_{\mathcal M}^{-}(Y).
\]
Evaluating it on \(\mathfrak D_\B\) and composing
with evaluation gives an algebra morphism
\begin{equation}
\label{eqn:kernel evaluation}
F_\M(K_{Z_\M})
\longrightarrow
\alpha_\M^-\bigl(T(\mathfrak D_\B)\bigr)
\xrightarrow{\alpha_\M^-(\mathrm{ev})}
\alpha_\M^-(\be)\cong\id_\M.
\end{equation}
Since $F_\M$ is a tensor functor, the adjunction
$F_\M\dashv R_\M$ restricts to an adjunction between the corresponding
categories of algebras \cite[Proposition~2.16]{KR09}.
 Hence the algebra morphism \eqref{eqn:kernel evaluation} corresponds to a canonical algebra morphism
\[
f: K_{Z_\M}\longrightarrow R_\M(\id_\M)=L_\M.
\]
To show that \(f\) is an isomorphism, compose maps from \(X\bt Y\) with \(f\) and apply the full center adjunction. By construction, the resulting map is the chain
\begin{equation}
\label{eqn:alpha kernel calculation}
\begin{aligned}
\Hom_{\B\bt\B^\rev}(X\bt Y,K_{Z_\M})
&\cong\Hom_\B\bigl(X,Z_\M(Y^*)\bigr)\\
&\cong\Hom_{\B^*_\M}
 \bigl(\alpha_\M^+(X),\alpha_\M^-(Y^*)\bigr)\\
&\cong\Hom_{\B^*_\M}
 \bigl(F_\M(X\bt Y),\id_\M\bigr).
\end{aligned}
\end{equation}
These are the co-Yoneda, adjunction, and rigidity isomorphisms, respectively; the latter two use precisely the counit and evaluation appearing in \eqref{eqn:kernel evaluation}. By additivity this gives a natural isomorphism on Hom spaces for every object of \(\B\bt\B^\rev\), so \(f\) is an isomorphism by Yoneda.
\end{proof}

Thus the full center and $Z_\M$ are two forms of the same structured invariant: $Z_\M$ is its endofunctor form and $L_\M$ its kernel.

\begin{remark}
\label{rem: modular invariant representation}
Proposition~\ref{prop: modular invariant multiplicativity} induces a unital ring homomorphism
\[
K^\mmod(\B)\longrightarrow\End_{\mathbb Z}(K_0(\B)),
\qquad
[\M]\longmapsto[Z_\M].
\]
When $\B$ is non-degenerate and pseudo-unitary, equipped with its
canonical ribbon structure, the classes of module categories are
sent to nonnegative integral matrices commuting with $S$ and $T$;
see \cite[Section~5]{FRS}.
\end{remark}

\begin{remark}
\label{rem: relation to previous work}
The endofunctor \(Z_\M\) already appears, in equivalent notation
in the ribbon setting, in \cite[Proposition~3.6]{FFRS}.
Numerical additivity and multiplicativity are established in
\cite[Proposition~5.3]{FRS}.

The description of full centers through module categories is given
in \cite[Section~6]{DavCenter}, and their \(2\)-functoriality in
\cite[Theorem~7.10]{DKR}.
In the unitary modular setting, Bischoff--Kawahigashi--Longo
\cite[Proposition~4.18]{BKL} identify the \(\alpha\)-induction
construction with the full center as a Q-system.
Related composition results for categories of bimodule functors,
relative tensor products of Lagrangian Q-systems, and internal Homs
appear in \cite[Theorem~3.1.7]{KZCenter},
\cite[Theorem~2.6]{KawahigashiRT}, and
\cite[Theorem~4.5]{KYZ}, respectively.

Here we equip \(Z_\M\) with a braided lax monoidal structure and
prove that \(Z_{\M\bt_\B\N}\cong Z_\M\circ Z_\N\) respects this
structure. We have not found this particular formulation stated
explicitly in the literature, although it is closely related to
the preceding constructions.
\end{remark}

\subsection{Convolution of \'etale algebras}
\label{subsect: Convolution of etale algebras}

Let $\B$ be a non-degenerate braided fusion category. We identify $\Z(\B)$ with $\B\bt\B^\rev$ via the canonical braided equivalence. Distinguish four copies of $\B$ and define a functor
\[
\operatorname{ctr}_{23}:\B_1\bt\B_2^\rev\bt\B_3\bt\B_4^\rev\longrightarrow \B_1\bt\B_4^\rev
\]
on pure tensors by
\begin{equation}
\label{eqn:middle contraction}
\operatorname{ctr}_{23}(X\bt U\bt V\bt Y):=\Hom_\B(\be,U\ot V)\otimes_k(X\bt Y).
\end{equation}
Here $X,U,V,Y$ are objects of $\B$ and $\otimes_k$ is the natural action of finite-dimensional vector spaces on the finite semisimple category $\B_1\bt\B_4^\rev$. For pure tensors $P=X\bt U\bt V\bt Y$ and $Q=X'\bt U'\bt V'\bt Y'$, define
\begin{equation}
\label{eqn:middle lax structure}
\mu_{P,Q}:\operatorname{ctr}_{23}(P)\ot\operatorname{ctr}_{23}(Q)\longrightarrow\operatorname{ctr}_{23}(P\ot Q)
\end{equation}
by tensoring the linear map
\begin{equation}
\label{eqn:middle crossing}
\begin{aligned}
\Hom_\B(\be,U\ot V)\otimes_k\Hom_\B(\be,U'\ot V')
&\longrightarrow
\Hom_\B\bigl(\be,(U\ot U')\ot(V\ot V')\bigr),\\
f\otimes_k g
&\longmapsto
f*g:=(\id_U\ot c_{V,U'}\ot\id_{V'})\circ(f\ot g)
\end{aligned}
\end{equation}
with $\id_{(X\ot X')\bt(Y\ot Y')}$, using the canonical
identifications for $\otimes_k$ and the componentwise tensor
product in $\B_1\bt\B_4^\rev$.
The unit morphism
\begin{equation}
\label{eqn:middle lax unit}
\be\bt\be\longrightarrow\operatorname{ctr}_{23}(\be\bt\be\bt\be\bt\be)
\end{equation}
corresponds to $\id_\be\in\Hom_\B(\be,\be)$.

\begin{lemma}
\label{lem:middle contraction}
The morphisms \eqref{eqn:middle lax structure} and \eqref{eqn:middle lax unit} make $\operatorname{ctr}_{23}$ a unital braided lax monoidal functor.
\end{lemma}
\begin{proof}
For $f_i:\be\to U_i\ot V_i$, both sides of the lax associativity axiom are obtained from $f_1\ot f_2\ot f_3$ by moving $V_1$ past $U_2,U_3$ and $V_2$ past $U_3$, without changing the relative order of the $U_i$'s or of the $V_i$'s. In one parenthesization one first moves $V_1$ past $U_2$ and then $V_1\ot V_2$ past $U_3$; in the other one first moves $V_2$ past $U_3$ and then $V_1$ past $U_2\ot U_3$. Expanding $c_{V_1\ot V_2,U_3}$ and $c_{V_1,U_2\ot U_3}$ by the two hexagon identities reduces the equality to the Yang--Baxter identity. The outer components are the ordinary associativity constraints in $\B_1\bt\B_4^\rev$.

The unit axioms follow from $c_{\be,U}=c_{U,\be}=\id_U$. The braiding between $U\bt V$ and $U'\bt V'$ in $\B^\rev\bt\B$ is $c_{U',U}^{-1}\bt c_{V,V'}$. 
It remains to check compatibility of the lax monoidal structure with the
braidings. On the middle multiplicity spaces, this amounts to
\begin{equation}
\label{eqn:middle braidedness}
(c_{U',U}^{-1}\ot c_{V,V'})(f*g)=(\id_{U'}\ot c_{V',U}\ot\id_V)(g\ot f).
\end{equation}
Naturality gives $g\ot f=c_{U\ot V,U'\ot V'}(f\ot g)$, and \eqref{eqn:middle braidedness} follows by expanding this block braiding with the hexagon identities and applying the Yang--Baxter identity. On the outer factors, the source and target braidings are respectively $c_{X,X'}$ and $c_{Y',Y}^{-1}$, so their compatibility is immediate.
\end{proof}

\begin{definition}
\label{def:convolution}
Let $A,C\in\B\bt\B^\rev$ be algebra objects. Their
\emph{convolution} is
\[
A\odot C:=\operatorname{ctr}_{23}(A_{12}\bt C_{34}),
\]
with multiplication
\[
(A\odot C)\ot(A\odot C)
\xrightarrow{\mu}
\operatorname{ctr}_{23}\bigl((A\ot A)_{12}\bt(C\ot C)_{34}\bigr)
\xrightarrow{\operatorname{ctr}_{23}(m_A\bt m_C)}
A\odot C
\]
and unit induced by \eqref{eqn:middle lax unit} and
$\eta_A\bt\eta_C$.
\end{definition}

By Lemma~\ref{lem:middle contraction} and the algebra axioms for
$A$ and $C$, these maps satisfy the associativity and unit axioms.
Thus $A\odot C$ is an algebra in $\B\bt\B^\rev$.

For local Q-systems, the corresponding construction by tensor product and restriction appears in \cite[Proposition~2.4]{KawahigashiRT}.

\begin{proposition}
\label{prop:etale convolution closure}
Let $A,C$ be \'etale algebras in $\B\bt\B^\rev$. Then $A\odot C$ is \'etale.
\end{proposition}
\begin{proof}
If $A$ and $C$ are commutative, then $D:=A_{12}\bt C_{34}$ is commutative. By Lemma~\ref{lem:middle contraction},
\[
\operatorname{ctr}_{23}(m_D)\mu\,c
=\operatorname{ctr}_{23}(m_D)\operatorname{ctr}_{23}(c_{D,D})\mu
=\operatorname{ctr}_{23}(m_Dc_{D,D})\mu
=\operatorname{ctr}_{23}(m_D)\mu,
\]
where $c$ is the braiding on the two output factors and $\mu$ is the lax structure map. Thus $A\odot C$ is commutative.

Put $\mathcal E=\B_1\bt\B_4^\rev$ and $\mathcal D=\B_1\bt\B\bt\B_4^\rev$. The functor
\[
T:\B_1\bt\B_2^\rev\bt\B_3\bt\B_4^\rev \longrightarrow\mathcal D,\qquad
X\bt U\bt V\bt Y \longmapsto X\bt(U\ot V)\bt Y
\]
is tensor, with the canonical outer factor
isomorphisms and with tensor functor structure on the middle factor given by
\[
(U\ot V)\ot(U'\ot V')
\xrightarrow{\ \id_U\ot c_{V,U'}\ot\id_{V'}\ }
(U\ot U')\ot(V\ot V').
\]
The tensor inclusion 
\[
\mathcal E\longrightarrow\mathcal D,
\qquad
X\bt Y\longmapsto X\bt\be\bt Y,
\]
has right adjoint
\[
R(X\bt W\bt Y)
=
\Hom_\B(\be,W)\otimes_k(X\bt Y).
\]
Its canonical lax structure tensors two maps from the tensor unit. Consequently $\operatorname{ctr}_{23}=R\circ T$ as lax monoidal functors: on the middle multiplicity spaces the composite lax monoidal structure morphism sends $f,g$ to
\[
(\id_U\ot c_{V,U'}\ot\id_{V'})\circ(f\ot g),
\]
as in \eqref{eqn:middle crossing}, and its unit is the same identity map as in \eqref{eqn:middle lax unit}.

Since the external tensor product of separable algebras is separable, \(A_{12}\boxtimes C_{34}\) is separable. A tensor functor preserves separable algebras, so \(T(A_{12}\boxtimes C_{34})\) is separable.
By \cite[Corollary~4.5]{JY}, the right adjoint $R$ of the inclusion $\E\hookrightarrow\D$ preserves separable
algebras, so
\[
A\odot C=R(T(A_{12}\bt C_{34}))
\]
is separable.
Together with commutativity, this proves \'etaleness.
\end{proof}

If $A=\bigoplus_i A_i'\bt A_i''$ and $C=\bigoplus_j C_j'\bt C_j''$, then the underlying object is
\[
A\odot C=
\bigoplus_{i,j}\Hom_\B(\be,A_i''\ot C_j')\otimes(A_i'\bt C_j'').
\]
For $f:\be\to A_i''\ot C_j'$ and $g:\be\to A_k''\ot C_\ell'$, the middle part of their product is obtained from
\[
\be\xrightarrow{f\ot g}
A_i''\ot C_j'\ot A_k''\ot C_\ell'
\xrightarrow{\id\ot c_{C_j',A_k''}\ot\id}
A_i''\ot A_k''\ot C_j'\ot C_\ell',
\]
followed by the appropriate components of $m_A$ and $m_C$.

Equivalently, convolution admits the finite semisimple coend formula
\begin{equation}
\label{eqn:convolution}
\Hom(X\bt Y,A\odot C)
\cong
\bigoplus_{S\in\O(\B)}
\Hom(X\bt S,A)\otimes
\Hom(S^*\bt Y,C),
\qquad X,Y\in\B,
\end{equation}
which is independent of the choice of the representative set $\O(\B)$.

\begin{lemma}
\label{lem:kernel convolution}
For braided lax monoidal endofunctors $Z,W$ of $\B$, there is a canonical algebra isomorphism $K_Z\odot K_W\cong K_{Z\circ W}$.
\end{lemma}
\begin{proof}
Using \eqref{eqn:convolution} and finite semisimple co-Yoneda we obtain
\begin{equation}
\label{eqn:composition of kernels calculation}
\begin{aligned}
\Hom_{\B\bt\B^\rev}(X\bt Y,K_Z\odot K_W)
&\cong\bigoplus_{S\in\O(\B)}
\Hom_\B(X,Z(S^*))\otimes\Hom_\B(S^*,W(Y^*))\\
&\cong\Hom_\B(X,Z(W(Y^*))).
\end{aligned}
\end{equation}
Hence $K_Z\odot K_W\cong K_{Z\circ W}$ as objects. Under this isomorphism the convolution structure morphism induces exactly the lax monoidal structure of the composite:
\[
ZW(U)\ot ZW(V)
\longrightarrow Z\bigl(W(U)\ot W(V)\bigr)
\longrightarrow ZW(U\ot V).
\]
The convolution unit maps to the composite unit
$\be\to Z(\be)\to ZW(\be)$, so the units agree as well.
Thus
\begin{equation}
\label{eqn:composition of kernels}
K_Z\odot K_W\cong K_{Z\circ W}
\end{equation}
canonically as algebras.
\end{proof}

\begin{proposition}
\label{prop:convolution etale}
Let $\M,\N$ be indecomposable $\B$-module categories, and let $L_\M,L_\N\in\Z(\B)$ be their full centers.
There is a canonical isomorphism of commutative separable algebras
\[
L_{\M\bt_\B\N}\cong L_\M\odot L_\N.
\]
\end{proposition}
\begin{proof}
Combining Lemma~\ref{lem:full center kernel},
Lemma~\ref{lem:kernel convolution}, and
Proposition~\ref{prop: modular invariant multiplicativity}
gives canonical algebra isomorphisms
\[
L_\M\odot L_\N
\cong K_{Z_\M}\odot K_{Z_\N}
\cong K_{Z_\M\circ Z_\N}
\cong K_{Z_{\M\bt_\B\N}}
\cong L_{\M\bt_\B\N}.\qedhere
\]
\end{proof}

\subsection{A parameterization of Lagrangian algebras by residual data}
\label{subsect: A parameterization of Lagrangian algebras by residual data}

Let $\B$ be a braided fusion category, and let $A\in \B$ be an \'etale algebra. For a fusion subcategory $\D \subset \B$, we denote by $A\cap \D$ the maximal \'etale subalgebra of $A$ contained in $\D$. This subalgebra is well-defined: if $F:\D \to \B_A$ denotes the restriction of the free module functor $\B \to \B_A$, $X \mapsto A\ot X$, and $I: \B_A \to \D$ is the right adjoint of $F$, then $A\cap \D=I(\be)$.

Recall the parameterization of Lagrangian algebras in the tensor product
$\B\bt\C$ of non-degenerate braided fusion categories
\cite[Section~3.2]{DNO}. The groupoid of such algebras is equivalent to
the groupoid of triples $(A_1,A_2,\phi)$, where $A_1\in\B$ and
$A_2\in\C$ are connected \'etale algebras and
$\phi:(\B_{A_1}^\loc)^\rev\xrightarrow{\sim}\C_{A_2}^\loc$
is a braided tensor equivalence.

An isomorphism between triples $(A_1,A_2,\phi)$ and
$(A_1',A_2',\phi')$ consists of algebra isomorphisms
$\alpha_1:A_1\xrightarrow{\sim}A_1'$ and
$\alpha_2:A_2\xrightarrow{\sim}A_2'$, together with a specified
monoidal natural isomorphism
\[
\tau:\phi'\circ\ind(\alpha_1)
\xRightarrow{\ \sim\ }
\ind(\alpha_2)\circ\phi.
\]
Here $\ind(\alpha)$ denotes the braided tensor equivalence between
the corresponding categories of local modules given by transport
of module structures along $\alpha$. 
Composition of isomorphisms of triples is obtained by composing the algebra isomorphisms and the specified monoidal natural isomorphisms, using the canonical composition isomorphisms for \(\ind\).

Under this equivalence, a Lagrangian algebra $L\in\B\bt\C$ corresponds
to the triple $(A_1,A_2,\phi)$ with
$A_1=L\cap(\B\bt\be)$ and $A_2=L\cap(\be\bt\C)$.
The image of $L$ in
\[
\D:=(\B\bt\C)_{A_1\bt A_2}^\loc
\simeq \B_{A_1}^\loc\bt\C_{A_2}^\loc
\]
is the graph algebra
\[
\Delta_\phi:=
\bigoplus_{X\in\O(\B_{A_1}^\loc)}
X^*\bt\phi(X).
\]

An algebra isomorphism $L\to L'$ restricts to isomorphisms $\alpha_i$
of the two endpoints. After transport along $\alpha_i$, its map on
the summand indexed by $X$ has the form
$\id_{\ind(\alpha_1)(X)^*}\bt\tau_X^{-1}$, using the canonical
duality identifications. Since the first tensor factors are
pairwise non-isomorphic simple objects, these component maps
determine a unique natural isomorphism $\tau$; compatibility
with the algebra multiplication and unit is equivalent to
monoidality of $\tau$. Conversely, every
$(\alpha_1,\alpha_2,\tau)$ defines an algebra isomorphism in this way.

\begin{definition}
We call a triple $(A_1,A_2,\phi)$ as above \emph{residual data}
for $\B\bt\C$.
\end{definition}

\begin{definition}
\label{def: Et and Loc}
Let $\B$ be a non-degenerate braided fusion category.
Denote by $\mathfrak{Et}(\B)$ the finite poset of isomorphism classes
of connected \'etale algebras in $\B$, ordered by inclusion.
For $A\in\mathfrak{Et}(\B)$, we call $\B_A^\loc$ a
\emph{localization} of $\B$.

Denote by $\mathfrak{Loc}(\B)$ the poset of braided equivalence classes
of localizations of $\B$, ordered by further localization:
$\C\leq\D$ if $\D$ is braided equivalent to a localization of $\C$.
\end{definition}

The transitivity of the order relation in $\mathfrak{Loc}(\B)$ follows from
\cite[Proposition~3.16]{DMNO}.

There is a surjective order-preserving map
\begin{equation}
\label{eqn: localization map}
\mathfrak{Et}(\B)\longrightarrow\mathfrak{Loc}(\B),
\qquad
A\longmapsto\B_A^\loc.
\end{equation}

\begin{remark}
For a metric group $(M,q)$, let $\C(M,q)$ denote the associated pointed
non-degenerate braided fusion category, defined up to braided equivalence
\cite[Section~2.11]{DGNO2}.
Let $(M_1,q_1)$ and $(M_2,q_2)$ be metric groups. The poset $\mathfrak{Et}(\C(M_i,q_i))$ identifies with the poset of isotropic subgroups $V_i\subset M_i$, and the localization map \eqref{eqn: localization map} sends $V_i$ to the metric group category associated to the residual quotient $V_i^\perp/V_i$. 
Thus, residual data for
$\C(M_1,q_1)\bt\C(M_2,q_2)^\rev$
consist of a pair of points
$V_i\in\mathfrak{Et}(\C(M_i,q_i))$
whose localizations are equivalent, together with an orthogonal
isomorphism
$\phi:V_1^\perp/V_1\xrightarrow{\sim}V_2^\perp/V_2$.
This justifies our terminology.
\end{remark}

To describe isomorphism classes of residual data, we next analyze the action of \(\Aut_{\alg,\B}(A)\) on the localization \(\B_A^\loc\).
By \cite[Corollary~3.30]{DMNO},
$\B_A^\loc$ is centrally embedded in $\B_A$, and there is a canonical
braided equivalence
\[
\Z(\B_A)\simeq\B\bt(\B_A^\loc)^\rev.
\]
Let $\E_A$ be the full subcategory of $\B_A$ spanned by the simple
objects $X$ such that $X^*\ot_A X\in\B_A^\loc$. We call $\E_A$ the
\emph{commutator subcategory} of the central embedding
$\B_A^\loc\subset\B_A$, following the terminology of \cite{GN}.

If $X,Y\in\E_A$ are simple and $Z$ is a simple constituent of
$X\ot_A Y$, then $Z^*\ot_A Z$ is a direct summand of
\[
Y^*\ot_A X^*\ot_A X\ot_A Y
\cong
(X^*\ot_A X)\ot_A(Y^*\ot_A Y)\in\B_A^\loc.
\]
Here the isomorphism uses the half-braiding of
$X^*\ot_A X\in\B_A^\loc$. Since
$\B_A^\loc$ is closed under subobjects, $Z\in\E_A$. Thus $\E_A$ is
closed under tensor products and subobjects and hence is rigid by
\cite[Corollary~F.7]{DGNO2}. Therefore $\E_A$ is a fusion subcategory of $\B_A$, cf.~\cite[Remark~4.9]{GN}
for the corresponding statement
when the ambient Grothendieck ring is commutative.

Let $\E_A^{\ad}$ denote the adjoint subcategory of $\E_A$.
We have $\E_A^{\ad}\subseteq\B_A^\loc$. By the universal grading
theory of \cite{GN}, fusion subcategories of $\E_A$ containing
$\E_A^{\ad}$ correspond to subgroups of the universal grading group
$U(\E_A)$. Let $H_A\leq U(\E_A)$ correspond to $\B_A^\loc$.
Every object of $\B_A^\loc$ has a half-braiding with every object of
$\E_A$, so every degree in $H_A$ commutes with every degree in
$U(\E_A)$. Hence $H_A$ is central. Put $G_A:=U(\E_A)/H_A$.
This gives a canonical faithful grading
\begin{equation}
\label{eqn: commutator grading etale}
\E_A=\bigoplus_{g\in G_A}(\E_A)_g,
\qquad
(\E_A)_e=\B_A^\loc.
\end{equation}
Every homogeneous component has Frobenius--Perron dimension
$\FPdim(\B_A^\loc)$.

\begin{theorem}
\label{thm: automorphisms etale algebra}
Let $\B$ be a non-degenerate braided fusion category, and let $A$ be a connected \'etale algebra in $\B$. There is a canonical group isomorphism
\[
\Aut_{\alg,\B}(A)\xrightarrow{\sim}G_A.
\]
Under this isomorphism, the autoequivalence of $\B_A^\loc$ induced by an algebra automorphism of $A$ agrees with the canonical action of the corresponding element of $G_A$ on the neutral component $\B_A^\loc$ of the $G_A$-crossed braided category $\E_A$.
\end{theorem}

\begin{proof}
Put $H:=\Aut_{\alg,\B}(A)$. This group is finite. Indeed, for $\alpha\in H$, twist the right action on the regular $A$-bimodule by $\alpha$ to obtain the invertible bimodule ${}_AA_\alpha$. If ${}_AA_\alpha\cong{}_AA_\beta$, the isomorphism is scalar as an endomorphism of the regular left $A$-module, since $A$ is connected; right $A$-linearity then gives $\alpha=\beta$. Thus $H$ embeds into the finite group $\Inv({}_A\B_A)$.

For $\alpha\in H$, let $\C_\alpha$ be the full subcategory of $\B_A$ consisting of modules $(X,\rho)$ satisfying
\begin{equation}
\label{eqn: twisted locality sector}
\rho\,c_{X,A}c_{A,X}
=
\rho\,(\alpha^{-1}\ot\id_X).
\end{equation}
The relative tensor product action and the two hexagon identities give
\begin{equation}
\label{eqn: twisted sector multiplication}
\C_\alpha\ot_A\C_\beta\subseteq\C_{\alpha\beta},
\qquad
(\C_\alpha)^*=\C_{\alpha^{-1}},
\qquad
\C_e=\B_A^\loc.
\end{equation}
Consider
\[
\widetilde\B:=\B\bt\Vec_H
\]
with its standard $H$-crossed braided structure. Equip $A$, placed in the
neutral component of $\widetilde\B$, with the $H$-equivariant algebra
structure whose component at $\alpha\in H$ is the automorphism $\alpha$.

Let
\[
\widetilde\E_H:=\bigoplus_{\alpha\in H}\C_\alpha
\subseteq\widetilde\B_A,
\]
where $\C_\alpha$ is placed in degree $\alpha$.

By the definition of the sectors $\C_\alpha$ and the equivariantization--de-equivariantization correspondence
\cite[Subsection~2.5]{DMNO},
\begin{equation}
\label{eqn: crossed local equivariantization}
(\widetilde\E_H)^H
\simeq
(\widetilde\B^H)_A^\loc,
\end{equation}
where on the right $A$ is regarded with the $H$-equivariant structure above.

By the equivariantization dimension formula and \eqref{eqn:local-module-dimension},
\begin{equation}
\label{eqn: crossed local dimension}
\FPdim(\widetilde\E_H)
= \frac{1}{|H|} \FPdim((\widetilde\E_H)^H)
=\frac{|H|\FPdim(\B)}{\FPdim(A)^2}
=|H|\FPdim(\B_A^\loc).
\end{equation}
If $S\subseteq H$ is the support of the grading of $\widetilde\E_H$, then $S$ is a subgroup, and every nonzero component has Frobenius--Perron dimension $\FPdim(\B_A^\loc)$. Hence $\FPdim(\widetilde\E_H)=|S|\FPdim(\B_A^\loc)$, and comparison with \eqref{eqn: crossed local dimension} gives $S=H$. In particular, every $\C_\alpha$ is nonzero.

For $X\in\C_\alpha$, duality and \eqref{eqn: twisted sector multiplication} give $X^*\ot_A X\in\B_A^\loc$, so $\C_\alpha\subseteq\E_A$. If $X,Y$ are simple objects of $\C_\alpha$, then $X^*\ot_A Y\in\B_A^\loc$; hence they have the same degree in \eqref{eqn: commutator grading etale}. Thus $\C_\alpha$ lies in one homogeneous component, and equality of Frobenius--Perron dimensions shows that it is that entire component. Let
\begin{equation}
\label{eqn: hom aA}
a_A:H\longrightarrow G_A
\end{equation}
send $\alpha$ to the degree of $\C_\alpha$. Equation \eqref{eqn: twisted sector multiplication} makes $a_A$ a homomorphism.

If $a_A(\alpha)=e$, then $\C_\alpha=\B_A^\loc$, so the unit object $A$ of $\B_A$ belongs to $\C_\alpha$. Writing $m:A\ot A\to A$ for multiplication, \eqref{eqn: twisted locality sector} becomes
\[
m\,c_{A,A}c_{A,A}=m(\alpha^{-1}\ot\id_A).
\]
By commutativity, the left-hand side is $m$. Composing with
$\id_A\ot\eta_A$, where $\eta_A:\be\to A$ is the unit, gives
$\alpha^{-1}=\id_A$. Hence $a_A$ is injective.

We now construct an injection in the other direction. 

Let $A_0\subseteq A$ be the connected \'etale subalgebra corresponding
to $\E_A\subseteq\B_A$ under
\cite[Theorem~4.10]{DMNO}.
The central embedding $\B_A^\loc\subseteq\E_A$ equips the faithful
$G_A$-grading with its compatible crossed braided structure.
Reconstruction and equivariantization give a braided equivalence of pairs
\begin{equation}
\label{eqn: equivalence of reconstructed pairs}
(\B_{A_0}^\loc,A)
\simeq
(\E_A^{G_A},\Fun(G_A,k)),
\end{equation}
where $A$ on the left is regarded as a local $A_0$-module algebra:
both pairs reconstruct the central embedding
$\B_A^\loc\subseteq\E_A$.

Right translations of $\Fun(G_A,k)$ commute with its left regular
action, so they give a faithful action of $G_A$ by algebra
automorphisms in $\Rep(G_A)\subseteq\E_A^{G_A}$.
Transporting this action through
\eqref{eqn: equivalence of reconstructed pairs} and forgetting the
$A_0$-module structure gives a faithful action of $G_A$ on $A$ viewed as an algebra in $\B$.
Hence $|G_A|\leq|\Aut_{\alg,\B}(A)|$.
Since $a_A$ is injective and both groups are finite, $a_A$ is an
isomorphism.

It remains to identify the action on $\B_A^\loc$. The induced autoequivalence
$\ind_A(\alpha)$ sends a local $A$-module $(Y,\rho_Y)$ to
$\bigl(Y,\rho_Y(\alpha^{-1}\ot\id_Y)\bigr)$.
If $X\in\C_\alpha=(\E_A)_{a_A(\alpha)}$, the crossed braiding has the form
\[
c^\times_{X,Y}:
X\ot_A Y\longrightarrow a_A(\alpha)(Y)\ot_A X.
\]
The square expressing its $A$-linearity is
\[
\begin{CD}
A\ot(X\ot_A Y) @>{\id_A\ot c^\times_{X,Y}}>>
A\ot(a_A(\alpha)(Y)\ot_A X)\\
@V{\rho_{X\ot_A Y}}VV
@VV{\rho_{a_A(\alpha)(Y)}\ot_A\id_X}V\\
X\ot_A Y @>{c^\times_{X,Y}}>>
a_A(\alpha)(Y)\ot_A X.
\end{CD}
\]
Commutativity of this square is exactly
\eqref{eqn: twisted locality sector}. It follows that
$a_A(\alpha)(Y)$ has underlying object $Y$ with $A$-module structure
$\rho_Y(\alpha^{-1}\ot\id_Y)$. Hence the action of $a_A(\alpha)$ on
$\B_A^\loc$ coincides with $\ind_A(\alpha)$, with its canonical tensor
structure.
\end{proof}

\begin{remark}
If $A$ is Lagrangian, then $\B_A^\loc=\Vect$ and $\E_A$ is the pointed subcategory of $\B_A$. Hence $G_A=\Inv(\B_A)$, and Theorem~\ref{thm: automorphisms etale algebra} recovers the isomorphism of \cite[Lemma~3.10]{S}.
\end{remark}

For a connected \'etale algebra $A$, set
\[
\Gamma_A
:=
\operatorname{Im}\!\left(
\ind_A:\Aut_{\alg}(A)\longrightarrow\Aut^{\mathrm{br}}(\B_A^\loc)
\right).
\]
By Theorem~\ref{thm: automorphisms etale algebra}, $\Gamma_A$ is the image of the canonical action of $G_A$ on $\B_A^\loc$.
The action homomorphism need not be faithful: if $A$ is Lagrangian, then $\B_A^\loc=\Vect$ and $\Gamma_A$ is trivial, even when $G_A$ is not.

\begin{corollary}
\label{cor: residual double orbits}
Fix connected \'etale algebras $A_1\in\B$ and $A_2\in\C$. Suppose that there is a braided equivalence
\[
\phi_0:(\B_{A_1}^\loc)^\rev\xrightarrow{\sim}\C_{A_2}^\loc,
\]
and put $\D_2:=\C_{A_2}^\loc$. The isomorphism classes of residual data with endpoints $A_1,A_2$ are parameterized by
\[
\Gamma_{A_2}\backslash\Aut^{\mathrm{br}}(\D_2)
/
\phi_0\Gamma_{A_1}\phi_0^{-1}.
\]
\end{corollary}
\begin{proof}
Every residual equivalence is of the form
$\gamma\phi_0$ with
$\gamma\in\Aut^{\mathrm{br}}(\D_2)$.
The isomorphism condition
\[
\phi'\circ\ind(\alpha_1)
\cong
\ind(\alpha_2)\circ\phi
\]
identifies $\gamma$ up to left multiplication by $\Gamma_{A_2}$
and right multiplication by
$\phi_0\Gamma_{A_1}\phi_0^{-1}$.
\end{proof}


\section{The Grothendieck ring of module categories}
\label{sect:The Grothendieck ring of module categories}

Let $\B$ be a braided fusion category. We denote by $K^\mmod(\B):=K_0(\ModCat(\B))$ the $\mathbb{Z}_+$-ring with basis $\mathbb{B}$ consisting of equivalence classes of indecomposable $\B$-module categories and product $[\M][\N]=[\M\bt_\B\N]$. Since $\B$-module categories are completely reducible, $[\M_i][\M_j]=\sum_k N_{ij}^k[\M_k]$ for $[\M_i],[\M_j]\in\mathbb{B}$.
We refer to $K^\mmod(\B)$ as the \emph{Grothendieck ring of $\B$-module categories} or, simply, the \emph{module Grothendieck ring of $\B$}. 

\begin{example}
$K^\mmod(\Rep(G))$ is the twisted Burnside ring of $G$; see \cite[Theorem~4.3]{GRR}.
\end{example}

\subsection{The product of residual data}
\label{subsect: The product of residual data}

Let $\B$ be a non-degenerate braided fusion category. The tensor product of $\B$-module categories translates to a product of residual data.
The product is obtained by convolving the corresponding Lagrangian algebras as in Subsection~\ref{subsect: Convolution of etale algebras}; Proposition~\ref{prop:convolution etale} identifies this convolution with the tensor product of module categories.

Let $(A_1,A_2,\phi)$ and $(A_1',A_2',\phi')$ be residual data corresponding to indecomposable $\B$-module categories $\M$ and $\M'$, respectively.
Their product can be written as a sum of residual data
\begin{equation}
\label{eqn: product of data}
(A_1,A_2,\phi)(A_1',A_2',\phi')=\sum_i(B_i,C_i,\phi_i).
\end{equation}

\begin{proposition}
\label{prop: etale monotonicity}
We have $A_1\subset B_i$ and $A_2'\subset C_i$ for every $i$ on the right-hand side of \eqref{eqn: product of data}.
\end{proposition}
\begin{proof}
Put $\mathcal P:=\M\bt_\B\M'\simeq\bigoplus_i\mathcal P_i$, where $\mathcal P_i$ corresponds to the residual data $(B_i,C_i,\phi_i)$. The endpoint inclusions are algebra monomorphisms
\[
A_1\bt\be\hookrightarrow L_\M,
\qquad
\be\bt A_2'\hookrightarrow L_{\M'}.
\]
Since the full center adjunction is monoidal, these maps are equivalent to unital algebra morphisms
\[
F_\M(A_1\bt\be)\longrightarrow\id_\M,
\qquad
F_{\M'}(\be\bt A_2')\longrightarrow\id_{\M'},
\]
or, in $\alpha$-induction notation, $\alpha_\M^+(A_1)\to\id_\M$ and $\alpha_{\M'}^-(A_2')\to\id_{\M'}$. 
Via the outer-action isomorphisms \eqref{eqn:alpha outer actions}, the two endpoint morphisms induce unital algebra morphisms
\[
\alpha_{\mathcal P}^{+}(A_1)\longrightarrow\id_{\mathcal P},
\qquad
\alpha_{\mathcal P}^{-}(A_2')\longrightarrow\id_{\mathcal P}.
\]
The \(\alpha\)-interchange isomorphism \eqref{eqn:alpha interchange} identifies the two possible orders of these actions. Together with \eqref{eqn:central action alpha}, it therefore combines the two morphisms above into a unital algebra morphism
\begin{equation}
\label{eqn:endpoint action on product}
F_{\mathcal P}(A_1\bt A_2')
\longrightarrow\id_{\mathcal P}.
\end{equation}
Thus for every \(i\), \eqref{eqn:endpoint action on product} restricts to a unital algebra morphism
$F_{\mathcal P_i}(A_1\boxtimes A_2')\longrightarrow \id_{\mathcal P_i}$.
Under the full center adjunction, it corresponds to a unital algebra morphism
\[
j_i:A_1\bt A_2'\longrightarrow L_{\mathcal P_i}.
\]
The latter is a monomorphism. Indeed, \(A_1\bt A_2'\) is connected étale and hence simple as a right module over itself. Since the kernel of a unital algebra morphism is a proper right ideal, it must therefore be zero.

Restricting $j_i$ along the unit maps of the two tensor factors gives subalgebras $A_1\bt\be\subset L_{\mathcal P_i}$ and $\be\bt A_2'\subset L_{\mathcal P_i}$. Since the residual endpoints of $L_{\mathcal P_i}$ are $B_i=L_{\mathcal P_i}\cap(\B\bt\be)$ and $C_i=L_{\mathcal P_i}\cap(\be\bt\B^\rev)$, we obtain $A_1\subset B_i$ and $A_2'\subset C_i$.
\end{proof}

\begin{remark}
\label{rem: product multiplicity}
The preceding discussion also gives a useful multiplicity formula. Let $s$ be the number of indecomposable summands of $\M\bt_\B\M'$, counted with multiplicity. 
Proposition~\ref{prop:convolution etale} and  the convolution formula \eqref{eqn:convolution} give
\begin{equation}
\label{eqn:product multiplicity}
\begin{aligned}
s
&=\dim_k\Hom_{\B\bt\B^\rev}(\be\bt\be,L_{\M\bt_\B\M'})\\
&=\dim_k\Hom_{\B\bt\B^\rev}(\be\bt\be,L_\M\odot L_{\M'})\\
&=\dim_k\Hom_\B\bigl(\be,A_2\ot(A_1')^*\bigr) \\
&=\dim_k\Hom_\B(A_1',A_2).
\end{aligned}
\end{equation}
\end{remark}

\subsection{Grading by the partially ordered set of localizations}
\label{subsect: Grading by the partially ordered set of localizations}

The poset $\mathfrak{Loc}(\B)$ (see Definition~\ref{def: Et and Loc}) is bounded: its least element is
$\B$ itself, and its greatest element is the completely anisotropic
representative of the Witt class of $\B$
\cite[Theorem~5.14]{DMNO}.

It follows from the definition of residual data and Proposition~\ref{prop: etale monotonicity} that $K^\mmod(\B)$ is graded by $\mathfrak{Loc}(\B)$:
\begin{equation}
\label{eqn:Kmod grading}
K^\mmod(\B) = \bigoplus_{\C \in \mathfrak{Loc}(\B)}\, K^\mmod(\B)_\C,\qquad K^\mmod(\B)_\C =\mathbb{Z}\langle (A_1,A_2,\phi)\mid \B_{A_1}^\loc \cong \B_{A_2}^\loc \cong \C \rangle,
\end{equation}
so that for all $x_1\in K^\mmod(\B)_{\C_1}$ and $x_2\in K^\mmod(\B)_{\C_2}$ we have
\[
x_1x_2 =\sum_{\D \geq \C_1, \C_2}\, y_\D,\qquad y_\D \in K^\mmod(\B)_\D.
\]
\begin{corollary}
\label{cor: max ideal}
Let $\B$ be a non-degenerate braided fusion category, and let $\M,\N$ be indecomposable $\B$-module categories. If an invertible $\B$-module category occurs as a summand of $\M\boxtimes_\B\N$, then both $\M$ and $\N$ are invertible. Consequently, the $\mathbb Z$-span of the noninvertible basis elements of $K^\mmod(\B)$ is a two-sided ideal; it is the unique largest proper ideal spanned by basis elements, and the quotient is isomorphic to $
\mathbb Z[\Aut^{\mathrm{br}}(\B)]$.
\end{corollary}
\begin{proof}
Proposition~\ref{prop: etale monotonicity} shows that a product of basis
elements can have an invertible summand only if both factors are invertible;
the remaining assertions follow since the invertible basis elements form the
group $\Aut^{\mathrm{br}}(\B)$.
\end{proof}

For $\B=\Rep(D(G))$, Corollary~\ref{cor: max ideal} recovers \cite[Proposition~3.1 and Corollary~3.2]{NikshychBurnside2026}.

Let $P$ be a finite poset, and let $R$ be an algebra over a commutative ground
ring.
We say that \(R\) is \emph{graded by \(P\)} (in the upper-triangular sense) if
\[
R=\bigoplus_{p\in P}R_p,
\qquad
R_pR_q\subseteq \bigoplus_{\substack{r\in P\\ r\geq p,\ r\geq q}}R_r
\quad (p,q\in P).
\]
For $p\in P$, put
\[
F_{\geq p}:=\bigoplus_{q\geq p}R_q,
\qquad
F_{>p}:=\bigoplus_{q>p}R_q,
\qquad
S_p(R):=F_{\geq p}/F_{>p}.
\]
The grading condition makes $F_{\geq p}$ and $F_{>p}$ two-sided ideals. We
call $S_p(R)$ the \emph{$p$-stratum}. Its underlying ground ring
module is $R_p$, with multiplication obtained by projecting a product to its
$p$-component. For a finite-dimensional algebra $R$ over a field, we write
$\operatorname{Jac}(R)$ for its Jacobson radical.

\begin{proposition}
\label{prop: semisimplicity poset grading}
Let $R$ be a finite-dimensional algebra graded by a finite poset $P$. Then
$R$ is semisimple if and only if every stratum $S_p(R)$ is
semisimple. When these conditions hold, there is a noncanonical algebra
isomorphism
\[
R\cong\bigoplus_{p\in P}S_p(R).
\]
\end{proposition}

\begin{proof}
Choose a linear extension $p_1,\ldots,p_r$ of the original partial order  such that $p_i<p_j$ implies
$i<j$, and put
\[
I_i:=\bigoplus_{j\geq i}R_{p_j},\qquad I_{r+1}=0.
\]
Each $I_i$ is a two-sided ideal and $I_i/I_{i+1}\cong S_{p_i}(R)$ as algebras. 

Suppose first that every $S_p(R)$ is semisimple, and let
$J=\operatorname{Jac}(R)$. Once $J\subseteq I_i$, its image in
$I_i/I_{i+1}$ is a nilpotent ideal of a semisimple algebra and is therefore
zero. Starting with $I_1=R$, induction gives $J\subseteq I_{r+1}=0$.

Conversely, if $R$ is semisimple, every $I_i$ and every quotient
$I_i/I_{i+1}$ are semisimple. Every ideal in a finite-dimensional semisimple
algebra has a complementary ideal, so successively splitting the chain
$R=I_1\supset I_2\supset\cdots\supset I_{r+1}=0$ gives the asserted
noncanonical direct sum decomposition.
\end{proof}

For the localization grading \eqref{eqn:Kmod grading}, define the
\emph{$\C$-stratum ring of $\B$} by
\begin{equation}
\label{eqn: stratum algebra}
S_\C(\B):=S_\C\bigl(K^\mmod(\B)\bigr).
\end{equation}
The stratum rings need not be unital. We write
\[
S_{\mathcal C}(\mathcal B)_{\mathbb C}
:=S_{\mathcal C}(\mathcal B)\otimes_{\mathbb Z}\mathbb C.
\]
By Proposition~\ref{prop: semisimplicity poset grading}, the complexified module Grothendieck ring is semisimple if and only if every complexified stratum algebra \(S_{\mathcal C}(\mathcal B)_{\mathbb C}\) is semisimple.

\subsection{The structure of stratum rings}
\label{subsect: The structure of stratum algebras}

Fix $\C\in\mathfrak{Loc}(\B)$. We describe multiplication in the stratum
ring $S_\C(\B)$.

Choose representatives $A_i$, $i\in I$, of the isomorphism classes of
connected \'etale algebras in the fiber over $\C$, together with braided
equivalences
\begin{equation}
\label{eqn: gauges}
\eta_i:\B_{A_i}^{\loc}\xrightarrow{\sim}\C.
\end{equation}
Set
\begin{equation}
\label{eqn: stratum groups}
G:=\Aut^{\mathrm{br}}(\C),
\qquad
H_i:=\eta_i\Gamma_{A_i}\eta_i^{-1}\leq G.
\end{equation}
For $g\in G$, write $x_{ij}(g)$ for the residual class represented by
\[
(A_i,A_j,\eta_j^{-1}g\eta_i).
\]
By Corollary~\ref{cor: residual double orbits}, this class depends only on the
double coset $H_jgH_i$. Put
\begin{equation}
\label{eqn: Sij}
S_{ij}:=\mathbb Z\langle x_{ij}(g)\mid g\in G\rangle
\cong\mathbb Z[H_j\backslash G/H_i],
\qquad
S_\C(\B)=\bigoplus_{i,j\in I}S_{ij}.
\end{equation}
We call residual data with distinct endpoints $A_i\ne A_j$ \emph{heterotic}; thus $S_{ij}$, $i\ne j$, are the heterotic endpoint blocks. We extend $g\mapsto x_{ij}(g)$ linearly to a map
\begin{equation}
\label{eqn: xij}
x_{ij}:\mathbb Z[G]\longrightarrow S_{ij}.
\end{equation}

Let two module categories in the \(\C\)-stratum have residual data with endpoints \((A_i,A_k)\) and \((A_\ell,A_j)\). The next proposition shows that every summand of their product retained in this stratum has endpoints \((A_i,A_j)\).

\begin{proposition}
\label{prop: stratum endpoint blocks}
For all $i,j,k,\ell\in I$, multiplication in the stratum ring satisfies
\[
S_{ik}S_{\ell j}\subseteq S_{ij}.
\]
In particular, every $S_{ij}$ is a (possibly nonunital) subring of
$S_\C(\B)$.
\end{proposition}

\begin{proof}
Consider a summand of a product whose factors have endpoints
$(A_i,A_k)$ and $(A_\ell,A_j)$.
Denote its  endpoints by $B,D$. Proposition~\ref{prop: etale monotonicity}
gives algebra embeddings $A_i\hookrightarrow B$ and $A_j\hookrightarrow D$.
If that summand remains in the $\C$-stratum,
$\B_B^{\loc}\simeq\B_{A_i}^{\loc}$ and
$\B_D^{\loc}\simeq\B_{A_j}^{\loc}$. Equation~\eqref{eqn:local-module-dimension} forces
$\FPdim(B)=\FPdim(A_i)$ and  $\FPdim(D)=\FPdim(A_j)$. Hence both embeddings above are isomorphisms, so the retained
summand has endpoints $(A_i,A_j)$ and belongs to $S_{ij}$. Taking $k=j$ and
$\ell=i$ proves the last assertion.
\end{proof}

The remaining problem is to determine which summands of
$\B_{A_k}\bt_\B\B_{A_\ell}$ survive in the $\C$-stratum and which residual equivalence each
contributes.

Define the \emph{transition category} from $k$ to $\ell$ ($k,\ell\in I$) by
\begin{equation}
\label{eqn: middle transition category}
\T_{k\ell}:=\B_{A_k}\bt_{\B}\B_{A_\ell}
\simeq\bigoplus_{\mathcal P\subset\T_{k\ell}}\, \mathcal{P},
\end{equation}
where $\B_{A_k}$  is regarded as a right $\B$-module
category using the convention of
\eqref{eqn:right module convention}, and the last
sum runs over indecomposable $\B$-module summands, counted with multiplicity.
Proposition~\ref{prop:convolution etale} gives a decomposition
of the corresponding convolution of Lagrangian algebras:
\[
L_{\B_{A_k}}\odot L_{\B_{A_\ell}}
\cong\bigoplus_{\mathcal P\subset\T_{k\ell}}\,L_{\mathcal P}.
\]

Fix a summand  $\mathcal P\subset\T_{k\ell}$ in \eqref{eqn: middle transition category} and let
$(B,D,\rho)$ be the residual data of $L_{\mathcal P}$.
Proposition~\ref{prop: etale monotonicity} gives algebra embeddings
$A_k\hookrightarrow B$ and $A_\ell\hookrightarrow D$.
For $s\in I$, put $\C_s:=\B_{A_s}^{\loc}$.
By \cite[Proposition~3.16]{DMNO}, the algebra extensions $A_k\subset B$ and $A_\ell\subset D$ correspond to
connected \'etale algebras $\overline B\in\C_k$ and
$\overline D\in\C_\ell$. Transitivity of  local modules gives a
braided equivalence
\[
\overline\rho:
(\C_k)_{\overline B}^{\loc}
\xrightarrow{\sim}
(\C_\ell)_{\overline D}^{\loc}.
\]
Set
\begin{equation}
\label{eqn: localized full center}
\widetilde L_{\mathcal P}
:=\operatorname{Loc}_{A_k\bt A_\ell}(L_{\mathcal P})
\in\C_k\bt\C_\ell^{\rev}.
\end{equation}
Here $\operatorname{Loc}$ denotes localization by the displayed endpoint
algebra. Put
\[
\overline L_{\mathcal P}:=
(\eta_k\bt\eta_\ell^{\rev})(\widetilde L_{\mathcal P})
\in\C\bt\C^{\rev}.
\]

\begin{proposition}
\label{prop: local transition components}
The algebra $\overline L_{\mathcal P}$ is Lagrangian in
$\C\bt\C^{\rev}\simeq\Z(\C)$, with the residual data
\[
\bigl(
\eta_k(\overline B),
\eta_\ell(\overline D),
\eta_\ell\overline\rho\,\eta_k^{-1}
\bigr).
\]
\end{proposition}

\begin{proof}
Since $L_{\mathcal P}$ is Lagrangian,
\[
(\B\bt\B^{\rev})_{L_{\mathcal P}}^{\loc}\simeq\Vect.
\]
Transitivity of localization gives
\[
\Bigl((\B\bt\B^{\rev})_{A_k\bt A_\ell}^{\loc}\Bigr)_{
\widetilde L_{\mathcal P}}^{\loc}
\simeq
(\B\bt\B^{\rev})_{L_{\mathcal P}}^{\loc}
\simeq\Vect.
\]
As $(\B\bt\B^{\rev})_{A_k\bt A_\ell}^{\loc}
\simeq\C_k\bt\C_\ell^{\rev}$,
the algebra $\widetilde L_{\mathcal P}$ is Lagrangian. The braided
equivalence $\eta_k\bt\eta_\ell^{\rev}$ then shows that
$\overline L_{\mathcal P}$ is Lagrangian as well.

The inclusion $B\bt D\hookrightarrow L_{\mathcal P}$ localizes to
\[
\overline B\bt\overline D\hookrightarrow\widetilde L_{\mathcal P}.
\]
Thus the left and right residual endpoints of $\widetilde L_{\mathcal P}$
contain $\overline B$ and $\overline D$, respectively. We show that there is
no further enlargement. Suppose $E$ is the left residual endpoint, so that
$\overline B\subseteq E$ and
$E\bt\be\subseteq\widetilde L_{\mathcal P}$. Let $B_E\supset A_k$ be the connected
\'etale algebra corresponding to $E$ under
\cite[Proposition~3.16]{DMNO}. The inclusion $\overline B\subseteq E$ lifts
to $B\subseteq B_E$, while $E\bt\be\subseteq\widetilde L_{\mathcal P}$
lifts to an algebra map
\[
B_E\bt A_\ell\longrightarrow L_{\mathcal P}.
\]
By maximality of the original
left residual endpoint $B$, this implies  $B_E=B$
and $E=\overline B$. The same argument on the right proves that the right
residual endpoint is $\overline D$.

Localizing $\widetilde L_{\mathcal P}$ further by
$\overline B\bt\overline D$ is, by transitivity, the same as localizing
$L_{\mathcal P}$ directly by $B\bt D$. The latter is the graph algebra of
\[
\rho:\B_B^{\loc}\xrightarrow{\sim}\B_D^{\loc}.
\]
Under the transitivity identifications
\[
\B_B^{\loc}\simeq(\C_k)_{\overline B}^{\loc},
\qquad
\B_D^{\loc}\simeq(\C_\ell)_{\overline D}^{\loc},
\]
this is the graph algebra of $\overline\rho$. Thus the residual data of
$\widetilde L_{\mathcal P}$ are
$(\overline B,\overline D,\overline\rho)$, and applying the gauges gives the
residual data stated for $\overline L_{\mathcal P}$. 
\end{proof}

\begin{definition}
\label{def: local transition component}
Let $\mathcal X_{\mathcal P}$ be the indecomposable $\C$-module category,
unique up to equivalence, whose full center is $\overline L_{\mathcal P}$.
We call $\mathcal X_{\mathcal P}$ the \emph{local transition component}
of $\mathcal P$.
\end{definition}

The following criterion concerns invertibility of $\mathcal X_{\mathcal P}$ as a one-sided $\C$-module category.

\begin{proposition}
\label{prop: local transition criterion}
For a summand $\mathcal P\subset\T_{k\ell}$, let $(B,D,\rho)$ be the
residual data of $L_{\mathcal P}$. The following conditions are equivalent:
\begin{enumerate}
\item $\mathcal P$ contributes to the $\C$-stratum;
\item $\overline B=\be=\overline D$;
\item $\mathcal X_{\mathcal P}$ is invertible as a $\C$-module category.
\end{enumerate}
Under these conditions, put
\[
\psi_{\mathcal P}:=\overline\rho:
\C_k\xrightarrow{\sim}\C_\ell.
\]
Its braided tensor isomorphism class is canonically determined by
$\mathcal P$, and
\begin{equation}
\label{eqn: graph algebra}
\widetilde L_{\mathcal P}\cong\Delta_{\psi_{\mathcal P}},
\qquad
\overline L_{\mathcal P}\cong\Delta_{\phi_{\mathcal P}},
\quad
\phi_{\mathcal P}:=\eta_\ell\psi_{\mathcal P}\eta_k^{-1}\in\Aut^{\mathrm{br}}(\C).
\end{equation}
\end{proposition}
\begin{proof}
For $(i)\Rightarrow(ii)$, suppose that $\mathcal P$ contributes to the
$\C$-stratum. Transitivity of localization gives
\[
\B_B^{\loc}\simeq(\C_k)_{\overline B}^{\loc}\simeq\C.
\]
Since $\C_k\simeq\C$, \eqref{eqn:local-module-dimension}
gives $\FPdim(\overline B)=1$ and so $\overline B=\be$.
The same argument on the right gives
$\overline D=\be$.

For $(ii)\Rightarrow(i)$, if $\overline B=\be=\overline D$, the corresponding localizations are identified with $\C_k$ and $\C_\ell$, respectively,
and therefore with $\C$. Thus $\mathcal P$ remains in the $\C$-stratum.

For $(ii)\Rightarrow(iii)$, the full center $\overline L_{\mathcal P}$ has
trivial residual endpoints $\eta_k(\overline B)$ and
$\eta_\ell(\overline D)$;  therefore $\mathcal X_{\mathcal P}$ is invertible. For
$(iii)\Rightarrow(ii)$,  the residual
endpoints of the full center of an invertible $\C$-module category are
trivial. Hence $\eta_k(\overline B)=\be=\eta_\ell(\overline D)$, and so
$\overline B=\be=\overline D$. 

Under these equivalent conditions, put
$\psi_{\mathcal P}:=\overline\rho$. The residual-data description gives
$\widetilde L_{\mathcal P}\cong\Delta_{\psi_{\mathcal P}}$, and applying
the gauges gives
$\overline L_{\mathcal P}\cong\Delta_{\phi_{\mathcal P}}$.
\end{proof}

With respect to the chosen equivalences \eqref{eqn: gauges}, we call
$\phi_{\mathcal P}\in G$ from \eqref{eqn: graph algebra} the \emph{Picard label} of the invertible
local transition component $\mathcal X_{\mathcal P}$. Using these labels, define
\begin{equation}
\label{eqn: transition element}
\Theta_{k\ell}:=
\sum_{\substack{\mathcal P\subset\T_{k\ell}\\
\mathcal P\text{ contributes to the }\C\text{-stratum}}}
\phi_{\mathcal P}
\in\mathbb Z_{\geq0}[G],
\end{equation}
where the sum is over summands, counted with multiplicity. 

\begin{definition}
\label{def: transition elements and matrix}
The elements $\Theta_{k\ell}$, $k,\ell\in I$, from \eqref{eqn: transition element} will be called {\em transition elements}.
The family
$(\Theta_{k\ell})_{k,\ell\in I} \in \operatorname{Mat}_{|I|}(\mathbb Z_{\geq0}[G])$ will be called the {\em transition matrix} of the $\C$-stratum.
\end{definition}

The transition elements \eqref{eqn: transition element} are bi-invariant in the following sense. Choose
$\alpha_s\in\Aut_{\alg}(A_s)$ representing $a_s\in H_s$ for $s=k,\ell$.
Theorem~\ref{thm: automorphisms etale algebra} identifies $\alpha_s$ with a
homogeneous invertible component of the commutator subcategory
$\E_{A_s}\subseteq\B_{A_s}$, and its action on the neutral component
$\B_{A_s}^{\loc}$ with $\ind_{A_s}(\alpha_s)$. The target and source actions by these components
permute the surviving summands and preserve their multiplicities.
If $\mathcal P$ is sent to $\mathcal P'$, then
\[
\psi_{\mathcal P'}=
\ind_{A_\ell}(\alpha_\ell)\psi_{\mathcal P}
\ind_{A_k}(\alpha_k)^{-1},
\]
so $\phi_{\mathcal P'}=a_\ell\phi_{\mathcal P}a_k^{-1}$. Summing over the
permuted summands gives
\begin{equation}
\label{eqn: transition bi-invariance}
a_\ell\Theta_{k\ell}a_k^{-1}=\Theta_{k\ell}
\qquad(a_\ell\in H_\ell,\ a_k\in H_k).
\end{equation}

In addition, transition elements satisfy the following gauge covariance. Replacing
$\eta_s$ from \eqref{eqn: gauges} by $t_s\eta_s$, with $t_s\in G$, gives
\begin{equation}
\label{eqn: transition gauge change}
H_s\longmapsto t_sH_st_s^{-1},
\qquad
\Theta_{k\ell}\longmapsto t_\ell\Theta_{k\ell}t_k^{-1}.
\end{equation}
Thus the gauge class of $\bigl((H_i),(\Theta_{k\ell})\bigr)$ is intrinsic to
the localization fiber. Replacing the representatives $A_i$ by isomorphic
algebras gives the same gauge class, up to the evident simultaneous
reindexing.

For residual data $(A_r,A_s,\gamma)$, with
$\gamma:\C_r\xrightarrow{\sim}\C_s$, write
$L(A_r,A_s,\gamma)$ for the corresponding Lagrangian algebra.

\begin{lemma}
\label{lem:outer transport}
Let $i,j,k,\ell\in I$, and let $
\gamma:\C_i\xrightarrow{\sim}\C_k,
\,
\delta:\C_\ell\xrightarrow{\sim}\C_j
$
be braided tensor equivalences. There is an isomorphism of algebras in
$\C_i\bt\C_j^{\rev}$,
\begin{equation}
\label{eqn: outer transport}
\operatorname{Loc}_{A_i\bt A_j}
\bigl(
L(A_i,A_k,\gamma)\odot L(A_\ell,A_j,\delta)
\bigr)
\cong
(\gamma^{-1}\bt\delta^{\rev})
\operatorname{Loc}_{A_k\bt A_\ell}
\bigl(
L(A_k,A_k,\id)\odot L(A_\ell,A_\ell,\id)
\bigr).
\end{equation}
\end{lemma}

\begin{proof}
Let $U_s:\C_s\to\B$ be the underlying-object functor, with its canonical
lax monoidal structure given by the quotient maps
$
U_s(X)\ot U_s(Y)\longrightarrow U_s(X\ot_{A_s}Y)
$
and the unit $\be\to A_s$.

We use $\operatorname{ctr}_{23}$ for the middle contraction with any pair
of outer local categories. With outer categories $\C_i$ and
$\C_j^{\rev}$, it is given by
\[
X\bt V\bt W\bt Y
\longmapsto
\Hom_\B(\be,V\ot W)\otimes_k(X\bt Y),
\]
where $X\in\C_i$, $V,W\in\B$, and $Y\in\C_j$. Its lax monoidal
structure is given by the middle crossing \eqref{eqn:middle crossing}.

By the residual-data description,
\[
\operatorname{Loc}_{A_i\bt A_k}
\bigl(L(A_i,A_k,\gamma)\bigr)
\cong\Delta_\gamma,
\qquad
\operatorname{Loc}_{A_\ell\bt A_j}
\bigl(L(A_\ell,A_j,\delta)\bigr)
\cong\Delta_\delta.
\]
The left-hand side of \eqref{eqn: outer transport} is therefore the algebra
\begin{equation}
\label{eqn: ctr23 algebra}
\operatorname{ctr}_{23}\!\left(
(\id_{\C_i}\bt U_k^{\rev})(\Delta_\gamma)
\bt
(U_\ell\bt\id_{\C_j^{\rev}})(\Delta_\delta)
\right).
\end{equation}
Indeed, the local-module structures on the two outer factors are retained,
while the middle ones are forgotten via $U_k$ and $U_\ell$ before applying
the middle contraction. The quotient maps defining the tensor products in
the outer local categories commute with the middle crossing
\eqref{eqn:middle crossing}, so this identification respects multiplication
and units.

The braided tensor equivalences $\gamma^{-1}$ and $\delta$ give algebra
isomorphisms
\[
(\gamma^{-1}\bt\id_{\C_k^{\rev}})
(\Delta_{\id_{\C_k}})
\cong\Delta_\gamma,
\qquad
(\id_{\C_\ell}\bt\delta^{\rev})
(\Delta_{\id_{\C_\ell}})
\cong\Delta_\delta.
\]
Since these equivalences act only on the outer factors, they commute with
the middle contraction and its lax monoidal structure. Hence the 
algebra  \eqref{eqn: ctr23 algebra} is isomorphic to
\begin{equation}
\label{eqn: ctr23 algebra+}
(\gamma^{-1}\bt\delta^{\rev})
\operatorname{ctr}_{23}\!\left(
(\id_{\C_k}\bt U_k^{\rev})(\Delta_{\id_{\C_k}})
\bt
(U_\ell\bt\id_{\C_\ell^{\rev}})(\Delta_{\id_{\C_\ell}})
\right).
\end{equation}
Applying \eqref{eqn: ctr23 algebra} with $i=k$, $j=\ell$ and residual
equivalences $\id_{\C_k},\id_{\C_\ell}$, the contraction in
\eqref{eqn: ctr23 algebra+} is
\[
\operatorname{Loc}_{A_k\bt A_\ell}
\bigl(
L(A_k,A_k,\id)\odot L(A_\ell,A_\ell,\id)
\bigr).
\]
Thus the full expression \eqref{eqn: ctr23 algebra+} is the right-hand
side of \eqref{eqn: outer transport}.
\end{proof}

\begin{remark}
With the notation of Lemma~\ref{lem:outer transport}  and its proof, expanding the graph algebras
gives the underlying object of the left-hand side of
\eqref{eqn: outer transport}:
\begin{equation}
\label{eqn: mixed contraction coefficients}
\operatorname{Loc}_{A_i\bt A_j}
\bigl(
L(A_i,A_k,\gamma)\odot L(A_\ell,A_j,\delta)
\bigr)
\cong
\bigoplus_{\substack{Z\in\O(\C_k)\\Y\in\O(\C_\ell)}}
\Hom_\B\bigl(U_\ell(Y),U_k(Z)\bigr)
\otimes_k
\bigl(\gamma^{-1}(Z^*)\bt\delta(Y)\bigr).
\end{equation}
\end{remark}

In the following theorem, we use the $\mathbb Z$-linear extension
$x_{ij}:\mathbb Z[G]\to S_{ij}$ from \eqref{eqn: xij}.

\begin{theorem}
\label{thm: stratum multiplication}
For all $i,j,k,\ell\in I$ and $g,h\in G$,
\begin{equation}
\label{eqn: stratum multiplication}
x_{ik}(g)x_{\ell j}(h)=x_{ij}(h\Theta_{k\ell}g).
\end{equation}
The right-hand side depends only on $H_kgH_i$ and $H_jhH_\ell$, and hence
the formula is well defined.
\end{theorem}

\begin{proof}
Put
\[
\gamma_g:=\eta_k^{-1}g\eta_i:\C_i\longrightarrow\C_k,
\qquad
\gamma_h:=\eta_j^{-1}h\eta_\ell:\C_\ell\longrightarrow\C_j.
\]
Apply Lemma~\ref{lem:outer transport} with
$\gamma=\gamma_g$ and $\delta=\gamma_h$.
Let $\M,\M'$ be the indecomposable $\B$-module categories representing
$x_{ik}(g)$ and $x_{\ell j}(h)$. By
Proposition~\ref{prop:convolution etale}, the left-hand side of
\eqref{eqn: outer transport} is the sum of the full centers of the
indecomposable summands of $\M\bt_\B\M'$, localized at $A_i\bt A_j$.
Its graph-algebra summands represent precisely the terms retained in
$x_{ik}(g)x_{\ell j}(h)$.
On the right, before transport, Proposition~\ref{prop: local transition criterion}
identifies the graph summands with those defining $\Theta_{k\ell}$.
The equivalence $\gamma_g^{-1}\bt\gamma_h^{\rev}$ sends each
$\Delta_{\psi_{\mathcal P}}$ to
$\Delta_{\gamma_h\psi_{\mathcal P}\gamma_g}$, whose Picard label is
$h\phi_{\mathcal P}g$. Since this equivalence preserves graph summands
and their multiplicities, retaining graph summands on both sides of
\eqref{eqn: outer transport} gives
$x_{ik}(g)x_{\ell j}(h)=x_{ij}(h\Theta_{k\ell}g)$, as required.

Finally, let $g'=a_k^{-1}ga_i$ and $h'=a_jha_\ell$, with $a_s\in H_s$.
By \eqref{eqn: transition bi-invariance},
$h'\Theta_{k\ell}g'=a_jh\Theta_{k\ell}ga_i$, so
$x_{ij}(h'\Theta_{k\ell}g')=x_{ij}(h\Theta_{k\ell}g)$.
Thus the right-hand side of \eqref{eqn: stratum multiplication}
depends only on the input double cosets.
\end{proof}

\begin{proposition}
\label{prop:diagonal component correspondence}
Let $A$ be a connected \'etale algebra in $\B$. The indecomposable
$\B$-module components of $\B_A\bt_{\B}\B_A$
are in a canonical multiplicity-preserving bijection with the indecomposable
$\B_A^{\loc}$-module components of $\B_A$.
\end{proposition}
\begin{proof}
Put 
\[
\mathcal D=\B^*_{\B_A},\,
\mathcal D^\pm=\alpha_{\B_A}^\pm(\B).
\]
Using \cite[Proposition~3.5]{ENO2} and the equivalence
$(\B_A)^{\op}\simeq\B_A$ coming from commutativity of $A$, we obtain
a $\B$-module equivalence
\[
\B_A\bt_{\B}\B_A\simeq\mathcal D.
\]
Under this equivalence, the indecomposable $\B$-module summands 
of $\B_A\bt_{\B}\B_A$ correspond to the
$\mathcal D^+$-components of $\mathcal D$. Under the standard
realization of $\mathcal D^-$ by $A$-modules, the intersection
$\mathcal D^+\cap\mathcal D^-$ identifies with $\B_A^{\loc}$,
and its action on $\mathcal D^-\simeq\B_A$ is given by the tensor product
$\ot_A$ \cite[Theorem~5.20]{FFRS}.

The subcategories $\mathcal D^+$ and $\mathcal D^-$ generate
$\mathcal D$, and the $\alpha$-interchange law
\eqref{eqn:alpha interchange} allows their factors to be reordered.
Hence every simple object of $\mathcal D$ is a summand of
$W\ot M$ for some $W\in\mathcal D^+$ and simple
$M\in\mathcal D^-$. Thus every $\mathcal D^+$-component of
$\mathcal D$ meets $\mathcal D^-$.

If simple objects $M,N\in\mathcal D^-$ belong to the same
$\mathcal D^+$-component, then
\[
\Hom_{\mathcal D}(W\ot M,N)\ne0
\]
for some simple $W\in\mathcal D^+$. By rigidity, $W$ is a subobject
of $N\ot M^*$, so $W\in\mathcal D^-\cap\mathcal D^+$. Hence
$M$ and $N$ already belong to the same $\B_A^{\loc}$-module
component. The converse is immediate since
$\mathcal D^+\cap\mathcal D^-\subset\mathcal D^+$.
Thus intersection with $\mathcal D^-$ gives the required bijection.
\end{proof}

For an indecomposable $\B$-module component
$\mathcal P\subset\B_A\bt_{\B}\B_A$, regard $\mathcal P$ as the
corresponding $\mathcal D^+$-component of $\mathcal D$ under the
equivalence used in the proof of
Proposition~\ref{prop:diagonal component correspondence}, and denote by
\[
\mathcal P^-:=\mathcal P\cap\mathcal D^-
\subset\mathcal D^-\simeq\B_A
\]
the corresponding $\B_A^{\loc}$-module component of $\B_A$.

Let $L_{\mathcal P}\in\Z(\B)\simeq\B\boxtimes\B^{\rev}$ be the full
center of $\mathcal P$, and let
\[
\widetilde L_{\mathcal P}
:=\operatorname{Loc}_{A\boxtimes A}(L_{\mathcal P})
\in \B_A^{\loc}\boxtimes (\B_A^{\loc})^{\rev}
\]
be its localization.

\begin{proposition}
\label{prop:diagonal local transition component}
There is an isomorphism of algebras
$
\widetilde L_{\mathcal P}\cong L_{\mathcal P^-}.
$
Consequently, the local transition component of $\mathcal P$ is
equivalent to $\mathcal P^-$.
\end{proposition}
\begin{proof}
Choose a simple object $X\in\mathcal P^-$. We claim that the internal  endomorphism algebra 
\[
Q:=\underline{\End}_{\B_A^{\loc}}(X)
\]
is isomorphic to \(R=\underline{\End}_{\mathcal D^+}(X)\).
Indeed, by rigidity, \(R\) is the maximal subobject of \(X\otimes X^*\) lying in \(\mathcal D^+\). Since \(X\in\mathcal D^-\), we have \(X\otimes X^*\in\mathcal D^-\) and hence \(R\in\mathcal D^+\cap\mathcal D^-=\B_A^{\loc}\). 
For every $V\in\B_A^{\loc}$ we have natural identifications
\[
\Hom_{\B_A^{\loc}}(V,R)
\cong \Hom_{\mathcal P}(V\ot X,X)
= \Hom_{\mathcal P^-}(V\ot X,X)
\cong \Hom_{\B_A^{\loc}}(V,Q).
\]
Thus $R$ and $Q$ represent the same functor on $\B_A^{\loc}$.
The resulting isomorphism $R\cong Q$ respects multiplication and unit,
since the algebra structures on both internal endomorphism objects are
induced by composition.

By Proposition~\ref{prop:diagonal component correspondence},
$\mathcal P^-$ is the $\B_A^{\loc}$-component of $\mathcal P$
containing $X$. Hence
\[
\mathcal P^-\simeq\Mod_{\B_A^{\loc}}Q,
\qquad
\mathcal P\simeq\Mod_{\mathcal D^+}R.
\]
Since $R\cong Q$, the algebra-over-$A$ correspondence
\cite[Subsection~3.6]{DMNO} associates to $Q$ an algebra
$\widehat Q$ in $\B$ such that
\[
\mathcal P\simeq\Mod_{\B}\widehat Q.
\]
By \cite{DavCenter}, $L_{\mathcal P^-}$ is the full
center of $Q$ in $\Z(\B_A^{\loc})$, while $L_{\mathcal P}$ is the
full center of $\widehat Q$ in $\Z(\B)$.

We use the universal characterization of the full center
\cite[Section~4]{DavCenter}: the full center of an algebra $S$ in a
monoidal category $\E$ is terminal among pairs consisting of an object
of $\Z(\E)$ and a central morphism to $S$.
The algebra map $A\to\widehat Q$ has the two central lifts
\[
A\boxtimes\be\longrightarrow L_{\mathcal P},
\qquad
\be\boxtimes A\longrightarrow L_{\mathcal P},
\]
which are the endpoint maps used in the localization
$\widetilde L_{\mathcal P}$.
Under the above algebra correspondence, central morphisms to
$\widehat Q$ which are compatible with these two inclusions of $A$
are precisely central morphisms to $Q$ in $\B_A^{\loc}$.
Compatibility with one endpoint $A$-action allows descent through the
relative tensor product, compatibility with the other is $A$-linearity
of the descended morphism, and locality makes the half-braiding descend,
identifying centrality in $\B$ with centrality in $\B_A^{\loc}$.
Hence the universal central morphism defining $L_{\mathcal P}$
descends after localization to a central morphism
\[
\widetilde L_{\mathcal P}\longrightarrow Q
\]
which is terminal among central morphisms to $Q$. Thus
$\widetilde L_{\mathcal P}$ has the universal property of the full
center of $Q$. Therefore $\widetilde L_{\mathcal P}\cong L_{\mathcal P^-}$
as algebras in
$\B_A^{\loc}\boxtimes(\B_A^{\loc})^{\rev}$.
\end{proof}

Recall from \eqref{eqn: middle transition category} and
\eqref{eqn: transition element}, respectively, that
\[
\T_{k\ell}=\B_{A_k}\bt_{\B}\B_{A_\ell}
\]
is the transition category from $k$ to $\ell$, and that
$\Theta_{k\ell}\in\mathbb Z_{\geq0}[G]$ is its transition element;
see also Definition~\ref{def: transition elements and matrix}.

Let $\Theta_{A_i}\in\mathbb Z_{\geq0}[G]$ denote the sum of the Picard
labels of the invertible $\C_i$-module components of $\B_{A_i}$,
counted with multiplicity and transported to $G$ through the chosen
equivalence $\eta_i$ from \eqref{eqn: gauges}.

\begin{corollary}\label{cor:diagonal transition element}
We have
\begin{equation}
\label{eqn: theta diagonal convention}
\Theta_{ii}=\Theta_{A_i}.
\end{equation}
\end{corollary}

\begin{proof}
Proposition~\ref{prop:diagonal component correspondence} gives a
multiplicity-preserving bijection between the indecomposable summands
of $\T_{ii}$ and the $\C_i$-module components of $\B_{A_i}$, while
Proposition~\ref{prop:diagonal local transition component} identifies
their local transition components. The formula follows from the
definition of $\Theta_{ii}$.
\end{proof}

\begin{corollary}\label{cor:transition component count}
For all $k,\ell\in I$,
\begin{equation}
\label{eqn: transition component count}
\#\{\text{indecomposable summands of }\T_{k\ell}\}
=
\dim_k\Hom_{\B}(A_\ell,A_k),
\end{equation}
with summands counted with multiplicity. For $k=\ell$, this is also
the number of $\C_k$-module components of $\B_{A_k}$.
\end{corollary}

\begin{proof}
The first statement is \eqref{eqn:product multiplicity} applied to
$\T_{k\ell}$. The diagonal statement follows from
Proposition~\ref{prop:diagonal component correspondence}.
\end{proof}

For $s=k,\ell$, let $U_s:\C_s\to\B$ be the underlying-object functor. Thus
$U_s(\be_{\C_s})=A_s$, and the chosen equivalences $\eta_s$ identify
both local categories with $\C$. Denote by
$\varepsilon:\mathbb Z[G]\to\mathbb Z$ the augmentation,
$\varepsilon(\sum_g n_gg)=\sum_g n_g$.

\begin{proposition}
\label{prop: all invertible transition components}
The following conditions are equivalent:
\begin{enumerate}[label=\textup{(\arabic*)}]
\item For every indecomposable summand
$\mathcal P\subset\T_{k\ell}$, the local transition component
$\mathcal X_{\mathcal P}$ of
Definition~\ref{def: local transition component} is invertible as a
one-sided $\C$-module category.
\item $\Hom_{\B}(U_\ell(Y),A_k)=0$ for every
$Y\in\O(\C_\ell)$ with $Y\not\cong\be_{\C_\ell}$.
\item $\Hom_{\B}(U_k(X),A_\ell)=0$ for every
$X\in\O(\C_k)$ with $X\not\cong\be_{\C_k}$.
\item $\varepsilon(\Theta_{k\ell})=
\dim_k\Hom_{\B}(A_\ell,A_k)$.
\end{enumerate}
\end{proposition}

\begin{proof}
For each summand $\mathcal P\subset\T_{k\ell}$, let
$\widetilde L_{\mathcal P}\in\C_k\bt\C_\ell^{\rev}$ be its localized
full center, with connected \'etale residual endpoints
$\overline B\in\C_k$ and $\overline D\in\C_\ell$, as in
Proposition~\ref{prop: local transition components}. For
$Y\in\O(\C_\ell)$, equation  \eqref{eqn: mixed contraction coefficients} with
$U_k(\be_{\C_k})=A_k$, gives
\[
\sum_{\mathcal P\subset\T_{k\ell}}
\dim_k\Hom_{\C_k\bt\C_\ell^{\rev}}
\bigl(\be_{\C_k}\bt Y,\widetilde L_{\mathcal P}\bigr)
=\dim_k\Hom_{\B}(U_\ell(Y),A_k).
\]
The sum counts all summands with multiplicity. By the residual data
description in
Subsection~\ref{subsect: A parameterization of Lagrangian algebras by residual data},
the part of $\widetilde L_{\mathcal P}$ supported in
$\be_{\C_k}\bt\C_\ell^{\rev}$ is $\be_{\C_k}\bt\overline D$.
Nonnegative multiplicities and connectedness identify condition~(2) with
all right residual endpoints $\overline D$ being the tensor unit.
The residual equivalence
$(\C_k)_{\overline B}^{\loc}\simeq(\C_\ell)_{\overline D}^{\loc}$ and
$\C_k\simeq\C_\ell$ give
$\FPdim(\overline B)=\FPdim(\overline D)$ by
\eqref{eqn:local-module-dimension}. Thus either endpoint is the tensor
unit exactly when both are, and
Proposition~\ref{prop: local transition criterion} proves
$(1)\Leftrightarrow(2)$. The same argument with the two endpoints
interchanged proves $(1)\Leftrightarrow(3)$.

By \eqref{eqn: transition element}, augmentation counts surviving summands
with multiplicity, retaining repeated labels, while
\eqref{eqn: transition component count} counts all summands with
multiplicity. Equality holds exactly when every summand survives, proving
$(1)\Leftrightarrow(4)$.
\end{proof}

\begin{lemma}
\label{lem: transition duality}
For every $k,\ell\in I$ and the fixed gauges,
\begin{equation}
\label{eqn: transition duality}
\Theta_{\ell k}=\Theta_{k\ell}^{\vee},
\end{equation}
where $(\sum_g n_g g)^\vee=\sum_g n_g g^{-1}$.
\end{lemma}

\begin{proof}
Duality identifies $\T_{k\ell}^{\vee}\simeq\T_{\ell k}$ and sends
$\mathcal P$ to $\mathcal P^{\vee}$, preserving multiplicities. It reverses
the localized residual data, taking the residual equivalence to its
inverse. Thus surviving summands correspond, and their Picard labels
satisfy $\phi_{\mathcal P^{\vee}}=
\eta_k\psi_{\mathcal P}^{-1}\eta_\ell^{-1}=\phi_{\mathcal P}^{-1}$.
Summation proves the identity. In particular,
$\varepsilon(\Theta_{k\ell})=\varepsilon(\Theta_{\ell k})$.
\end{proof}

\begin{corollary}
\label{cor: scalar sandwich strata}
Suppose $G=1$. Put $e_{ij}=x_{ij}(1)$ and
$m_{k\ell}=\varepsilon(\Theta_{k\ell})$. The stratum ring has basis $e_{ij}$
with multiplication
\begin{equation}
\label{eqn: trivial group sandwich}
e_{ik}e_{\ell j}=m_{k\ell}e_{ij}.
\end{equation}
Under $e_{ij}\leftrightarrow E_{ij}$, it is the integral matrix sandwich
ring $X*Y=XM_\C Y$, where $M_\C=(m_{k\ell})$ is symmetric.
If $\C$ is maximal, then $m_{k\ell}=\dim_k\Hom_\B(A_\ell,A_k)$.
\end{corollary}
\begin{proof}
The multiplication follows from
Theorem~\ref{thm: stratum multiplication}, and symmetry follows from
Lemma~\ref{lem: transition duality}. 
If $\C$ is maximal, then every local transition component is
invertible, so Proposition~\ref{prop: all invertible transition components}
gives
\[
m_{k\ell}
=\varepsilon(\Theta_{k\ell})
=\dim_k\Hom_\B(A_\ell,A_k).\qedhere
\]
\end{proof}

\begin{lemma}
\label{lem: sandwich radical}
Let $H\in\operatorname{Mat}_m(\mathbb C)$ have rank $r$, and let $R_H$ be
the vector space $\operatorname{Mat}_m(\mathbb C)$ with multiplication
$X*Y=XHY$. Then
\[
\operatorname{Jac}(R_H)
=\{X\in\operatorname{Mat}_m(\mathbb C)\mid HXH=0\},
\]
\[
R_H/\operatorname{Jac}(R_H)\cong\operatorname{Mat}_r(\mathbb C),
\qquad
\dim_{\mathbb C}\operatorname{Jac}(R_H)=m^2-r^2,
\qquad
\operatorname{Jac}(R_H)^3=0.
\]
In particular, $R_H$ is semisimple if and only if $H$ is nonsingular.
\end{lemma}

\begin{proof}
Invertible row and column operations reduce $H$ to
$H_0=\operatorname{diag}(I_r,0)$. If $H_0=PHQ$, then
$X\mapsto Q^{-1}XP^{-1}$ is an algebra isomorphism $R_H\to R_{H_0}$, so it
suffices to consider $H_0$. Projection to the upper-left $r\times r$ block is
then a surjective algebra homomorphism
$R_{H_0}\to\operatorname{Mat}_r(\mathbb C)$. Its kernel is
$\{X\mid H_0XH_0=0\}$, has dimension $m^2-r^2$, and the product of any three
of its elements is zero. Since the quotient is semisimple and the kernel is
nilpotent, the kernel is the Jacobson radical.
\end{proof}

\begin{corollary}
\label{cor: maximal Gram sandwich}
Suppose $\B=\Z(\A)$ and $\C=\Vect$, and let $L_1,\ldots,L_m$
represent all isomorphism classes of Lagrangian algebras in $\B$.
The maximal stratum ring has basis $x_{ij}=x_{L_i,L_j}$ and multiplication
\begin{equation}
\label{eqn: maximal Gram sandwich}
x_{ij}x_{k\ell}=H_{jk}x_{i\ell},
\qquad H_{jk}=\dim_k\Hom_\B(L_j,L_k).
\end{equation}
Thus the integral stratum ring is the matrix sandwich ring with Gram matrix
$H$, and
\[
\ker H=\left\{(\lambda_i)\in\mathbb C^m\ \middle|\
\sum_i\lambda_i[L_i]=0
\text{ in }K_0(\B)\ot_{\mathbb Z}\mathbb C\right\}.
\]
The complexified maximal stratum algebra is semisimple if and only if the underlying
classes $[L_i]$ are linearly independent. Its radical and semisimple
quotient are given by Lemma~\ref{lem: sandwich radical}, with
$r=\dim_\mathbb C\langle[L_1],\ldots,[L_m]\rangle$.
\end{corollary}

\begin{proof}
Apply Corollary~\ref{cor: scalar sandwich strata} with
$G=\Aut^{\mathrm{br}}(\Vect)=1$. The entries of $H$ are the dot products
of the integer multiplicity vectors of the $L_i$ in the simple object
basis of $K_0(\B)$. Over $\mathbb R$, the kernel of this Gram matrix is
the space of relations between those vectors; extending scalars to
$\mathbb C$ gives the displayed kernel. The remaining assertions follow
from Lemma~\ref{lem: sandwich radical}.
\end{proof}

For \(\B=\Z(\Vec_G)\), the maximal stratum specializes to what \cite{NikshychBurnside2026} calls the bottom ideal, and \(H\) to the Lagrangian Gram matrix studied there.

\begin{remark}
\label{rem: Tannakian regular endpoint}
Let $A$ be the regular algebra of a Tannakian subcategory
$\Rep(\Gamma)\subset\B$. Then $\B_A$ is the faithfully $\Gamma$-graded
de-equivariantization of $\B$, with neutral component
$\C=\B_A^{\loc}$. Every homogeneous component is an invertible
$\C$-bimodule category. By
Propositions~\ref{prop:diagonal component correspondence}
and~\ref{prop:diagonal local transition component}, these are precisely the
diagonal transition components. Consequently the entire product in
$K^\mmod(\B)$ of any two residual classes $(A,A,\phi)$ and $(A,A,\psi)$ lies
in the $\C$-stratum. Theorem~\ref{thm: automorphisms etale algebra} and the
grading give
\begin{equation}
\label{eqn: Tannakian transition element}
\Theta_A
=\sum_{\alpha\in\Aut_{\alg}(A)}\ind_A(\alpha)
=|\ker(\ind_A)|\sum_{h\in H}h,
\qquad H=\operatorname{Im}(\ind_A).
\end{equation}
Thus the diagonal block has the corresponding scaled ordinary Hecke
multiplication already in the full module Grothendieck ring.

More generally, whenever the support of a diagonal $\Theta_A$ is contained
in $H$, $H$-bi-invariance forces
$\Theta_A=m_A\sum_{h\in H}h$ for some $m_A>0$, and the corresponding diagonal block is
a scaled Hecke ring.
\end{remark}

\section{Illustrations from classical geometry}
\label{sect: Example classical}

We illustrate the results of Section~\ref{sect:The Grothendieck ring of module categories}
using the classical geometry of totally isotropic subspaces of a split
quadratic space. Throughout this section, $p$ is an odd prime,
$V=\mathbb F_p^n$, and $\B=\Z(\Vec_V)$.

\subsection{Localization and stratum multiplication}
\label{subsect: orthogonal localization data}

Choose a primitive $p$-th root of unity $\zeta\in k^\times$ and put
\[
E=\mathbb H(V)=V\oplus V^*,
\qquad Q(v,f)=f(v)\in\mathbb F_p,
\qquad \mathbf q(v,f)=\zeta^{Q(v,f)}.
\]
Let $b(x,y)=Q(x+y)-Q(x)-Q(y)$ be the polar form of $Q$; all orthogonal
complements below refer to $b$. There is a standard braided equivalence
$\B\simeq\C(E,\mathbf q)$ \cite[Section~2.11]{DGNO2}.
For each totally isotropic subspace $U\leq E$, let $A_U$ be the connected
\'etale algebra with underlying object $\bigoplus_{u\in U}u$. Its algebra
structure is unique up to algebra isomorphism, and every connected
\'etale algebra in $\B$ is of this form. The corresponding localization is
$\B_{A_U}^{\loc}\simeq\C(E_U,\mathbf q_U)$, where \(E_U=U^\perp/U\) and \(\mathbf q_U\) is induced from \(\mathbf q\).

If $\dim U=r$, then $E_U$ is split of dimension $2(n-r)$. Put
$X_r=\{U\leq E\mid\dim U=r,\ Q|_U=0\}$ for $0\leq r\leq n$.
The localization poset is the chain indexed by $r=0,\ldots,n$ 
and the fiber over the localization of Frobenius--Perron dimension
$p^{2(n-r)}$ is the polar Grassmannian $X_r$.

For $U,W\in X_n$, the algebras $A_U,A_W$ are Lagrangian and the residual
category is $\Vect$. Corollary~\ref{cor: maximal Gram sandwich} gives
\begin{equation}
\label{eqn: maximal orthogonal transition matrix}
\Theta^{(n)}_{UW}
=\dim_k\Hom_\B(A_U,A_W)
=|U\cap W|=p^{\dim(U\cap W)}.
\end{equation}
Thus the maximal stratum is the integral matrix sandwich ring with
product $X*Y=X\Theta^{(n)}Y$.

For $n\geq1$, the projectivizations of the maximal totally isotropic
subspaces are traditionally called the \emph{generators} of the hyperbolic
quadric $Q^+(2n-1,p)$; we use the same term for the corresponding vector
subspaces. They form two systems of generators
$X_n=X_n^+\sqcup X_n^-$: two generators $U,W$ belong to the same system
if and only if $n-\dim(U\cap W)$ is even
\cite[Theorem~7.12(i)]{TitsBuildings}. Fixing one generator determines
the two systems by this parity rule; the unordered partition is independent
of that choice. For $n=1$ the systems consist of the two singular lines.

The \emph{full dual polar graph} has vertex set $X_n$, with distinct vertices
adjacent when their intersection has dimension $n-1$. Its distance is
$d(U,W)=n-\dim(U\cap W)$ \cite[Section~9.4]{BCN}. Every edge joins
opposite systems, so the graph is bipartite with parts $X_n^+,X_n^-$.
For $n=0$ there is a single vertex and no partition into two systems. Let $R_i$
denote the distance-$i$ matrix, for $0\leq i\leq n$, so $R_0$ is the
identity matrix. Since $\Theta^{(n)}_{UW}=p^{n-d(U,W)}$, we have
\[
\Theta^{(n)}=\sum_{i=0}^n p^{n-i}R_i.
\]

Let $N$ have rows indexed by $U\in X_n$ and columns indexed by all vectors
$z\in E$, including zero, with $N_{U,z}=\mathbf1_U(z)$, where
$\mathbf1_U$ is the indicator function of $U$. Since
$(NN^{\mathsf t})_{UW}=|U\cap W|$, we obtain
\begin{equation}
\label{eqn: maximal orthogonal Gram factorization}
\Theta^{(n)}=NN^{\mathsf t}.
\end{equation}
In particular, $\Theta^{(n)}$ is positive semidefinite and, over
$\mathbb C$, $\ker\Theta^{(n)}=\ker N^{\mathsf t}$. This is the space of
linear relations among the classes $[A_U]$ in
$K_0(\B)\otimes_{\mathbb Z}\mathbb C$.

In a split four-dimensional quadratic space, the maximal totally isotropic
planes projectivize to the generator lines of $Q^+(3,p)$. Their two
systems are called the \emph{rulings}.

\begin{proposition}
\label{prop: orthogonal ruling relations}
Suppose $n\geq2$. For $T\in X_{n-2}$, the maximal isotropic subspaces
containing $T$ correspond to the maximal isotropic planes of the split
four-dimensional quotient $T^\perp/T$. Its two rulings are the restrictions
of $X_n^+,X_n^-$, up to interchanging their names, and
\begin{equation}
\label{eqn: orthogonal ruling relations}
\sum_{\substack{U\in X_n^+\\T\subseteq U}}[A_U]
=
\sum_{\substack{U\in X_n^-\\T\subseteq U}}[A_U]
\quad\text{in }K_0(\B).
\end{equation}
As $T$ varies, these integral identities span over $\mathbb C$ all linear
relations among the classes $[A_U]$.
\end{proposition}

\begin{proof}
To prove \eqref{eqn: orthogonal ruling relations}, compare coefficients
in the simple object basis of $K_0(\B)$. Fix $z\in E$. If $z\in T$,
the coefficient is $p+1$ on each side, since each ruling of $T^\perp/T$
contains $p+1$ generators. If $z\in T^\perp\setminus T$ and $Q(z)=0$,
then $\langle z+T\rangle$ is a point of the hyperbolic quadric $Q^+(3,p)$
in $\mathbb P(T^\perp/T)$. Exactly one generator from each ruling contains
this point, so exactly one term on each side contains $z$, giving
coefficient $1$. In every other case the coefficient is $0$ on both sides.

The full dual polar graph is distance-regular of diameter $n$, and its
adjacency matrix $R_1$ has $n+1$ distinct eigenvalues
\cite[Section~9.4]{BCN}. Let $V_0,\ldots,V_n$ be the corresponding
eigenspaces, ordered by decreasing eigenvalue. Let
$\{e_U\}_{U\in X_n}$ be the standard basis of $\mathbb C^{X_n}$.
For $0\leq t\leq n$, let $C_t$ be the incidence matrix with rows indexed
by $S\in X_t$, columns indexed by $U\in X_n$, and entry $1$ if
$S\subseteq U$ and $0$ otherwise. Its row indexed by $S$ is
$f_S:=\sum_{U\in X_n,\ S\subseteq U}e_U$.
The incidence decomposition of \cite[Subsection~2.4]{VanhoveAntidesigns} gives
\[
\operatorname{Im}(C_t^{\mathsf t})
=\operatorname{span}\{f_S:S\in X_t\}
=V_0\oplus\cdots\oplus V_t.
\]

Define the sign operator $D$ by $De_U=e_U$ for $U\in X_n^+$ and
$De_U=-e_U$ for $U\in X_n^-$. Since the graph is bipartite,
$DR_1D=-R_1$. Thus $D$ sends an eigenvector with eigenvalue $\theta$ to one with
eigenvalue $-\theta$. The decreasing order of the eigenvalues then gives
$D(V_j)=V_{n-j}$. For $T\in X_{n-2}$, the vectors $f_T$ span
$V_0\oplus\cdots\oplus V_{n-2}$ by the $t=n-2$ case of the incidence
decomposition. Their images under $D$ therefore span
\[
D(V_0\oplus\cdots\oplus V_{n-2})=V_2\oplus\cdots\oplus V_n.
\]
But
\[
Df_T=
\sum_{\substack{U\in X_n^+\\T\subseteq U}}e_U
-\sum_{\substack{U\in X_n^-\\T\subseteq U}}e_U
\]
is exactly the difference of the two ruling sums attached to $T$. Hence these
vectors span $V_2\oplus\cdots\oplus V_n$.

For $z\neq0$ with $Q(z)=0$, the column of $N$ indexed by $z$ is
$f_{\langle z\rangle}$, since $z\in U$ if and only if
$\langle z\rangle\subseteq U$.  If $Q(z)\neq0$, no totally isotropic subspace
contains $z$, so that column is zero. The column at $z=0$ is the
all-ones vector $\mathbf1$, which is already in the span of the
$f_S$ for $S\in X_1$: every generator contains
$(p^n-1)/(p-1)$ one-dimensional subspaces, so
\[
\sum_{S\in X_1}f_S=\frac{p^n-1}{p-1}\,\mathbf1.
\]
Consequently, the $t=1$ case of the same incidence decomposition gives
\[
\operatorname{col}(N)
=\operatorname{span}\{f_S:S\in X_1\}=V_0\oplus V_1.
\]
Since the adjacency matrix $R_1$ is real symmetric, its distinct
eigenspaces are mutually orthogonal. The Gram factorization therefore gives
\[
\ker\Theta^{(n)}=\ker N^{\mathsf t}
=(\operatorname{col}N)^\perp=V_2\oplus\cdots\oplus V_n.
\]
Thus these ruling differences span all relations.
\end{proof}

For $n\geq2$, Proposition~\ref{prop: orthogonal ruling relations} also gives
a geometric description of the radical of the complexified maximal stratum. Under $x_{U,W}\leftrightarrow e_U\otimes e_W$,
Lemma~\ref{lem: sandwich radical} identifies the radical of
$S_{\Vect}(\B)\otimes_{\mathbb Z}\mathbb C$ with
\[
(\ker\Theta^{(n)})\otimes\mathbb C^{X_n}
+\mathbb C^{X_n}\otimes(\ker\Theta^{(n)}).
\]
Thus, for each fixed $W$, replacing $[A_U]$ in
\eqref{eqn: orthogonal ruling relations} by $x_{U,W}$ or by $x_{W,U}$
gives two types of ruling differences that together linearly span the
radical over $\mathbb C$.

We now turn to the general localization stratum. Fix $0\leq r\leq n$,
choose a split quadratic space
$E_r\simeq\mathbb H(\mathbb F_p^{n-r})$, and choose isometries
\begin{equation}
\label{eqn: orthogonal gauges}
\eta_U:E_U\xrightarrow{\sim}E_r
\qquad(U\in X_r).
\end{equation}
Put $G_r=O(E_r,Q_r)$, so $G_r\cong O^+_{2(n-r)}(\mathbb F_p)$ for $r<n$
and $G_n=1$. This is the group of braided autoequivalences of the
localization up to braided tensor isomorphism
\cite[Section~2.11]{DGNO2}.

For $U,W\in X_r$, the pairing
$U/(U\cap W)\times W/(U\cap W)\to\mathbb F_p$ induced by $b$ is between
vector spaces of the same dimension. Its left and right radicals are
$(U\cap W^\perp)/(U\cap W)$ and $(W\cap U^\perp)/(U\cap W)$,
respectively. Hence one radical vanishes if and only if the other does,
and either condition is equivalent to nondegeneracy of the pairing.
Thus $U\cap W^\perp=U\cap W$ if and only if $W\cap U^\perp=U\cap W$. The form $Q$ descends to
$(U+W)^\perp/(U\cap W)$. When $U\cap W^\perp=U\cap W$, the natural quotient
maps from $(U+W)^\perp/(U\cap W)$ to $E_U$ and $E_W$ are isometries:
their kernels are the two radical quotients, all three spaces have
dimension $2(n-r)$, and the maps preserve the induced quadratic forms.
Consequently there is a canonical isometry
\begin{equation}
\label{eqn: canonical residual isometry}
\psi_{UW}:E_U\xrightarrow{\sim}E_W,
\qquad \psi_{UW}(z+U)=z+W
\quad(z\in(U+W)^\perp).
\end{equation}
Every class in $E_U$ has such a representative, and its image is
independent of that representative. We have
$\psi_{WU}=\psi_{UW}^{-1}$ and $\psi_{UU}=\id$.
Under the chosen identifications \(\eta_U\), this isometry corresponds to
\begin{equation}
\label{eqn: orthogonal transport labels}
g_{UW}:=\eta_W\psi_{UW}\eta_U^{-1}\in G_r.
\end{equation}

\begin{proposition}
\label{thm: orthogonal transition matrix}
For $U,W\in X_r$, the transition matrix of the $r$-th localization stratum is
\begin{equation}
\label{eqn: orthogonal transition matrix}
\Theta^{(r)}_{UW}=
\begin{cases}
p^{\dim(U\cap W)}g_{UW},&U\cap W^\perp=U\cap W,\\
0,&U\cap W^\perp\supsetneq U\cap W,
\end{cases}
\end{equation}
where $g_{UW}$ is a standard basis element of $\mathbb Z[G_r]$.
The subgroups $H_U$ defined by \eqref{eqn: stratum groups} are trivial
for every $U\in X_r$.
In particular, $\Theta^{(r)}_{UU}=p^r1_{G_r}$ and
$\Theta^{(r)}_{WU}=(\Theta^{(r)}_{UW})^\vee$.
Replacing $\eta_U$ by $t_U\eta_U$ sends $g_{UW}$ to
$t_Wg_{UW}t_U^{-1}$.
\end{proposition}

\begin{proof}
In the notation of \eqref{eqn: middle transition category}, put
$\T_{UW}=\B_{A_U}\bt_\B\B_{A_W}$.
Equation~\eqref{eqn: transition component count} and
$\dim_k\Hom_\B(A_W,A_U)=|U\cap W|$ give an equivalence of
$\B$-module categories
\[
\T_{UW}\simeq\bigoplus_{a=1}^{p^{\dim(U\cap W)}}\mathcal P_a,
\]
where every $\mathcal P_a$ is indecomposable and equivalence classes may
repeat. The localized full centers $\widetilde L_{\mathcal P_a}$ defined
by \eqref{eqn: localized full center} are connected Lagrangian algebras in
$\C(E_U,\mathbf q_U)\bt\C(E_W,\mathbf q_W)^\rev$, by
Proposition~\ref{prop: local transition components}.

For $\bar x=x+U\in E_U$, let $F_U(\bar x)$ be the underlying object in
$\B$ of the corresponding simple local $A_U$-module. Then
$F_U(\bar x)\cong\bigoplus_{u\in U}(x+u)$. For $\bar y=y+W\in E_W$,
define $F_W(\bar y)$ similarly. We have
\begin{equation}
\label{eqn: pointed local Hom}
\dim_k\Hom_\B\bigl(F_W(\bar y),F_U(\bar x)\bigr)
=|(x+U)\cap(y+W)|.
\end{equation}
A nonempty intersection is a coset of $U\cap W$, and any common
representative lies in $U^\perp\cap W^\perp=(U+W)^\perp$.
Formula~\eqref{eqn: mixed contraction coefficients}, with the outer
equivalences equal to the identity, identifies this number with the
total multiplicity of $(-\bar x,\bar y)$ in
$\bigoplus_a\widetilde L_{\mathcal P_a}$. Consequently this direct sum has
homogeneous multiplicity $p^{\dim(U\cap W)}$ on
\begin{equation}
\label{eqn: the entire displayed set}
\{(-(z+U),z+W)\mid z\in(U+W)^\perp\}
\end{equation}
and zero elsewhere. 

The signed map from $(U+W)^\perp$ to this set has kernel $U\cap W$,
so its image has cardinality
$p^{\dim((U+W)^\perp)-\dim(U\cap W)}=p^{2(n-r)}=|E_U|$.
A connected \'etale algebra in a pointed braided category has
multiplicity-free isotropic subgroup support; a connected Lagrangian
algebra in the present product category has support of cardinality
$|E_U|$. Thus the support of each $\widetilde L_{\mathcal P_a}$ is the
entire set \eqref{eqn: the entire displayed set}.

The intersections of this support with the first and second coordinate
axes are parametrized, respectively, by
$(W\cap U^\perp)/(U\cap W)$ and $(U\cap W^\perp)/(U\cap W)$, via the
maps induced by $z\mapsto (-(z+U),0)$ and $z\mapsto(0,z+W)$. These are the supports of
the residual endpoint algebras. If $U\cap W^\perp\supsetneq U\cap W$,
both parameter spaces have positive dimension, so every localized full
center has nontrivial residual endpoints.
Proposition~\ref{prop: local transition criterion} shows that none
contributes to the $r$-th stratum.

If $U\cap W^\perp=U\cap W$, the common support is
$\{(-x,\psi_{UW}(x))\mid x\in E_U\}$.
The commutative algebra structure on this isotropic subgroup is unique
up to algebra isomorphism, so every $\widetilde L_{\mathcal P_a}$ is
isomorphic as an algebra to the graph algebra of $\psi_{UW}$. Each of
the $p^{\dim(U\cap W)}$ terms therefore contributes $g_{UW}$ after
transport by $\eta_U,\eta_W$. This proves
\eqref{eqn: orthogonal transition matrix} with its multiplicities.

On the diagonal, \(\psi_{UU}=\id\) gives \(\Theta^{(r)}_{UU}=p^r1_{G_r}\). Since \(\Theta^{(r)}_{UU}\) is \(H_U\)-bi-invariant, this forces \(H_U=1\).
Finally, reversal and change of
identifications follow from the definition of $g_{UW}$, in agreement with
Lemma~\ref{lem: transition duality} and
\eqref{eqn: transition gauge change}.
\end{proof}

Since $H_U=1$, the additive group of the stratum is
$\bigoplus_{P,Q\in X_r}\mathbb Z[G_r]$, with basis $x_{P,Q}(g)$ for
$g\in G_r$. Theorem~\ref{thm: stratum multiplication} gives
\begin{equation}
\label{eqn: orthogonal stratum multiplication}
x_{P,U}(g)x_{W,Q}(h)=
\begin{cases}
p^{\dim(U\cap W)}x_{P,Q}(h\,g_{UW}\,g),&U\cap W^\perp=U\cap W,\\
0,&\text{otherwise}.
\end{cases}
\end{equation}

\begin{remark}
\label{rem: orthogonal local system}
Join distinct $U,W\in X_r$ when $U\cap W^\perp=U\cap W$.
The residual spaces $E_U$ and the isometries $\psi_{UW}$ form a canonical
\emph{orthogonal local system}: a quadratic space at each vertex and an
isometry along each oriented edge, with the reverse edge carrying the
inverse isometry. It is defined by $(E,Q)$ independently of the choices
of $\eta_U$.
Transport along a path is the composition of its edge isometries.
The maps $\eta_U$ express these transports in common coordinates; all
$g_{UW}$ can be made the identity if and only if every closed path has
identity transport. The example with $n=2,r=1$ below exhibits an
obstruction to such a choice.
\end{remark}

\subsection{Low-rank examples}
\label{subsect: low-rank orthogonal examples}

The following examples give a geometric explanation of nonsemisimplicity and show how intermediate strata differ from maximal ones. Over \(\mathbb C\), the maximal stratum is semisimple in Witt rank one, whereas in rank two the relation between the two rulings produces a nonzero radical and forces the whole module Grothendieck ring to be nonsemisimple.    For intermediate strata, the augmented transition matrices need not share the positivity of the maximal-stratum Gram matrix: for \(r=1\), the matrix is positive semidefinite and singular in Witt rank two, but indefinite and nonsingular in Witt rank three.   
Moreover, the rank-two example exhibits nonidentity transport around a triangle, so the labels \(g_{UW}\) in \eqref{eqn: orthogonal transport labels} cannot all be made equal to the identity by changing the identifications \(\eta_U\).
They therefore carry information about stratum multiplication that is lost under augmentation.

\paragraph{Witt rank one: maximal stratum.}
For $n=1$, there are two maximal isotropic lines and the full dual polar
graph is $K_2$. The transition matrix is
\[
\Theta=\begin{pmatrix}p&1\\1&p\end{pmatrix},
\]
with eigenvalues $p+1$ and $p-1$. The complexified maximal stratum is
semisimple by Lemma~\ref{lem: sandwich radical}.

\paragraph{Witt rank two: maximal stratum.}
For $n=2$, the maximal isotropic planes form two families, traditionally
called the two rulings, $\mathcal R_\pm=X_2^\pm$, each of size
$p+1$. Distinct planes in the same family intersect only in zero, while
planes from different families intersect in a line. Thus the full dual
polar graph is $K_{p+1,p+1}$. Ordering the two families separately gives
\begin{equation}
\label{eqn: rank two maximal transition matrix}
\Theta=
\begin{pmatrix}
(p^2-1)I+J&pJ\\
pJ&(p^2-1)I+J
\end{pmatrix},
\end{equation}
where $I$ and $J$ are the identity and all-ones matrices of size $p+1$, respectively. Its eigenvalues are $2p(p+1)$, $p^2-1$, and
$0$, with multiplicities $1$, $2p$, and $1$, respectively. Its kernel is
generated by the vector that is $1$ on $\mathcal R_+$ and $-1$ on
$\mathcal R_-$. Proposition~\ref{prop: orthogonal ruling relations},
with $T=0$ the unique member of $X_0$, gives
\begin{equation}
\label{eqn: rank two ruling relation}
\sum_{L\in\mathcal R_+}[A_L]=\sum_{L\in\mathcal R_-}[A_L]
\qquad\text{in }K_0(\B).
\end{equation}
This recovers the Gram matrix calculation and Lagrangian relation of
\cite[Proposition~3.10]{NikshychBurnside2026}.
The spectrum shows that every linear relation among these classes is a
scalar multiple of this one. By Corollary~\ref{cor: maximal Gram sandwich}
and Lemma~\ref{lem: sandwich radical}, the complexified maximal stratum
has a nonzero radical. Proposition~\ref{prop: semisimplicity poset grading}
therefore implies that the whole complexified module Grothendieck ring
for $\Z(\Vec_{\mathbb F_p^2})$ is not semisimple.

\paragraph{Witt rank two: intermediate stratum.}
For \(n=2,r=1\), \(X_1\) is the set of isotropic lines in \(E\), equivalently the set of \(\mathbb F_p\)-points of the hyperbolic quadric \(Q^+(3,p)\subset\mathbb P(E)\). Under the Segre identification
\[
Q^+(3,p)\simeq\mathbb P^1(\mathbb F_p)\times\mathbb P^1(\mathbb F_p),
\]
the two rulings are the fibers of the two projections. Thus two distinct points of \(X_1\) are nonorthogonal exactly when they differ in both factors.
Hence
\[
\Theta^{(1)}_{UW}=
\begin{cases}
p\,1_{G_1},&U=W,\\
g_{UW},&U\ne W\text{ and }U\not\perp W,\\
0,&\text{otherwise}.
\end{cases}
\]
Applying the augmentation \(\varepsilon:\mathbb Z[G_1]\to\mathbb Z\) entrywise gives the scalar matrix
\[
M^{(1)}=pI_{(p+1)^2}+(J-I)\otimes(J-I),
\]
where $I,J$ in the tensor factors have size $p+1$. The matrix $J-I$ has
eigenvalues $p$ and $-1$, with multiplicities $1$ and $p$. Hence $M^{(1)}$
has eigenvalues
$p(p+1)$, $0$, and $p+1$, with multiplicities $1$, $2p$, and $p^2$,
respectively.  Thus \(M^{(1)}\) is positive semidefinite and singular.

The labels \(g_{UW}\) in \eqref{eqn: orthogonal transport labels} retain further information.
Choose a hyperbolic basis
$e_1,e_2,f_1,f_2$, with $b(e_i,f_j)=\delta_{ij}$, and put
\[
U=\langle e_1\rangle,
\qquad W=\langle f_1\rangle,
\qquad T=\langle e_1+e_2+f_1-f_2\rangle.
\]
These isotropic lines are pairwise nonorthogonal. For nonorthogonal isotropic lines \(\langle u\rangle,\langle v\rangle\)
and
$x\in\langle u\rangle^\perp$, the isometry is
\begin{equation}
\label{eqn: singular-line perspectivity}
\psi_{\langle u\rangle,\langle v\rangle}(x+\langle u\rangle)
=x-\frac{b(x,v)}{b(u,v)}u+\langle v\rangle.
\end{equation}
The representative on the right is orthogonal to $v$, is unchanged when
$x$ is replaced by $x+au$, and has the same quadratic value as $x$.
A direct calculation along $U\longrightarrow W\longrightarrow T\longrightarrow U$
sends the representatives $e_2,f_2$ to $-e_1+f_2,e_1+e_2$, respectively.
Thus $\psi_{TU}\psi_{WT}\psi_{UW}$ exchanges $e_2+U$ and $f_2+U$.
This is nonidentity transport around a triangle in the local system of
Remark~\ref{rem: orthogonal local system}. Under a change of endpoint identifications, the product
$g_{TU}g_{WT}g_{UW}$ is conjugated at $U$. 
Its nonidentity conjugacy class therefore shows that no simultaneous choice of the identifications \(\eta_U\) makes every \(g_{UW}\) equal to the identity.
Discarding the group labels loses this information.

\paragraph{Witt rank three: intermediate stratum.}
For \(n=3\) and \(r=1\), we use the Klein model of the split
six-dimensional quadratic space. Let \(T=\mathbb F_p^4\) and choose
a nonzero \(\operatorname{vol}\in\bigwedge^4T\). We may identify
\(E\) with \(\bigwedge^2T\) so that
\[
Q(\omega)\operatorname{vol}=\frac12\,\omega\wedge\omega.
\]
A projective line \(\ell=\mathbb P\langle u,v\rangle\) in
\(\operatorname{PG}(3,p)\) determines an
isotropic line \(\langle u\wedge v\rangle\in X_1\), and every isotropic line
arises from a unique projective line. This is the Klein
correspondence; under it, two elements of \(X_1\) are orthogonal precisely when the corresponding projective lines meet
\cite[Section~8.2]{CameronPPS}. Indeed, for the polar form \(b\) of \(Q\),
\[
b(u\wedge v,s\wedge t)\operatorname{vol}
=u\wedge v\wedge s\wedge t,
\]
and the right-hand side vanishes precisely when \(u,v,s,t\)
are linearly dependent.

Thus two distinct elements \(U,W\in X_1\) give a nonzero transition entry precisely when the corresponding projective lines are skew, that is, disjoint.  Hence
\[
\Theta^{(1)}_{UW}=
\begin{cases}
p\,1_{G_1},&U=W,\\
g_{UW},&U\ne W\text{ and the corresponding projective lines are skew},\\
0,&\text{otherwise}.
\end{cases}
\]
Let \(A_{\mathrm{skew}}\) denote the adjacency matrix on the set of projective lines in \(\operatorname{PG}(3,p)\), with entry \(1\) for two distinct skew lines and \(0\) otherwise. Applying the augmentation \(\varepsilon:\mathbb Z[G_1]\to\mathbb Z\) entrywise to \(\Theta^{(1)}\) gives
\[
M^{(1)}=pI+A_{\mathrm{skew}}.
\]

Let $C$ be the point--line incidence matrix of $\operatorname{PG}(3,p)$,
with rows indexed by projective points and columns by projective lines.
The matrix $C^{\mathsf t}C$ has diagonal $p+1$, entry $1$ for distinct
intersecting lines, and entry $0$ for skew lines. Thus
$A_{\mathrm{skew}}=J+pI-C^{\mathsf t}C$, and
\begin{equation}
\label{eqn: rank three scalar incidence matrix}
M^{(1)}=J+2pI-C^{\mathsf t}C.
\end{equation}
The numbers of points and lines are $N_{\mathrm{pt}}=p^3+p^2+p+1$ and
$N_{\mathrm{ln}}=(p^2+1)(p^2+p+1)$. Every point lies on $p^2+p+1$ lines,
and every two distinct points determine one line, so
$CC^{\mathsf t}=p(p+1)I+J$.
Its eigenvalues are \((p+1)(p^2+p+1)\) and \(p(p+1)\), with multiplicities \(1\) and \(p^3+p^2+p\), respectively. Thus \(C\) has full row rank. Since \(C^{\mathsf t}C\) and \(CC^{\mathsf t}\) have the same nonzero eigenvalues with the same multiplicities, \(C^{\mathsf t}C\) has in addition the eigenvalue \(0\) with multiplicity \(p^4+p^2\).

The eigenspace of \(C^{\mathsf t}C\) for the eigenvalue \((p+1)(p^2+p+1)\) is spanned by the all-ones vector \(\mathbf1\) indexed by projective lines, while the eigenspaces for \(p(p+1)\) and \(0\) are orthogonal to \(\mathbf1\).
Hence \(J\) acts on these three eigenspaces by \(N_{\mathrm{ln}},0,0\), respectively, and \eqref{eqn: rank three scalar incidence matrix} gives the corresponding eigenvalues of \(M^{(1)}\):
\[
N_{\mathrm{ln}}+2p-(p+1)(p^2+p+1)=p^4+p, \qquad
2p-p(p+1)=-p(p-1),
\qquad\text{and}\qquad
2p,
\]
with multiplicities \(1\), \(p^3+p^2+p\), and \(p^4+p^2\), respectively.
Therefore \(M^{(1)}\) is indefinite and nonsingular.

\section{Module Grothendieck ring of the Drinfeld double of a group}
\label{sect: Example Drinfeld}

We now apply the general results of
Section~\ref{sect:The Grothendieck ring of module categories} to
$\B=\Rep(D(G))$, where $G$ is a finite group, with the aim of computing
$K^\mmod(\Rep(D(G)))$ explicitly. This extends the classical picture of
Section~\ref{sect: Example classical} to the group-theoretical setting.
An alternative description of $K^\mmod(\Rep(D(G)))$ in terms of the
degree-two cocycle-decorated double Burnside ring, including an explicit
multiplication formula, is given in
\cite{NikshychBurnside2026}.

\subsection{Relative cohomology and crossed extensions}
\label{subsect: relative cohomology}

Let
\begin{equation}
\label{eqn: group extension}
1\longrightarrow N\longrightarrow K\xrightarrow{\pi}Q\longrightarrow 1
\end{equation}
be an extension of finite groups. Throughout this section, group cohomology has
coefficients in $k^\times$ with the trivial action. Inflation identifies the
normalized cochain complex $C^\bullet(Q,k^\times)$ with a subcomplex of
$C^\bullet(K,k^\times)$.  The relative cohomology of \eqref{eqn: group extension}
is
\[
H^n_{\mathrm{rel}}(Q,K;k^\times)
:=H^n\!\left(C^\bullet(K,k^\times)/
\operatorname{inf} C^\bullet(Q,k^\times)\right).
\]
The short exact sequence of cochain complexes gives the exact sequence
\begin{equation}
\label{eqn: relative exact sequence}
H^2(Q,k^\times)\xrightarrow{\operatorname{inf}}H^2(K,k^\times)
\longrightarrow H^2_{\mathrm{rel}}(Q,K;k^\times)
\xrightarrow{\Delta}H^3(Q,k^\times)
\xrightarrow{\operatorname{inf}}H^3(K,k^\times).
\end{equation}
We recall the low-dimensional interpretation of the middle term following
\cite{HJPR}; see also \cite{Hue} for the crossed pair exact sequence and
its relation to the Lyndon--Hochschild--Serre spectral sequence.

A \emph{$K$-crossed extension of $N$ by $k^\times$} is a central extension
\begin{equation}
\label{eqn: crossed extension}
1\longrightarrow k^\times\longrightarrow\widetilde N
\xrightarrow{\partial}N\longrightarrow1
\end{equation}
together with an action of $K$ on $\widetilde N$ by group automorphisms,
trivial on the kernel $k^\times$, satisfying
\begin{equation}
\label{eqn: crossed module identities}
\partial({}^k x)=k\partial(x)k^{-1},
\qquad
{}^{\partial(x)}y=xyx^{-1}
\end{equation}
for $k\in K$ and $x,y\in\widetilde N$. Here $K$ acts on $N$ by
conjugation, and $\partial(x)$ in the second identity is viewed in
$N\subseteq K$. The composite
$\widetilde N\xrightarrow{\partial}N\hookrightarrow K$ is a crossed
module. Equivalence of $K$-crossed extensions is an isomorphism of central extensions inducing
the identity on $k^\times$ and $N$ and commuting with the $K$-actions.
Equivalence classes form an abelian group under Baer sum, denoted
$\operatorname{XExt}_K(N;k^\times)$.

\begin{proposition}
\label{prop: relative crossed extensions}
There is a canonical isomorphism \cite{HJPR}
\[
H^2_{\mathrm{rel}}(Q,K;k^\times)
\cong \operatorname{XExt}_K(N;k^\times).
\]
\end{proposition}

For later comparison with Davydov's parameters
\cite[Proposition~3.11]{D}, we recall the
explicit factor system description of crossed extensions from
\cite[Section~3]{Rat}, written multiplicatively for coefficients in
$k^\times$. Choose a normalized section $n\mapsto u_n$ of
$\partial:\widetilde N\to N$. There are normalized functions $\mu:N\times N\to k^\times$ and $\eps:K\times N\to k^\times$ defined by
\begin{equation}
\label{eqn: mu epsilon definition}
u_n u_m=\mu(n,m)u_{nm},
\qquad
{}^k u_n=\eps_k(n)u_{knk^{-1}}.
\end{equation}
The crossed module axioms are equivalent to
\begin{align}
\mu(n,m)\mu(nm,\ell)
  &=\mu(m,\ell)\mu(n,m\ell),                                      \label{eqn: relative cocycle 1}\\
\eps_{kh}(n)
  &=\eps_k(hnh^{-1})\eps_h(n),                                    \label{eqn: relative cocycle 2}\\
\mu(n,m)\eps_k(nm)
  &=\eps_k(n)\eps_k(m)
    \mu(knk^{-1},kmk^{-1}),                                       \label{eqn: relative cocycle 3}\\
\mu(n,m)
  &=\eps_n(m)\mu(nmn^{-1},n).                                    \label{eqn: relative cocycle 4}
\end{align}
Replacing $u_n$ by
$u_n'=c(n)u_n$, where $c:N\to k^\times$ is a normalized cochain, replaces
$(\mu,\eps)$ by
\begin{equation}
\label{eqn: relative coboundary}
\mu'(n,m)=\mu(n,m)\frac{c(n)c(m)}{c(nm)},
\qquad
\eps'_k(n)=\eps_k(n)\frac{c(n)}{c(knk^{-1})}.
\end{equation}
Consequently, modulo \eqref{eqn: relative coboundary}, the pairs $(\mu,\eps)$
satisfying \eqref{eqn: relative cocycle 1}--\eqref{eqn: relative cocycle 4}
give an explicit model for
$H^2_{\mathrm{rel}}(Q,K;k^\times)$.

We use the differential
\[
(df)(g,h,\ell)=\frac{f(h,\ell)f(g,h\ell)}{f(gh,\ell)f(g,h)}.
\]
Let \(\xi\in H^2_{\mathrm{rel}}(Q,K;k^\times)\). We may choose a normalized cochain \(f:K^2\to k^\times\) whose image in the quotient complex represents \(\xi\), with
\[
df=\pi^*\nu
\]
for a normalized \(3\)-cocycle \(\nu:Q^3\to k^\times\). Then \(\Delta(\xi)=[\nu]\).
We fix the isomorphism in
Proposition~\ref{prop: relative crossed extensions} by assigning to $f$
the parameters
\begin{equation}
\label{eqn: relative cochain parameters}
\mu(n,m)=f(n,m),
\qquad
\eps_g(n)=\frac{f(g,n)}{f(gng^{-1},g)}
\quad(g\in K,\ n,m\in N).
\end{equation}
These are Ratcliffe's restriction formulas
\cite[equations~(5.2)--(5.3)]{Rat}, written multiplicatively. 
Although \(f\) need not be a \(2\)-cocycle on \(K\), the identity \(df=\pi^*\nu\) implies that \(f\) satisfies the cocycle identity whenever one of the three arguments lies in \(N\). Thus Ratcliffe's calculation gives \eqref{eqn: relative cocycle 1}--\eqref{eqn: relative cocycle 4} for the pair \((\mu,\eps)\) in \eqref{eqn: relative cochain parameters}.
If \(f\) is replaced by an equivalent representative \(f\,db\,\pi^*a\), then \(db\) changes \((\mu,\eps)\) by \eqref{eqn: relative coboundary} with \(c=b|_N\), while the inflated factor \(\pi^*a\) does not change \((\mu,\eps)\). Hence the class of \((\mu,\eps)\) modulo \eqref{eqn: relative coboundary} depends only on \(\xi\).

We will also need a description at the cocycle level of the connecting map $\Delta$
in \eqref{eqn: relative exact sequence}. We use Ratcliffe's
crossed extension obstruction construction \cite[Section~6]{Rat}.
Davydov--Simmons \cite{DS}
describe the quotient associator and the corresponding category of local
modules in the more general $3$-cocycle-twisted setting.
Let $\xi\in H^2_{\mathrm{rel}}(Q,K;k^\times)$ be represented by a
$K$-crossed extension as in \eqref{eqn: crossed extension}.
Choose a normalized section $s:Q\to K$ and put
$\kappa(x,y)=s(x)s(y)s(xy)^{-1}\in N$.
Choose normalized lifts $\widehat\kappa(x,y)\in\widetilde N$ of
$\kappa(x,y)$, with $\widehat\kappa(1,y)=\widehat\kappa(x,1)=1$.
Associativity in $K$ gives
$\kappa(x,y)\kappa(xy,z)=s(x)\kappa(y,z)s(x)^{-1}\kappa(x,yz)$.
The corresponding two lifted products have the same image in $N$, so
their ordered quotient lies in the central kernel $k^\times$. Define
$\omega_\xi:Q^3\to k^\times$ by
\begin{equation}
\label{eqn: characteristic cocycle}
\omega_\xi(x,y,z)=
{}^{s(x)}\widehat\kappa(y,z)\widehat\kappa(x,yz)
\bigl(\widehat\kappa(x,y)\widehat\kappa(xy,z)\bigr)^{-1}\in k^\times.
\end{equation}

\begin{proposition}
\label{prop: characteristic class}
The function $\omega_\xi$ is a normalized $3$-cocycle. Its cohomology class
is independent of all choices and equals $\Delta(\xi)$. If
$\widetilde N$ is written in set-theoretic coordinates  $(\mu,\eps)$ as in \eqref{eqn: mu epsilon definition},
and  $\widehat\kappa(x,y)=(1,\kappa(x,y))$, then
\begin{equation}
\label{eqn: omega formula}
\omega_\xi(x,y,z)=
\frac{\eps_{s(x)}(\kappa(y,z))\,
 \mu(s(x)\kappa(y,z)s(x)^{-1},\kappa(x,yz))}
{\mu(\kappa(x,y),\kappa(xy,z))}.
\end{equation}
\end{proposition}

\begin{proof}
Put $t_{x,y}=\widehat\kappa(x,y)$ and let $\theta_x$ be the action of
$s(x)$ on $\widetilde N$. Since $s(x)s(y)=\kappa(x,y)s(xy)$, the
crossed module identities give
\[
\theta_x\theta_y=\operatorname{Ad}(t_{x,y})\theta_{xy}.
\]
Choose the normalized lifts $u_n$ used in
\eqref{eqn: mu epsilon definition}. Write $g=n_gs(x)$ and $h=n_hs(y)$,
where $x=\pi(g)$ and $y=\pi(h)$. Define $f(g,h)\in k^\times$ by
\begin{equation}
\label{eqn: crossed extension cochain}
\begin{aligned}
u_{n_g}\theta_x(u_{n_h})t_{x,y}&=f(g,h)u_{n_{gh}},\\
n_{gh}&=n_gs(x)n_hs(x)^{-1}\kappa(x,y).
\end{aligned}
\end{equation}
Both sides of the first equality have the same image in $N$, so the
scalar is well-defined, and $f$ is normalized. Under
\eqref{eqn: relative cochain parameters}, it recovers the given crossed
extension: its restriction to $N^2$ is $\mu$, and
$f(g,n)/f(gng^{-1},g)=\eps_g(n)$ because conjugation by $u_{n_g}$
implements the action of $n_g$.

For $\ell=n_\ell s(z)$, compare the elements
\[
L=u_{n_g}\theta_x(u_{n_h})t_{x,y}
       \theta_{xy}(u_{n_\ell})t_{xy,z},\qquad
R=u_{n_g}\theta_x\!\left(u_{n_h}\theta_y(u_{n_\ell})t_{y,z}\right)
       t_{x,yz}
\]
of $\widetilde N$. The relation between $\theta_x\theta_y$ and
$\theta_{xy}$ and the definition of $\omega_\xi$ give
\[
R=u_{n_g}\theta_x(u_{n_h})t_{x,y}\theta_{xy}(u_{n_\ell})
    t_{x,y}^{-1}\theta_x(t_{y,z})t_{x,yz}
 =\omega_\xi(x,y,z)L.
\]
On the other hand, \eqref{eqn: crossed extension cochain} gives
\[
L=f(g,h)f(gh,\ell)u_{n_{gh\ell}},
\qquad
R=f(h,\ell)f(g,h\ell)u_{n_{gh\ell}}.
\]
Consequently, $
df=\pi^*\omega_\xi$.
Thus $f$ is a relative cocycle representing $\xi$, and the definition of
the connecting map gives $[\omega_\xi]=\Delta(\xi)$. Applying $d$  and using injectivity of inflation on cochains
gives $d\omega_\xi=1$. The normalization of $s$ and $t$ gives
normalization of $\omega_\xi$.

Changing $s$, the lifts, or the representative within the equivalence
class of the crossed extension gives a cochain $f'$ with the same
relative cohomology class, by
\eqref{eqn: relative cochain parameters}. Hence
$f'=f\,db\,\pi^*a$ for normalized $b$ and $a$, and taking differentials
gives $\omega_\xi'=\omega_\xi\,da$. This proves independence of all
choices and invariance under equivalent extensions. Finally, substituting
$t_{x,y}=(1,\kappa(x,y))$ into
\eqref{eqn: characteristic cocycle} gives \eqref{eqn: omega formula}.
\end{proof}

The exact sequence \eqref{eqn: relative exact sequence} has two useful
consequences. The possible values of $\Delta(\xi)$ are precisely
\[
\operatorname{Ker}\!\left(
H^3(Q,k^\times)\xrightarrow{\operatorname{inf}}H^3(K,k^\times)
\right),
\]
and, for a fixed class in this kernel, its lifts to relative cohomology form a
torsor under
\begin{equation}
\label{eqn: fiber delta}
\operatorname{Coker}\!\left(
H^2(Q,k^\times)\xrightarrow{\operatorname{inf}}H^2(K,k^\times)
\right).
\end{equation}

\begin{remark}
\label{rem: LHS interpretation}
The relative class packages the two successive obstructions visible in the
Lyndon--Hochschild--Serre spectral sequence $E_2^{p,q}=H^p(Q,H^q(N,k^\times))\Longrightarrow H^{p+q}(K,k^\times)$.
Under the cocycle description
\eqref{eqn: mu epsilon definition}--\eqref{eqn: relative coboundary},
forgetting the $K$-equivariant data $\eps$ defines a map
\[
H^2_{\mathrm{rel}}(Q,K;k^\times)
\longrightarrow H^2(N,k^\times)^Q,
\qquad [\mu,\eps]\longmapsto[\mu].
\]
Its image is
\[
\operatorname{Ker}\!\left(
 d_2^{0,2}:H^2(N,k^\times)^Q
 \longrightarrow H^2(Q,H^1(N,k^\times))
\right).
\]
For a fixed $[\mu]$ in this kernel, the possible crossed lifts form a torsor under $H^1(Q,H^1(N,k^\times))$.
For the usual left action
$(x\cdot\chi)(n)=\chi(s(x)^{-1}ns(x))$ on $H^1(N,k^\times)$,
let $\eta:Q\to H^1(N,k^\times)$ be a normalized left $1$-cocycle,
so that $\eta_{xy}=\eta_x(x\cdot\eta_y)$.
Its class $[\eta]$ acts on a crossed lift with fixed $\mu$ by
\[
\eps_g(n)\longmapsto
\eps_g(n)\,\eta_{\pi(g)}(gng^{-1}).
\]
For a coboundary $\eta_x=(x\cdot\chi)\chi^{-1}$, this is the gauge
change \eqref{eqn: relative coboundary} by the character $\chi$.
Changing the lift by $[\eta]$ changes its image under $\Delta$ by
$d_2^{1,1}([\eta])$. Equivalently, the image of $\Delta(\xi)$ in
$E_3^{3,0}=H^3(Q,k^\times)/\operatorname{Im}(d_2^{1,1})$ is
$d_3^{0,2}([\mu])$ \cite{Hue}. Thus relative cohomology records simultaneously the class
$[\mu]\in H^2(N,k^\times)^Q$, its equivariant lift, and the resulting
quotient $3$-cocycle.
\end{remark}

\subsection[Davydov's parameterization of \'etale algebras in Rep(D(G))]{Davydov's parameterization of \'etale algebras in $\Rep(D(G))$}
\label{subsect: Davydov parameterization}

We identify $\Rep(D(G))$ with $\Z(\Vec_G)$. Davydov's classification
\cite[Section~3]{D} admits a concise reformulation in terms of the
relative cohomology above. The $3$-cocycle-twisted extension and the
calculation of local modules were developed in  \cite{DS},
while \cite{HLRC} gives a total complex formulation of the cochain data. We
express the cochain data and the resulting quotient associator in terms of
a single relative cohomology class.

We recall the relevant part of the argument in \cite{D}, phrased in the
localization language used here. Let $A$ be a connected \'etale algebra in
$\Z(\Vec_G)$. Its trivial component is an indecomposable commutative
separable $G$-algebra, and hence $A_e\cong \Fun(G/K)$
for a subgroup $K\leq G$, determined up to conjugacy. Localization by
$\Fun(G/K)$ identifies $\Z(\Vec_G)_{\Fun(G/K)}^{\loc}\simeq \Z(\Vec_K)$,
and transitivity of localization reduces the classification to connected
\'etale algebras $B\in\Z(\Vec_K)$ with $B_e=k$. The support $N=\{n\in K\mid B_n\neq0\}$
is a normal subgroup of $K$, every nonzero homogeneous component is
one-dimensional, and after choosing nonzero vectors $u_n\in B_n$ one has $B\cong k_\mu[N]$, with $u_nu_m=\mu(n,m)u_{nm}$. The $K$-action has the form $k(u_n)=\eps_k(n)u_{knk^{-1}}$.
The conditions on $(\mu,\eps)$ are exactly
\eqref{eqn: relative cocycle 1}--\eqref{eqn: relative cocycle 4}, and his
change-of-basis relation is \eqref{eqn: relative coboundary}. Thus these
data determine a single class $\xi\in H^2_{\mathrm{rel}}(K/N,K;k^\times)$.

\begin{definition}
A \emph{twisted section of $G$} is a pair $(N\triangleleft K,\xi)$, where $N\triangleleft K\leq G$ and $\xi\in H^2_{\mathrm{rel}}(K/N,K;k^\times)$.
The group $G$ acts on twisted sections by conjugation and transport of
cohomology classes.
\end{definition}

For a twisted section $(N\triangleleft K,\xi)$, let
$A(N\triangleleft K,\xi)$ denote the connected \'etale algebra in
$\Z(\Vec_G)$ corresponding to  a representative pair
$(\mu,\eps)$ of $\xi$ as above.

\begin{theorem}
\label{thm: twisted sections}
The assignment $(N\triangleleft K,\xi)\mapsto A(N\triangleleft K,\xi)$ induces a bijection between $G$-conjugacy classes of twisted sections of $G$
and isomorphism classes of connected \'etale algebras in $\Rep(D(G))$.
Equivalently, after choosing representatives of the $G$-conjugacy classes of
normal sections,
\[
\mathfrak{Et}(\Rep(D(G)))
\cong
\coprod_{[N\triangleleft K]_G}
H^2_{\mathrm{rel}}(K/N,K;k^\times)/N_G(K,N),
\]
where $N_G(K,N)=\{g\in G\mid gKg^{-1}=K,\ gNg^{-1}=N\}$ acts by transport of
structure. Moreover, $\FPdim (A(N\triangleleft K,\xi))=[G:K]|N|$.
\end{theorem}

\begin{proof}
It was shown in \cite[Theorem~3.15]{D} that every
connected \'etale algebra in $\Z(\Vec_G)$ arises from a normal section
$N\triangleleft K\leq G$ and a pair $(\mu,\eps)$ satisfying
\eqref{eqn: relative cocycle 1}--\eqref{eqn: relative cocycle 4}.
For a fixed normal section, the reduced algebras are isomorphic exactly
when their pairs differ by the change of basis in
\eqref{eqn: relative coboundary}, by \cite[Lemma~3.12]{D}.
Proposition~\ref{prop: relative crossed extensions} therefore identifies
their isomorphism classes with $H^2_{\mathrm{rel}}(K/N,K;k^\times)$.
This identification commutes with conjugation, by the cocycle formulas.

Choosing a minimal idempotent $p\in A_e$ determines
$K=\operatorname{Stab}_G(p)$ and the reduced algebra $pA$. Since $G$ acts
transitively on these idempotents and Davydov's reduction and transfer
are inverse up to isomorphism \cite[Corollary~3.9]{D},
two twisted sections give isomorphic algebras exactly when they are
$G$-conjugate. For a fixed representative $N\triangleleft K$, the
remaining identifications are precisely the action of $N_G(K,N)$.

Finally, $A$ has $[G:K]$ idempotent summands of equal dimension, and the
reduced algebra has one-dimensional homogeneous components indexed by
$N$. Hence $\FPdim (A)=[G:K]|N|$.
\end{proof}

We use the convention that the associator of $\Vec_Q^\omega$,
\[
(U\otimes V)\otimes W\longrightarrow U\otimes(V\otimes W),
\]
acts on homogeneous degrees $x,y,z$ by multiplication by
$\omega(x,y,z)$. We also use the corresponding convention for
$D^\omega(Q)$.

Recall that the normalized cocycle \(\omega_\xi\) \eqref{eqn: characteristic cocycle} represents \(\Delta(\xi)\). 
The following localization theorem is the untwisted case of  \cite[Theorem 3.13]{DS}. We compute the quotient associator directly, since our calculation 
differs from the displayed quotient $3$-cocycle formula in \cite[Theorem 3.13]{DS}\footnote{After translating the associator convention of \cite{DS} to ours, formula (18) there still differs from (5.13): the \(\varepsilon\)-factor and the final \(\gamma\)-factor in that formula occur inverted. The direct calculation below gives the formula used here.}.

\begin{theorem}
\label{thm: localization twisted section}
Let $Q=K/N$, and let
$\xi\in H^2_{\mathrm{rel}}(Q,K;k^\times)$. Then there is a braided
equivalence
\begin{equation}
\label{eqn: localization twisted double}
\Rep(D(G))_{A(N\triangleleft K,\xi)}^{\loc}
\simeq
\Z(\Vec_Q^{\omega_\xi})
\simeq
\Rep(D^{\omega_\xi}(Q)).
\end{equation}
\end{theorem}

\begin{proof}
Choose a representative $(\mu,\varepsilon)$ of $\xi$, and let
$B_\xi=k_\mu[N]\in\mathcal Z(\operatorname{Vec}_K)$ be the reduced algebra
described above. Thus $B_\xi$ has basis $u_n$, $n\in N$, with
$\deg(u_n)=n$, multiplication
$
u_nu_m=\mu(n,m)u_{nm},
$
and $K$-action
$g(u_n)=\varepsilon_g(n)u_{gng^{-1}}$.
The degree-$e$ subalgebra of $A(N\triangleleft K,\xi)$ is
$\Fun(G/K)$. Under the equivalence
\[
\mathcal Z(\operatorname{Vec}_G)_{\Fun(G/K)}^{\mathrm{loc}}
\simeq
\mathcal Z(\operatorname{Vec}_K),
\]
the algebra $A(N\triangleleft K,\xi)$, regarded as a local
$\Fun(G/K)$-module, corresponds to $B_\xi$. Hence
transitivity of localization gives
\[
\operatorname{Rep}(D(G))_{A(N\triangleleft K,\xi)}^{\mathrm{loc}}
\simeq
\mathcal Z(\operatorname{Vec}_K)_{B_\xi}^{\mathrm{loc}}.
\]

Let $\mathcal C$ be the tensor category of left $B_\xi$-modules in
$\operatorname{Vec}_K$, with tensor product over $B_\xi$ determined by
its central structure. By \cite[Theorem~3.20]{DMNO},
\[
\mathcal Z(\operatorname{Vec}_K)_{B_\xi}^{\mathrm{loc}}
\simeq
\mathcal Z(\mathcal C).
\]

Choose a normalized section $s:Q\to K$ and lifts
$t_{x,y}=\widehat\kappa(x,y)$ as in the definition of $\omega_\xi$.
We regard $\widetilde N$ as the group of nonzero homogeneous elements
of $B_\xi$. The simple objects of $\mathcal C$ are the free left
modules $M_x=B_\xi v_x,\, x\in Q$,
with $\deg(v_x)=s(x)$. The right $B_\xi$-action on $M_x$ induced by
the half-braiding is $
v_x\cdot a={}^{s(x)}a\,v_x$.
Hence in the tensor product over $B_\xi$ we have
\[
v_x\otimes_{B_\xi}(av_y)
=
({}^{s(x)}a\,v_x)\otimes_{B_\xi}v_y
\qquad(a\in B_\xi).
\]

There are $B_\xi$-module isomorphisms
\[
j_{x,y}:M_x\otimes_{B_\xi}M_y\longrightarrow M_{xy},
\qquad
v_x\otimes_{B_\xi}v_y\longmapsto t_{x,y}v_{xy}.
\]
The two composites
\[
j_{xy,z}(j_{x,y}\otimes_{B_\xi}\operatorname{id})
\quad\text{and}\quad
j_{x,yz}(\operatorname{id}\otimes_{B_\xi}j_{y,z}),
\]
evaluated on the tensor of the three generators, give respectively
\[
t_{x,y}t_{xy,z}v_{xyz}
\quad\text{and}\quad
{}^{s(x)}t_{y,z}t_{x,yz}v_{xyz}.
\]
Consequently, the transported associator 
\[
(M_x\otimes_{B_\xi}M_y)\otimes_{B_\xi}M_z
\longrightarrow
M_x\otimes_{B_\xi}(M_y\otimes_{B_\xi}M_z)
\]
is the scalar $a(x,y,z)$ satisfying
\[
a(x,y,z)t_{x,y}t_{xy,z}
=
{}^{s(x)}t_{y,z}t_{x,yz}.
\]

On the other hand, $\omega_\xi$ was defined in \eqref{eqn: characteristic cocycle} by
\[
\omega_\xi(x,y,z)\,
\widehat\kappa(x,y)\widehat\kappa(xy,z)
=
{}^{s(x)}\widehat\kappa(y,z)\widehat\kappa(x,yz).
\]
Since $t_{x,y}=\widehat\kappa(x,y)$, comparison gives
$
a(x,y,z)=\omega_\xi(x,y,z).
$
In the coordinates $(\mu,\varepsilon)$, choosing
$\widehat\kappa(x,y)=(1,\kappa(x,y))$ gives
\[
\omega_\xi(x,y,z)=
\frac{\varepsilon_{s(x)}(\kappa(y,z))\,
\mu(s(x)\kappa(y,z)s(x)^{-1},\kappa(x,yz))}
{\mu(\kappa(x,y),\kappa(xy,z))},
\]
which is exactly formula~(5.13).
\end{proof}

\begin{corollary}
\label{cor: localization lattice DG}
Every localization of $\Rep(D(G))$ is a twisted double of a section of $G$.
More precisely, the points of the localization poset are represented by
\begin{equation}
\label{eqn: localization lattice DG}
\Z(\Vec_{K/N}^{\omega}),
\qquad
N\triangleleft K\leq G,
\quad
[\omega]\in
\operatorname{Ker}\!\left(
H^3(K/N,k^\times)\xrightarrow{\operatorname{inf}}H^3(K,k^\times)
\right),
\end{equation}
with braided equivalent categories identified. Conversely, every datum in
\eqref{eqn: localization lattice DG} occurs. For a fixed embedded normal
section $N\triangleleft K$ and a fixed $[\omega]$, the set of relative
cohomology classes mapping to $[\omega]$ is a torsor under
\[
\operatorname{Coker}\!\left(
H^2(K/N,k^\times)\xrightarrow{\operatorname{inf}}H^2(K,k^\times)
\right).
\]
\end{corollary}

\begin{proof}
The description of the possible classes \([\omega]\) and of each fiber follows from Theorem~\ref{thm: localization twisted section} and exactness of \eqref{eqn: relative exact sequence}.
\end{proof}

\subsection[The module Grothendieck ring of Rep(D(S3))]{The module Grothendieck ring of $\Rep(D(S_3))$}
\label{subsect: module ring DS3}

Let $\B=\Rep(D(S_3))$.
There are eight connected \'etale algebras in $\B$; see \cite[Table~1]{D}. In the notation of Theorem~\ref{thm: twisted sections}, the
corresponding normal sections are
\[
1\triangleleft S_3,
\qquad
1\triangleleft C_3,
\qquad
1\triangleleft C_2,
\qquad
C_3\triangleleft S_3,
\qquad
H\triangleleft H\quad(H=1,C_2,C_3,S_3).
\]
All relevant relative cohomology groups vanish. For $1\triangleleft K$,
this is immediate since $K/N=K$ and the quotient cochain complex is zero;
for $K\triangleleft K$, it follows from the triviality of the Schur
multipliers of $1,C_2,C_3,S_3$; and for $C_3\triangleleft S_3$, it follows
from \eqref{eqn: relative exact sequence} and the splitting of
$S_3\to C_2$. Hence no additional cohomological parameter occurs.
The corresponding quotient groups are respectively
$S_3,C_3,C_2,C_2$, and the trivial group. 

The fiber over $\Vect$ consists of four Lagrangian algebras
\[
L_1=A(1\triangleleft1),\qquad
L_2=A(C_3\triangleleft C_3),\qquad
L_3=A(C_2\triangleleft C_2),\qquad
L_4=A(S_3\triangleleft S_3).
\]
The remaining connected \'etale algebras are
\[
\be,\qquad
A=A(1\triangleleft C_3)=\Fun(S_3/C_3),\qquad
B_1=A(1\triangleleft C_2)=\Fun(S_3/C_2),\qquad
B_2=A(C_3\triangleleft S_3).
\]

The localization poset $\mathfrak{Loc}(\B)$ has four elements,
\[
\B,\qquad
\C_2:=\B_{B_1}^{\loc}\simeq\B_{B_2}^{\loc}\simeq\Z(\Vec_{C_2}),\qquad
\C_3:=\B_A^{\loc}\simeq\Z(\Vec_{C_3}),\qquad
\Vect.
\]
In the paired diagrams below, algebras are horizontally aligned with their
localizations.

\[
\begin{tikzpicture}[paired poset]
  \node (et-heading) at (0,3.9) {$\mathfrak{Et}(\B)$};
  \node (loc-heading) at (7,3.9) {$\mathfrak{Loc}(\B)$};
  \draw[->, line width=0.4pt] (et-heading.east) --
    node[above=2pt] {$A\mapsto\B_A^{\loc}$} (loc-heading.west);

  \node (unit) at (0,0) {$\be$};
  \node (A) at (0,1) {$A$};
  \node (B1) at (-3,2) {$B_1$};
  \node (B2) at (3,2) {$B_2$};
  \node (L3) at (-4,3) {$L_3$};
  \node (L1) at (-1,3) {$L_1$};
  \node (L2) at (1,3) {$L_2$};
  \node (L4) at (4,3) {$L_4$};
  \draw[poset cover] (unit) -- (A);
  \draw[poset cover] (unit) -- (B1);
  \draw[poset cover] (unit) -- (B2);
  \draw[poset cover] (A) -- (L1);
  \draw[poset cover] (A) -- (L2);
  \draw[poset cover] (B1) -- (L1);
  \draw[poset cover] (B1) -- (L3);
  \draw[poset cover] (B2) -- (L2);
  \draw[poset cover] (B2) -- (L4);

  \node (locB) at (7,0) {$\B$};
  \node (locC) at (6,1) {$\C_3$};
  \node (locD) at (8,2) {$\C_2$};
  \node (locVec) at (7,3) {$\Vect$};
  \draw[poset cover] (locB) -- (locC);
  \draw[poset cover] (locB) -- (locD);
  \draw[poset cover] (locC) -- (locVec);
  \draw[poset cover] (locD) -- (locVec);
\end{tikzpicture}
\]

The simple objects in $\B$ are
$\{\be,X_1^{\mathrm{sgn}},X_1^{\mathrm{st}},X_3^{\mathrm{tr}},
X_3^\omega,X_3^{\bar\omega},X_2^{\mathrm{tr}},X_2^{\mathrm{sgn}}\}$,
where the subscript $n$ represents the $n$-cycles on which the object is
supported, and the superscript denotes the representation of the stabilizer
of one of these cycles. The standard representation of $S_3$ is denoted
$\mathrm{st}$, and $\omega,\bar\omega$ are the values of the representation
on $(123)$, with $\omega$ a primitive third root of unity.

The fiber over $\B$ consists only of $\be$. Since
$\Aut^{\mathrm{br}}(\B)\cong C_2$, Theorem~\ref{thm: stratum multiplication}
gives $S_\B(\B)\cong\mathbb Z[C_2]$.

\medskip\noindent\textbf{Maximal stratum.}
Using the character decompositions in \cite[Table~1]{D}, one obtains the
sandwich matrix
\begin{equation}
\label{eqn: S3 sandwich matrix}
H=\big(\dim_k\Hom_{\B}(L_i,L_j)\big)_{i,j=1}^4
=
\begin{pmatrix}
6&2&3&1\\
2&6&1&3\\
3&1&3&2\\
1&3&2&3
\end{pmatrix}.
\end{equation}
Thus Corollary~\ref{cor: maximal Gram sandwich} gives
$x_{ij}x_{k\ell}=H_{jk}x_{i\ell}$ for the standard basis of the maximal stratum.
Since $H$ has rank $3$, the complexified maximal stratum is a nonsemisimple
sandwich algebra. By Proposition~\ref{prop: semisimplicity poset grading},
it follows that $K^\mmod(\B)\ot_{\mathbb Z}\mathbb C$ is nonsemisimple.
Nonsemisimplicity for noncyclic $G$ also follows from
\cite[Corollary~4.10]{NikshychBurnside2026}; the calculation above identifies
an explicit obstruction in the maximal stratum.

\medskip\noindent\textbf{The $\C_3$-stratum.}
The only algebra in this fiber is
$A=\Fun(S_3/C_3)=\be\oplus X_1^{\mathrm{sgn}}$.
It is the regular algebra of the Tannakian subcategory
$\Rep(C_2)\subset\B$ obtained by inflation along $S_3\to C_2$.

The category $\C_3\simeq\Z(\Vec_{C_3})$ has metric group
$(\mathbb F_3^2,q)$, where $q(i,j)=\omega^{ij}$.
Its isometry group is generated by simultaneous inversion and
interchange of coordinates, so
$G=\Aut^{\mathrm{br}}(\C_3)\cong C_2^2$.
The nontrivial algebra automorphism of $A$ exchanges its two minimal
idempotents and induces conjugation by a transposition on
$\Z(\Vec_{C_3})$. This inverts both $C_3$ and its characters, so the
induced autoequivalence is $\sigma(i,j)=(-i,-j)$.
Thus $\Gamma_A=\langle\sigma\rangle$.

By Remark~\ref{rem: Tannakian regular endpoint},
$\Theta_A=1+\sigma$, and products in this block have no terms in
higher localizations. Since $G$ is abelian, the stratum basis is
indexed by $\Gamma_A\backslash G/\Gamma_A=G/\Gamma_A\cong C_2$, and
\eqref{eqn: stratum multiplication} gives
\[
x_{\bar g}x_{\bar h}=2x_{\overline{hg}},
\qquad
S_{\C_3}(\B)\cong2\mathbb Z[C_2],
\]
where $2\mathbb Z[C_2]$ is the ideal of $\mathbb Z[C_2]$ with its
induced multiplication. Its complexification is $\mathbb C^2$, hence semisimple.

\medskip\noindent\textbf{The $\C_2$-stratum.}
There are two algebras in this fiber, $B_1$ and $B_2$. The first is supported
entirely on the identity and is the function algebra on $S_3/C_2$; the
second is its image under the nontrivial braided autoequivalence of $\B$.
As objects in $\B$, they are
$B_1=\be\oplus X_1^{\mathrm{st}}$ and
$B_2=\be\oplus X_3^{\mathrm{tr}}$.

The simple objects of $\B_{B_1}^{\loc}$ and $\B_{B_2}^{\loc}$ have the
respective ordered lists of underlying objects
\begin{equation}
\label{eqn: S3 C2 underlying objects}
\bigl(
B_1,\,
X_1^{\mathrm{sgn}}\oplus X_1^{\mathrm{st}},\,
X_2^{\mathrm{tr}},\,
X_2^{\mathrm{sgn}}
\bigr),
\qquad
\bigl(
B_2,\,
X_1^{\mathrm{sgn}}\oplus X_3^{\mathrm{tr}},\,
X_2^{\mathrm{tr}},\,
X_2^{\mathrm{sgn}}
\bigr).
\end{equation}
Each local simple is self-dual, and the final object in each list is the
only one with nontrivial twist.   The nontrivial braided autoequivalence of $\C_2$ exchanges the second and
third local simples in \eqref{eqn: S3 C2 underlying objects} and fixes the
first and fourth; thus $\Aut^{\mathrm{br}}(\C_2)\cong C_2$.
Fix $\eta_1=\id_{\C_2}$ and the braided equivalence
$\eta_2:\B_{B_2}^{\loc}\to\C_2$ matching the displayed lists in order.
Let $U_i:\C_2\to\B$ be the underlying object functor composed with
$\eta_i^{-1}$.

With rows indexed by $U_i(X)$ and columns by $U_j(Y)$, the entries
$\dim_k\Hom_{\B}(U_j(Y),U_i(X))$ for $(i,j)=(1,1)$ and $(1,2)$ are
\[
\begin{array}{c|cccc}
\dim_k\Hom_{\B}&\be\oplus X_1^{\mathrm{st}}
&X_1^{\mathrm{sgn}}\oplus X_1^{\mathrm{st}}
&X_2^{\mathrm{tr}}&X_2^{\mathrm{sgn}}\\ \hline
\be\oplus X_1^{\mathrm{st}}&2&1&0&0\\
X_1^{\mathrm{sgn}}\oplus X_1^{\mathrm{st}}&1&2&0&0\\
X_2^{\mathrm{tr}}&0&0&1&0\\
X_2^{\mathrm{sgn}}&0&0&0&1
\end{array}
\]
and
\[
\begin{array}{c|cccc}
\dim_k\Hom_{\B}&\be\oplus X_3^{\mathrm{tr}}
&X_1^{\mathrm{sgn}}\oplus X_3^{\mathrm{tr}}
&X_2^{\mathrm{tr}}&X_2^{\mathrm{sgn}}\\ \hline
\be\oplus X_1^{\mathrm{st}}&1&0&0&0\\
X_1^{\mathrm{sgn}}\oplus X_1^{\mathrm{st}}&0&1&0&0\\
X_2^{\mathrm{tr}}&0&0&1&0\\
X_2^{\mathrm{sgn}}&0&0&0&1
\end{array}.
\]
The underlying objects in \eqref{eqn: S3 C2 underlying objects} give
the same diagonal table for $(i,j)=(2,2)$. The table for $(2,1)$ is the
transpose of that for $(1,2)$, so both off-diagonal matrices are $I_4$.

By \eqref{eqn: mixed contraction coefficients}, these tables give the
multiplicities of simple objects in the localized transition algebras,
transported to $\C_2\bt\C_2^{\rev}$ by the chosen gauges.
For each diagonal entry,
Corollary~\ref{cor:diagonal transition element} supplies an
identity graph algebra summand, corresponding to the regular
$\B_{B_i}^{\loc}$-module component of $\B_{B_i}$. The tensor unit
$\be_{\C_2}\bt\be_{\C_2}$ occurs with multiplicity two in the localized
transition algebra, while it occurs with multiplicity one in the
identity graph summand. Hence precisely one further connected summand remains.
Subtracting the identity matrix leaves ones in the upper-left
$2\times2$ block and zeros elsewhere. Let $E$ be the Lagrangian algebra in $\C_2$
whose underlying object is the sum of the first two local simples. 
The remaining connected algebra has both residual endpoints equal to $E$.
Since
${\C_2}_E^{\loc}\simeq\Vect$, these residual data identify it as $E\bt E$.
Then, for $i=1,2$,
\[
(\eta_i\bt\eta_i^{\rev})
\operatorname{Loc}_{B_i\bt B_i}
\bigl(L(B_i,B_i,1)\odot L(B_i,B_i,1)\bigr)
\simeq\Delta_{\id_{\C_2}}\oplus(E\bt E).
\]
Only the identity graph survives in the $\C_2$-stratum, so
$\Theta_{11}=\Theta_{22}=1$.

Under the correspondence between \'etale algebras over $B_i$ and local
\'etale algebras \cite[Proposition~3.16]{DMNO}, the two lifts of $E$
have underlying objects
\[
U_1(E)=\be\oplus X_1^{\mathrm{sgn}}\oplus2X_1^{\mathrm{st}},\qquad
U_2(E)=\be\oplus X_1^{\mathrm{sgn}}\oplus2X_3^{\mathrm{tr}}.
\]
These lifts are $L_1=\Fun(S_3)$ and $L_2$, respectively.
Hence the two transition summands with noninvertible local transition
components have endpoints $(L_1,L_1)$ and $(L_2,L_2)$, respectively.
For the off-diagonal entries $(1,2)$ and $(2,1)$, the coefficient matrix is
$I_4$. Hence $\be_{\C_2}\bt\be_{\C_2}$ occurs with multiplicity one in each
localized transition algebra, so each consists of a single connected
\'etale summand. Its residual endpoints are trivial, so it is a graph algebra.
The matrix $I_4$ shows that the corresponding braided autoequivalence
fixes every simple object. Hence it represents the identity in
$\Aut^{\mathrm{br}}(\C_2)\cong C_2$.
Thus $\Theta_{12}=\Theta_{21}=1$, and neither off-diagonal middle pair
has a component in a higher localization.

For $i=1,2$, an algebra automorphism of $B_i$ induces a braided
autoequivalence of $\B_{B_i}^{\loc}$ by twisting the $B_i$-module structure,
and therefore preserves the underlying object in $\B$. The four underlying
objects in each list in \eqref{eqn: S3 C2 underlying objects} are pairwise
distinct, whereas the nontrivial braided autoequivalence of $\C_2$ exchanges
the second and third local simples. Hence $H_1=H_2=1$,
and the group ring transition matrix is
\[
\Theta^{(\C_2)}=\begin{pmatrix}1&1\\1&1\end{pmatrix},
\]
where each $1$ is the identity basis element of $\mathbb Z[C_2]$.

As an abelian group, the stratum is $\operatorname{Mat}_2(\mathbb Z[C_2])$; in this matrix decomposition, $S_{12}$ and $S_{21}$ are the heterotic endpoint blocks. The multiplication is $X*Y=XJ_2Y$, where $J_2$ is the all-ones matrix.
After complexification, $\mathbb C[C_2]\cong\mathbb C\oplus\mathbb C$, so
the stratum is the direct sum of two matrix sandwich algebras with sandwich
matrix $J_2$. Since $\rank J_2=1$, Lemma~\ref{lem: sandwich radical} gives
radical dimension $3$ for each summand. Hence the complexified stratum has
radical dimension $6$ and semisimple quotient $\mathbb C^2$.

\begin{remark}
Ostrik's classification \cite[Subsection~4.2 of the corrected version]{O} gives the
$28$ indecomposable module categories over
$\Rep(D(S_3))\simeq\Z(\Vec_{S_3})$. The $\B$-, $\C_3$-, $\C_2$-, and
maximal strata above have ranks $2,2,8$, and $16$, respectively, summing
to $28$. Clark and Davydov \cite{CD} computed the complete
multiplication table for these classes. Our purpose here is different: we
organize the calculation by localization strata and exhibit the resulting
sources of nonsemisimplicity. The calculation of the maximal stratum above also
appears in \cite[Example~3.13]{NikshychBurnside2026}.
\end{remark}


\section{Module Grothendieck rings for two exceptional WZW categories}
\label{sect: Example WZW}

We consider two exceptional WZW categories using the full stratum
multiplication theorem of
Subsection~\ref{subsect: The structure of stratum algebras}.
It organizes both singleton localization fibers and fibers with heterotic blocks, i.e., blocks with distinct endpoints. 
The classification and categorical input used below are due to Edie-Michell--Gannon
\cite{EG}, together with the
braided autoequivalence calculations of Edie-Michell \cite{EM}. For
$\C(\mathfrak{sl}_4,4)$, their modular invariant method does not determine
every relative tensor product because the relevant modular invariants are not
linearly independent \cite[Lemma~4.9]{EG}; our 
stratum formula extracts the remaining multiplication from their
classification and data on the local categories. For $\C(\mathfrak{sl}_6,6)$,
Edie-Michell--Gannon compute the complete $16\times16$ multiplication table
\cite[Example 1.4]{EG}. We instead derive and organize the stratum structure
independently.

For $\C(\mathfrak{sl}_n,k)$, write $V_{a_1\ldots a_{n-1}}$ for the simple
object of highest weight $a_1\omega_1+\cdots+a_{n-1}\omega_{n-1}$, where
$a_i\geq0$ and $\sum_i a_i\leq k$. We describe the \'etale algebras below by
their underlying objects in this notation; their algebra structures are those
of the cited constructions.

\subsection{Beyond modular invariants: \texorpdfstring{\(\mathcal C(\mathfrak{sl}_4,4)\)}{C(sl4,4)}}
\label{subsect: sl4 rank seven}

Let $\B=\C(\mathfrak{sl}_4,4)$. The classification in
\cite[Subsection~5.2]{EG} gives three connected \'etale algebras,
\[
\be,\qquad A_2,\qquad A_{\mathfrak{so}_{15}},
\]
whose local categories are pairwise braided inequivalent. Consequently, there
are no heterotic residual data between distinct algebras. The numbers of
residual data classes at these three algebras are respectively $2$, $4$, and
$1$, and hence $\operatorname{rk}K^\mmod(\B)=7$. More explicitly,
\[
\Aut^{\mathrm{br}}(\B_{A_2}^{\loc})\cong C_2^3
\]
by \cite[Theorem 1.4]{EM}, and the four middle classes are indexed by
\[
\Gamma_{A_2}\backslash
\Aut^{\mathrm{br}}(\B_{A_2}^{\loc})/
\Gamma_{A_2}.
\]
We use the stratum calculations together with the multiplicity formula to determine the full based ring multiplication in \(K^\mmod(\B)\), including products between different localization strata.
In the notation fixed above, the underlying objects of the two nontrivial
\'etale algebras, translated from the pointed algebra and adjoint
conformal embedding data in \cite{EG}, are
\[
A_2\cong V_{000}\oplus V_{040},
\qquad
A_{\mathfrak{so}_{15}}
\cong V_{000}\oplus V_{210}\oplus V_{012}\oplus V_{040}.
\]

The pointed algebra $A_2$ is the order-two simple current subalgebra of
$A_{\mathfrak{so}_{15}}$. Thus the \'etale algebra poset is the chain
$\be<A_2<A_{\mathfrak{so}_{15}}$. Put
\[
\D_2:=\B_{A_2}^{\loc},
\qquad
\mathcal I:=\B_{A_{\mathfrak{so}_{15}}}^{\loc}
\simeq\C(\mathfrak{so}_{15},1).
\]
The category $\mathcal I$ is an Ising category, so
$\Aut^{\mathrm{br}}(\mathcal I)=1$; hence the $\mathcal I$-stratum has a
single residual class.

The localization chain is $\B<\D_2<\mathcal I$. The paired posets are
\[
\begin{tikzpicture}[paired poset]
  \node (et-heading) at (0,2.9) {$\mathfrak{Et}(\B)$};
  \node (loc-heading) at (7,2.9) {$\mathfrak{Loc}(\B)$};
  \draw[->, line width=0.4pt] (et-heading.east) --
    node[above=2pt] {$A\mapsto\B_A^{\loc}$} (loc-heading.west);

  \node (unit) at (0,0) {$\be$};
  \node (A2) at (0,1) {$A_2$};
  \node (Aso) at (0,2) {$A_{\mathfrak{so}_{15}}$};
  \draw[poset cover] (unit) -- (A2);
  \draw[poset cover] (A2) -- (Aso);

  \node (locB) at (7,0) {$\B$};
  \node (locD) at (7,1) {$\D_2$};
  \node (locI) at (7,2) {$\mathcal I$};
  \draw[poset cover] (locB) -- (locD);
  \draw[poset cover] (locD) -- (locI);
\end{tikzpicture}
\]
Set $A=A_2=\be\oplus V_{040}$, $\C=\D_2$, and
$G=\Aut^{\mathrm{br}}(\D_2)\cong C_2^3$. Let $\Gamma_A$ be the image of
$\Aut_{\alg}(A_2)$ in $G$. By the classification in \cite[Subsection~5.2]{EG}, this endpoint block has four residual classes. Since \(G\) is abelian, \(\Gamma_A\backslash G/\Gamma_A=G/\Gamma_A\), so \(|G/\Gamma_A|=4\) and hence \(|\Gamma_A|=2\).
Moreover, $\Aut_{\alg}(A_2)\cong C_2$: an algebra automorphism fixes the unit and acts
on the order-two simple current summand $V_{040}$ by a scalar $\lambda$, while
compatibility with its square gives $\lambda^2=1$. Thus induction is injective
in this example; both choices of $\lambda$ occur.

$A_2$ is the regular algebra of the order-two Tannakian simple current
subcategory. Hence $\B_{A_2}$ is a faithfully $C_2$-graded extension of
$\D_2$. By Propositions~\ref{prop:diagonal component correspondence}
and~\ref{prop:diagonal local transition component}, its two
homogeneous components are precisely the diagonal transition components. By
Remark~\ref{rem: Tannakian regular endpoint}, both are invertible, their
Picard labels are the two elements of $\Gamma_A$, and no term in a higher stratum can
occur in a product within this block. The transition count
\eqref{eqn: transition component count} provides the numerical check
\[
\#\{\text{diagonal transition summands}\}
=\dim_k\End_\B(A_2)=2.
\]
Thus $\Theta_A=\sum_{h\in \Gamma_A}h$.

Put $V:=G/\Gamma_A\cong C_2^2$ and label the four classes in the middle stratum by $x_v$,
$v\in V$. The diagonal case of Theorem~\ref{thm: stratum multiplication} gives
\begin{equation}
\label{eqn: sl4 middle multiplication}
x_vx_w=2x_{v+w},
\qquad
S_{\D_2}(\B)\cong2\mathbb Z[C_2^2].
\end{equation}
Here the right-hand side denotes the based ring with basis $x_v$ and
multiplication $x_vx_w=2x_{v+w}$. Remark~\ref{rem: Tannakian regular endpoint} shows
that the same formula holds in the full module Grothendieck ring, without any
correction from the Ising stratum.

Let \(1,q\) be the two basis elements in the \(\B\)-stratum, corresponding respectively to the identity and charge conjugation \(c\in\Aut^{\mathrm{br}}(\B)\), and let \(z\) be the unique basis element in the Ising stratum. For \(x_vz\) and \(zx_v\), Remark~\ref{rem: product multiplicity} gives total multiplicity \(2\), since
\[
\dim_k\Hom_{\B}(A_2,A_{\mathfrak{so}_{15}})
=2.
\]
Endpoint monotonicity, together with the inequivalence of \(\D_2\) and \(\mathcal I\), forces both output endpoints to be \(A_{\mathfrak{so}_{15}}\). Since \(z\) is the unique residual class with these endpoints, \(x_vz=zx_v=2z\). Maximality of \(A_{\mathfrak{so}_{15}}\) and \(\dim_k\End_{\B}(A_{\mathfrak{so}_{15}})=4\) give \(z^2=4z\). Finally, \(qz=zq=z\), since the relevant total multiplicity is one and \(z\) is the unique possible output.

Since $c(A_2)\cong A_2$ as an algebra, a choice of algebra isomorphism
$c(A_2)\xrightarrow{\sim}A_2$ induces a braided autoequivalence
$\bar c\in G=\Aut^{\mathrm{br}}(\D_2)$. Changing the algebra isomorphism
changes $\bar c$ by an element of $\Gamma_A$, so the coset
\[
a:=\bar c\Gamma_A\in V=G/\Gamma_A
\]
is canonical. Composition of residual graph equivalences gives
\[
qx_v=x_{v+a},
\qquad
x_vq=x_{v+a};
\]
the left and right translations agree because $G\cong C_2^3$ is abelian.

Every element of \(\Gamma_A\) preserves the underlying object of a local \(A_2\)-module, since it is induced by twisting the module structure by an algebra automorphism of \(A_2\). On the other hand, the free local $A_2$-module \(A_2 \otimes V_{200}\) has underlying object $V_{200}\oplus V_{022}$,
while its image under \(\bar c\) has underlying object $V_{002}\oplus V_{220}$.
Hence \(\bar c\notin \Gamma_A\), so \(a\ne0\) in \(V\cong C_2^2\).
Choose coordinates on $V$ with
$a=(1,0)$. The complete based ring multiplication of $K^\mmod(\B)$, with all indices in $C_2$,
is
\[
q^2=1,
\qquad
x_{(i,j)}x_{(k,\ell)}=2x_{(i+k,j+\ell)},
\qquad
qx_{(i,j)}=x_{(i+1,j)}=x_{(i,j)}q,
\]
\[
qz=zq=z,
\qquad
x_{(i,j)}z=zx_{(i,j)}=2z,
\qquad
z^2=4z.
\]
The three complexified strata are respectively $\mathbb C^2$, $\mathbb C^4$,
and $\mathbb C$. Proposition~\ref{prop: semisimplicity poset grading} therefore shows
that the full rank-seven complexified algebra is semisimple, and its
commutativity gives a noncanonical algebra isomorphism
\[
K^\mmod\bigl(\mathcal C(\mathfrak{sl}_4,4)\bigr)
\ot_{\mathbb Z}\mathbb C
\cong\mathbb C^7.
\]

Thus the stratum calculation determines the full based ring multiplication
from the classification and data on the local categories of Edie-Michell--Gannon and
Edie-Michell, in a case where modular invariants alone do not distinguish all
module categories.

\subsection{A noncommutative semisimple example:
\texorpdfstring{$\mathcal C(\mathfrak{sl}_6,6)$}{C(sl6,6)}}
\label{subsect: sl6 semisimple}

Let $\B=\C(\mathfrak{sl}_6,6)$ and abbreviate
\[
P=A_3,
\qquad O=A_{\mathfrak{so}_{35}},
\qquad S=A_{\mathfrak{sp}_{20}},
\qquad E=A_{\mathfrak{sp}_{20}}^{\mathrm{ext}},
\qquad T=A_{\mathfrak{sp}_{20}}^{\mathrm{tr}}.
\]
The six connected \'etale algebras are $\be,P,O,S,E,T$. 
The classification and the explicit conformal embedding and lifting constructions in \cite[Subsection~5.4]{EG} give the complete list of cover relations
\[
\be\lessdot P,\qquad P\lessdot O,\qquad P\lessdot S,\qquad P\lessdot T,\qquad S\lessdot E,\qquad T\lessdot E.
\]
There are no other inclusions except those forced transitively.
Here $S<E$ is obtained by lifting the $E_6$-algebra from
$\B_S^{\loc}\simeq\C(\mathfrak{sp}_{20},1)$, and $T<E$ follows from the
level-rank/transposition construction. We use this poset as classification input.

Put
\[
\D_3:=\B_P^{\loc},
\qquad
\D_S:=\B_S^{\loc}\simeq\C(\mathfrak{sp}_{20},1),
\qquad
\D_T:=\B_T^{\loc},
\qquad
\mathcal I:=\B_O^{\loc}\simeq\B_E^{\loc}.
\]
Here $\mathcal I$ is an Ising category, while $\D_S$ and $\D_T$ have the
$SU(2)_{10}$ fusion rules but are not braided equivalent. There are exactly
five localization classes, and the paired posets are
\[
\begin{tikzpicture}[paired poset]
  \node (et-heading) at (0,3.9) {$\mathfrak{Et}(\B)$};
  \node (loc-heading) at (7,3.9) {$\mathfrak{Loc}(\B)$};
  \draw[->, line width=0.4pt] (et-heading.east) --
    node[above=2pt] {$A\mapsto\B_A^{\loc}$} (loc-heading.west);

  \node (unit) at (0,0) {$\be$};
  \node (P) at (0,1) {$P$};
  \node (S) at (0,2) {$S$};
  \node (T) at (2,2) {$T$};
  \node (O) at (-2,3) {$O$};
  \node (E) at (1,3) {$E$};
  \draw[poset cover] (unit) -- (P);
  \draw[poset cover] (P) -- (O);
  \draw[poset cover] (P) -- (S);
  \draw[poset cover] (P) -- (T);
  \draw[poset cover] (S) -- (E);
  \draw[poset cover] (T) -- (E);

  \node (locB) at (7,0) {$\B$};
  \node (locD3) at (7,1) {$\D_3$};
  \node (locDS) at (6,2) {$\D_S$};
  \node (locDT) at (8,2) {$\D_T$};
  \node (locI) at (7,3) {$\mathcal I$};
  \draw[poset cover] (locB) -- (locD3);
  \draw[poset cover] (locD3) -- (locDS);
  \draw[poset cover] (locD3) -- (locDT);
  \draw[poset cover] (locDS) -- (locI);
  \draw[poset cover] (locDT) -- (locI);
\end{tikzpicture}
\]
The fiber over
$\mathcal I$ is $\{O,E\}$. In particular, the localization map identifies two
distinct maximal \'etale algebras, and the cover $P\lessdot O$ maps to the
non-cover $\D_3<\mathcal I$.

The five stratum ranks are $4,4,2,2,4$, summing to $16$. The rank-four Ising
stratum consists of the residual data with ordered endpoints $(O,O)$, $(E,E)$,
$(E,O)$, and $(O,E)$; denote the corresponding classes by $x_{OO}$,
$x_{EE}$, $x_{EO}$, and $x_{OE}$. The braided autoequivalence group of an
Ising category is trivial, so each ordered pair contributes one basis element;
the two off-diagonal pairs give the heterotic module categories.

We now derive the five stratum rings from
Theorem~\ref{thm: stratum multiplication}. For $A=\be$, the
category $\B_A=\B$ has a single invertible $\B$-module component, with
identity Picard label. By
Proposition~\ref{prop:diagonal component correspondence}, this is the unique
diagonal transition summand. Corollary~\ref{cor:diagonal transition element}
therefore gives $\Theta_{\be}=1$. Since
$\Aut^{\mathrm{br}}(\B)\cong C_2^2$, the
diagonal formula gives
\[
S_\B(\B)\cong\mathbb Z[C_2^2],
\qquad
S_\B(\B)_\mathbb C\cong\mathbb C^4.
\]

The algebra $P=A_3$ is the regular algebra of its order-three Tannakian
simple current subcategory. Thus $\B_P$ is a faithfully $C_3$-graded extension of
$\D_3$. Proposition~\ref{prop:diagonal component correspondence} identifies its
three homogeneous components with the diagonal transition summands, and
Remark~\ref{rem: Tannakian regular endpoint} shows that all three are
invertible and that no term in a higher stratum occurs in their diagonal products. 
Their Picard labels form the normal subgroup
$\Gamma_P\cong C_3$ of
$\Aut^{\mathrm{br}}(\D_3)\cong D_3\times C_2$, and
$\Aut^{\mathrm{br}}(\D_3)/\Gamma_P\cong C_2^2$.
Corollary~\ref{cor:diagonal transition element} gives
$\Theta_P=\sum_{h\in\Gamma_P}h$. Writing the four
residual classes as $x_v$, $v\in C_2^2$, we obtain
\[
x_vx_w=3x_{v+w},
\qquad
S_{\D_3}(\B)\cong3\mathbb Z[C_2^2],
\qquad
S_{\D_3}(\B)_\mathbb C\cong\mathbb C^4.
\]

We next use the same short ADE count for the \(\mathcal D_S\)- and \(\mathcal D_T\)-strata. Both categories have the \(SU(2)_{10}\) fusion rules, whose indecomposable module categories have ADE types \(A_{11}\), \(D_7\), and \(E_6\) \cite{Ost03}. Let \(A\in\{S,T\}\), put \(\mathcal D_A=\mathcal B_A^{\mathrm{loc}}\), and label its simple objects \(X_0,\ldots,X_{10}\), with \(X_0=\mathbf1\). Write
\[
\left.\mathcal B_A\right|_{\mathcal D_A}
\simeq a\mathcal M_{A_{11}}\oplus d\mathcal M_{D_7}\oplus e\mathcal M_{E_6},
\]
where in the \(T\)-case the \(E_6\)-type module category is the transposed one, denoted below by \(\mathcal M_{E_6}^{\mathrm{tr}}\).

By \eqref{eqn: mixed contraction coefficients} and
Propositions~\ref{prop:diagonal component correspondence}
and~\ref{prop:diagonal local transition component}, for \(i=0,1,2\) the number
\(\dim_k\End_{\mathcal B}(U_A(X_i))\) is the sum of the corresponding diagonal full-center coefficients of these components. At the labels \(0,1,2\), the \(A_{11}\)-, \(D_7\)-, and \(E_6\)-matrices have diagonal coefficients \((1,1,1)\), \((1,0,1)\), and \((1,0,0)\), respectively \cite{CIZ,Ost03}; transposition does not change these coefficients. For \(S\), the restriction rules in \cite[Subsection~2.4.4]{EG}, and for \(T\), Definition 5.4 and the displayed branching rules in \cite[Subsection~5.4]{EG}, show that \(U_A(X_0),U_A(X_1),U_A(X_2)\) are multiplicity-free with \(7,6,6\) simple summands, respectively. Hence
\[
a+d+e=7,\qquad a=6,\qquad a+d=6,
\]
so \((a,d,e)=(6,0,1)\). Therefore
\[
\left.\mathcal B_S\right|_{\mathcal D_S}
\simeq\mathcal D_S^{\oplus6}\oplus\mathcal M_{E_6},
\qquad
\left.\mathcal B_T\right|_{\mathcal D_T}
\simeq\mathcal D_T^{\oplus6}\oplus\mathcal M_{E_6}^{\mathrm{tr}}.
\]
In each case the six regular components are invertible and have identity Picard label, while the \(E_6\)-type component is noninvertible. The classification data give \(\operatorname{Aut}^{\mathrm{br}}(\mathcal D_S)\cong\operatorname{Aut}^{\mathrm{br}}(\mathcal D_T)\cong C_2\) and \(\Gamma_S=\Gamma_T=1\). Thus
\[
\Theta_S=\Theta_T=6\cdot1,\qquad
S_{\mathcal D_S}(\mathcal B)\cong6\mathbb Z[C_2],
\qquad
S_{\mathcal D_T}(\mathcal B)\cong6\mathbb Z[C_2].
\]

The multiplicity formula gives total multiplicity seven in a product within
either diagonal block, while the projected stratum term has total
multiplicity six. The unique remaining component is the noninvertible
$E_6$-component, and the only possible higher endpoint is $E$. Thus, if
$s_\epsilon$ and $t_\epsilon$, for $\epsilon\in C_2$, denote the two classes
in the $S$- and $T$-blocks, respectively, then the relations in the full ring are
\begin{equation}
\label{eqn: sl6 ADE block products}
s_\epsilon s_\delta=6s_{\epsilon+\delta}+x_{EE},
\qquad
t_\epsilon t_\delta=6t_{\epsilon+\delta}+x_{EE}.
\end{equation}

It remains to analyze the maximal Ising stratum. Since
$\Aut^{\mathrm{br}}(\mathcal I)=1$, its four residual classes are indexed by
the ordered endpoint pairs in $\{O,E\}^2$. Translating the relevant
forgetful functor formulas into the Dynkin label notation fixed above gives
the following underlying objects. The object $O=A_{\mathfrak{so}_{35}}$ comes
from the adjoint conformal embedding restriction data in
\cite[Subsection 2.4.3]{EG}; the object
$E=A_{\mathfrak{sp}_{20}}^{\mathrm{ext}}$ is constructed in
\cite[Definition 5.3]{EG} from the forgetful functor whose decompositions are
listed in \cite[Subsection 2.4.4]{EG}:
\begin{align*}
O\cong{}&V_{00000}\oplus V_{20010}\oplus V_{01002}\oplus V_{41000}
\oplus V_{21101}\oplus V_{02020}\oplus V_{10112}\oplus V_{00014}\notag\\
&\oplus V_{23010}\oplus V_{03200}\oplus V_{11211}\oplus V_{00230}
\oplus V_{01032}\oplus V_{06000}\oplus V_{01410}\oplus V_{00060},
\end{align*}
and
\begin{align*}
E\cong{}&V_{00000}\oplus V_{06000}\oplus V_{00060}\oplus V_{00200}
\oplus V_{04002}\oplus V_{20040}\oplus2V_{02020}\notag\\
&\oplus V_{00303}\oplus V_{30003}\oplus V_{30300}\oplus V_{21101}
\oplus V_{11211}\oplus V_{10112}.
\end{align*}
Thus $O$ has sixteen distinct simple summands; $E$ has twelve summands of
multiplicity one and one of multiplicity two. Their common summands are
$V_{00000}$, $V_{06000}$, $V_{00060}$, $V_{21101}$, $V_{11211}$,
$V_{10112}$, and $V_{02020}$, the last with multiplicity two in $E$.
Consequently
\[
\dim_k\End_\B(O)=16,
\qquad
\dim_k\End_\B(E)=16,
\qquad
\dim_k\Hom_\B(O,E)=8.
\]
Since $G=1$ and
the Ising localization is maximal, Corollary~\ref{cor: scalar sandwich strata}
identifies the stratum with the matrix sandwich ring
\begin{equation}
\label{eqn: sl6 Ising sandwich}
x_{ij}x_{k\ell}=h_{jk}x_{i\ell},
\qquad
H_{\mathcal I}=(h_{jk})_{j,k\in\{O,E\}}
=\begin{pmatrix}16&8\\8&16\end{pmatrix}.
\end{equation}
In particular,
\[
x_{OO}x_{EE}=8x_{OE},
\qquad
x_{EE}x_{OO}=8x_{EO},
\]
so the stratum, and hence the full ring, is noncommutative. Since
$H_{\mathcal I}$ is nonsingular,
\[
S_{\mathcal I}(\B)_\mathbb C\cong\operatorname{Mat}_2(\mathbb C).
\]

The other four complexified strata are $\mathbb C^4$, $\mathbb C^4$,
$\mathbb C^2$, and $\mathbb C^2$. Proposition~\ref{prop: semisimplicity poset grading}
therefore gives the noncanonical algebra isomorphism
\[
K^\mmod(\C(\mathfrak{sl}_6,6))\ot_{\mathbb Z}\mathbb C
\cong\mathbb C^{12}\oplus\operatorname{Mat}_2(\mathbb C).
\]
Edie-Michell--Gannon computed the full $16\times16$ table by the
modular invariant method. The calculation above rederives and organizes its
stratum structure from their classification and forgetful functor data,
together with the general formulas of
Subsection~\ref{subsect: The structure of stratum algebras}.

This example contrasts sharply with Subsection~\ref{subsect: module ring DS3}.
For $\Rep(D(S_3))$, the maximal sandwich matrix is singular and produces a
radical. For $\C(\mathfrak{sl}_6,6)$, the Ising sandwich matrix is nonsingular
and every other stratum is semisimple. Thus noncommutativity of a module
Grothendieck algebra is independent of the presence of a radical.


\section{A radical in the module Grothendieck ring of the Haagerup center}
\label{sect: Haagerup center}

\subsection{The three Lagrangian algebras}
\label{subsect: Haagerup full centers}

Let $\H=\H_2$ be the six-object Haagerup fusion category in the convention of
Grossman--Snyder \cite{GS}, and put $\B:=\Z(\H)$.
Write $\be,\alpha,\alpha^2,\rho,\alpha\rho,\alpha^2\rho$ for the simple objects of $\H_2$.
In the notation of
Hong--Rowell--Wang \cite[Subsection 2.2]{HRW}, the simple objects of $\B$ are
\[
\be,\ \pi_1,\ \pi_2,\ \sigma_1,\ \sigma_2,\ \sigma_3,
\ \mu_1,\ldots,\mu_6.
\]
We fix the $\H_2$-adapted convention through the forgetful functor
$F_2:\B\to\H_2$. Put $M:=\rho+\alpha\rho+\alpha^2\rho$. The forgetful images
needed to distinguish $\pi_2$ from $\sigma_1$ are
\[
F_2(\pi_1)=\be+M,
\qquad
F_2(\pi_2)=2\be+M,
\qquad
F_2(\sigma_1)=\alpha+\alpha^2+M,
\]
as computed in \cite[Section 8]{IzumiLR2}.

The three Lagrangian algebras in $\B$ are
\begin{equation}
\label{eqn: F2 images of L}
L_{\H_2}=\be+\pi_1+2\pi_2,
\qquad
L_{\H_1}=\be+\pi_1+\pi_2+\sigma_1,
\qquad
L_{\H_3}=\be+\pi_1+2\sigma_1.
\end{equation}
Grossman--Snyder classify three indecomposable $\H_2$-module categories, with dual fusion categories $\H_2,\H_1,\H_3$ \cite[Section~3]{GS}. The computations of the canonical sectors in \cite[Subsection~2.2]{EGHaag} give the first two displayed decompositions in \eqref{eqn: F2 images of L}, corresponding to the $\H_2$- and $\H_1$-dual module categories, while \cite[Theorem~4.3]{EP06} gives the third. By exhaustiveness, the third therefore corresponds to the $\H_3$-dual module category.

\subsection{The \texorpdfstring{\'etale}{etale} algebra poset}
\label{subsect: Haagerup etale lattice}

Every connected \'etale algebra in
$\B=\Z(\H)$ is contained in a Lagrangian algebra, and the interval below a
fixed $L_{\H_i}$ is anti-isomorphic to the lattice of fusion subcategories of
$\H_i$ \cite[Theorem 4.10]{DMNO}. The Haagerup fusion rules give
\[
\operatorname{Sub}(\H_1)=\{\Vect,\H_1\},
\qquad
\operatorname{Sub}(\H_i)=\{\Vect,\Vec_{\mathbb Z_3},\H_i\}
\quad(i=2,3).
\]
For $\H_1$, every nontrivial simple object generates the full category. For
$\H_2$ and $\H_3$, the unique proper nontrivial fusion subcategory is the
pointed subcategory on the three invertible objects; the two categories have
the same fusion ring by \cite[Theorem 1.4]{GS}.

Put
\[
A:=\be+\pi_1.
\]
Evans--Pinto identify \(A\) as a connected \'{e}tale algebra with
\[
\mathcal C:=\mathcal B_A^{\mathrm{loc}}\simeq \mathcal Z(\Vec_{\mathbb Z_3}),
\]
and realize \(L_{\mathcal H_2}\) and \(L_{\mathcal H_3}\) as Lagrangian algebras containing \(A\) \cite[Section 4]{EP06}.
Although the underlying object of $L_{\H_1}$ contains $\be+\pi_1$ as a
direct summand, $A\not\subset L_{\H_1}$ as \'etale algebras. Indeed, such an
algebra inclusion would give a proper nontrivial fusion subcategory of
$\H_1$, and none exists.

It follows that the complete list of connected \'etale algebras in \(\mathcal B\) is
\[
\mathbf 1,\quad A,\quad L_{\mathcal H_1},\quad L_{\mathcal H_2},\quad L_{\mathcal H_3}.
\]
The paired posets are
\[
\begin{tikzpicture}[paired poset]
  \node (et-heading) at (0,2.9) {$\mathfrak{Et}(\B)$};
  \node (loc-heading) at (7,2.9) {$\mathfrak{Loc}(\B)$};
  \draw[->, line width=0.4pt] (et-heading.east) --
    node[above=2pt] {$A\mapsto\B_A^{\loc}$} (loc-heading.west);

  \node (unit) at (0,0) {$\be$};
  \node (A) at (1,1) {$A$};
  \node (LH1) at (-2,2) {$L_{\H_1}$};
  \node (LH2) at (0,2) {$L_{\H_2}$};
  \node (LH3) at (2,2) {$L_{\H_3}$};
  \draw[poset cover] (unit) -- (A);
  \draw[poset cover] (unit) -- (LH1);
  \draw[poset cover] (A) -- (LH2);
  \draw[poset cover] (A) -- (LH3);

  \node (locB) at (7,0) {$\B$};
  \node (locC) at (7,1) {$\C$};
  \node (locVec) at (7,2) {$\Vect$};
  \draw[poset cover] (locB) -- (locC);
  \draw[poset cover] (locC) -- (locVec);
\end{tikzpicture}
\]
Thus the localization poset is the chain $\B<\C<\Vect$, and its three fibers
have cardinalities $1,1,3$.

\subsection{The stratum rings}
\label{subsect: Haagerup strata}

Grossman--Snyder prove $\BrPic(\H)=1$ \cite[Theorem 1.5]{GS}; hence
$\Aut^{\mathrm{br}}(\B)=1$ by the identification
$\BrPic(\H)\cong\Aut^{\mathrm{br}}(\Z(\H))$
\cite[Theorem 1.1]{ENO2}. Therefore the bottom stratum is
\[
S_{\B}(\B)\cong\mathbb Z.
\]

For the middle localization put
\[
\C:=\B_A^{\loc}\simeq\Z(\Vec_{\mathbb Z_3}),
\qquad
G:=\Aut^{\mathrm{br}}(\C).
\]
The description by the associated metric group gives $G\cong O^+(2,3)\cong C_2^2$.

Let $\D:=\Vec_{\mathbb Z_3}\subset\H_2$. Under the equivalence
$\B_A\simeq\Z_\D(\H_2)$, where $\Z_\D(\H_2)$ denotes the relative center of $\H_2$ with respect to $\D$, the local subcategory $\C$ identifies with
$\Z(\D)$. The forgetful tensor functor to $\H_2$ sends every nonlocal
simple object to an object containing a noninvertible simple, so its
Frobenius--Perron dimension is at least $(3+\sqrt{13})/2>3$.
Since every homogeneous component of the commutator grading has
Frobenius--Perron dimension $\FPdim(\C)=9$, no nonneutral component
can occur. Theorem~\ref{thm: automorphisms etale algebra} therefore gives
$\Aut_{\alg,\B}(A)=1$, and hence
$\Gamma_A=1$.

As a \(\D\)-bimodule category,
\[
\H_2=\D\oplus\mathcal M,
\qquad
\mathcal M:=\langle\rho,\alpha\rho,\alpha^2\rho\rangle.
\]
Both components are invertible: \(\D\) is the regular bimodule, while \(\mathcal M\) is the \(\mathbb Z_3\)-bitorsor whose right action is twisted by inversion, as follows from
\(\alpha\rho=\rho\alpha^{-1}\).
By Propositions~\ref{prop:diagonal component correspondence}
and~\ref{prop:diagonal local transition component}, they give the two diagonal transition summands. The regular component has Picard label \(1\), while the inversion-twisted component has Picard label a nontrivial involution \(g\in G\). Hence
\[
\Theta_A=1+g\in\mathbb Z[G].
\]
Write $G\cong C_2^2$ additively. The diagonal specialization of
Theorem~\ref{thm: stratum multiplication} gives the complete multiplication in the middle stratum
\begin{equation}
\label{eqn: Haagerup middle multiplication}
x_ux_v=x_{u+v}+x_{u+v+g}
\qquad (u,v\in G).
\end{equation}
There are no heterotic endpoint blocks because the localization fiber over
$\C$ contains only $A$.

Choose \(h\in G\setminus\langle g\rangle\). It follows from \eqref{eqn: Haagerup middle multiplication} that
\[
J_{\C}:=\operatorname{Span}_{\mathbb C}\{x_0-x_g,\;x_h-x_{h+g}\}
=\operatorname{Jac}\bigl(S_{\C}(\B)_{\mathbb C}\bigr),
\]
with
\[
\dim_{\mathbb C}J_{\C}=2,\qquad J_{\C}^2=0,
\qquad
S_{\C}(\B)_{\mathbb C}/J_{\C}\cong\mathbb C[C_2]\cong\mathbb C^2.
\]

The maximal stratum is described by
Corollary~\ref{cor: maximal Gram sandwich}.  It has the nine basis elements
\[
x_{ij}=x_{L_{\H_i},L_{\H_j}},
\qquad 1\leq i,j\leq3.
\]
The multiplicity vectors of the underlying objects, in the ordered basis
\((\be,\pi_1,\pi_2,\sigma_1)\), are
\[
L_{\H_1}=(1,1,1,1),\qquad
L_{\H_2}=(1,1,2,0),\qquad
L_{\H_3}=(1,1,0,2).
\]
Thus the sandwich matrix, indexed by
\((L_{\H_1},L_{\H_2},L_{\H_3})\), is
\begin{equation}
\label{eqn: Haagerup sandwich matrix}
H=
\begin{pmatrix}
4&4&4\\
4&6&2\\
4&2&6
\end{pmatrix},
\end{equation}
and
\begin{equation}
\label{eqn: Haagerup sandwich multiplication}
x_{ij}x_{k\ell}=H_{jk}x_{i\ell}.
\end{equation}

The underlying objects satisfy
\begin{equation}
\label{eqn: Haagerup full center relation}
-2[L_{\H_1}]+[L_{\H_2}]+[L_{\H_3}]=0
\qquad\text{in }K_0(\B).
\end{equation}
Hence
\[
\ker H=\mathbb C(-2,1,1),
\qquad
\operatorname{rk}H=2.
\]

Applying Lemma~\ref{lem: sandwich radical} gives
\begin{equation}
\label{eqn: Haagerup maximal radical}
J_{\Vect}:=\operatorname{Jac}\bigl(S_{\Vect}(\B)_{\mathbb C}\bigr)
=\{X\in\operatorname{Mat}_3(\mathbb C)\mid HXH=0\},
\qquad
\dim_{\mathbb C}J_{\Vect}=5,
\end{equation}
with
\[
S_{\Vect}(\B)_{\mathbb C}/J_{\Vect}
\cong\operatorname{Mat}_2(\mathbb C),
\qquad
J_{\Vect}^3=0.
\]

The bottom, middle, and maximal strata have ranks \(1,4,9\), respectively, so
$\operatorname{rk}K^\mmod(\B)=14$.
Since the middle and maximal strata have nonzero radicals, \(K^\mmod(\B)\otimes_{\mathbb Z}\mathbb C\) is not semisimple by Proposition~\ref{prop: semisimplicity poset grading}. The rank \(14\) is larger than that of the fusion algebra of realized modular invariant characters computed by Evans--Pinto \cite[Section~4]{EP06}; the difference reflects the loss of information under $[\M]\mapsto[L_\M]\in K_0(\Z(\B))$, which forgets the algebra structure of the full center.

\bibliographystyle{amsalpha}
\bibliography{Grothendieckrefs}

\end{document}